\documentclass[11pt,reqno]{amsart}
\usepackage[margin=20mm]{geometry}
\usepackage{amssymb}
\usepackage{bbm}
\usepackage{dsfont}
\usepackage{amsfonts}
\usepackage{pifont}
\usepackage{bbding}
\usepackage{latexsym}
\usepackage{mathrsfs}
\usepackage{amsmath}
\allowdisplaybreaks[4]
\usepackage{amsthm}
\usepackage{amssymb}
\usepackage{graphicx}
\usepackage{hyperref}
\usepackage[capitalise]{cleveref}
\usepackage{fancyhdr}
\usepackage{color}
\newtheorem{theorem}{\indent Theorem}[section]

\newtheorem{proposition}{\indent Proposition}[section]

\newtheorem{assumption}{\indent Assumption}[section]
\newtheorem{lemma}{\indent Lemma}[section]
\newtheorem{remark}{\indent Remark}[section]
\numberwithin{equation}{section}

\newcommand{\R}{\mathbb{R}}
\newcommand{\s}{\mathbb{S}}

\newcommand{\va}{\varepsilon}
\newcommand{\al}{\alpha}

\newcommand{\om}{\omega}

\newcommand{\ga}{\gamma}
\newcommand{\de}{\delta}
\newcommand{\mb}{\mathbb}

\newcommand{\proofof}[1]
\allowdisplaybreaks

\begin{document}
	
\allowdisplaybreaks
\title[convergence rates for slow-fast system]{Strong and weak convergence rates for slow-fast system driven by   multiplicative L\'{e}vy noises}

\author{Qiu-Chen Yang, Kun Yin}

\address{\textbf{Qiu-Chen Yang:} School of Mathematical Sciences, Shanghai Jiao Tong University, 200240 Shanghai, PR China.}\email{yangqiuchen@sjtu.edu.cn}

\address{\textbf{Kun Yin:} School of Mathematical Sciences, Shanghai Jiao Tong University, 200240 Shanghai, PR China.}\email{epsilonyk@sjtu.edu.cn}

%\thanks{\textbf{Kun Yin:} Department of Mathematics, Shanghai Jiao Tong University, 200240 Shanghai, PR China. e-mail address: epsilonyk@sjtu.edu.cn}



\allowdisplaybreaks

\begin{abstract}
In this paper, we study strong and weak convergence rates for a class of time-inhomogeneous slow-fast systems driven by multiplicative $\alpha$-stable noises. The state-dependent jump coefficients give rise to substantial difficulties in establishing exponential ergodicity of the frozen equation and regularity estimates for the associated nonlocal Poisson equation. By combining the coupling method with $L^p$-Wasserstein distance, we establish the exponential ergodicity and derive a crucial gradient estimate for the corresponding nonlocal Poisson equation. This enables us to demonstrate strong and weak averaging principles with explicit convergence rates for multiplicative stable noise systems. Moreover, the optimal strong convergence rate is obtained under sufficient H\"older regularity conditions on the coefficients of the slow process.
\end{abstract}

\maketitle

\section{Introduction}

\subsection{Background}
It is known that the following model represents a classical framework in
statistical fluid mechanics, since it describes the dynamics of (anomalous) diffusive particles convected by a random velocity field $b(x)$,
$$
dX_t=b(X_t) dt+ \text{L\'evy noise}.
$$
However, in many realistic scenarios, the model is often influenced by rapidly oscillating terms, 
then we refer to it as a multiscale system (or slow‑fast system). Multiscale systems are widely applied in chemistry, biology, materials science, and physics. The study of multiscale systems involves stochastic analysis and partial differential equations.

Given the complexity of these coupled dynamics, a direct analysis is 
difficult. This motivates the study of the averaging principle, which serves as a powerful tool to simplify the multiscale system.
The averaging principle can be viewed as a variant of the  functional law of large numbers \cite{RX2}, i.e., for the Markov process $Y_{t}\in\R^d$,
\begin{align*}
\lim_{\varepsilon\rightarrow0}\varepsilon\int_{0}^{\frac{t}{\varepsilon}}f(x,Y_{s})ds=\lim_{\varepsilon\rightarrow0}\int_{0}^{t}f(x,Y_{\frac{s}{\varepsilon}})ds=t\int_{\R^d} f(x,y)\mu(dy)=t\bar{f}(x),
\end{align*}
where $\mu$ is the invariant measure for the transition
semigroup of $Y_{t}$.
The slow-fast stochastic differential equations driven by Brownian motions are demonstrated in fruitful  references such as \cite{RZK,RX,AYV}, especially the system with time-dependent coefficients has been studied in \cite{LRSX},
%\begin{equation}\label{1}\left\{\begin{aligned}&dX_{t}^{\varepsilon}=b(X_{t}^{\varepsilon},Y_{t}^{\varepsilon})dt+\delta_{1}(X_{t}^{\varepsilon},Y_{t}^{\varepsilon})dB^{1}_{t},\ X^{\varepsilon}_{0}=x\in \mathbb{R}^{d_{1} },\\	&dY_{t}^{\varepsilon}=\frac{1}{\varepsilon}f(X_{t}^{\varepsilon},Y_{t}^{\varepsilon})dt+\frac{1}{\varepsilon^{\frac{1}{2}}}\delta_{2}(X_{t}^{\varepsilon},Y_{t}^{\varepsilon})dB^{2}_{t},\ Y^{\varepsilon}_{0}=y\in \mathbb{R}^{d_{2}},\end{aligned}\right.\end{equation}
%i.e., $\exists\beta>0$, s.t.,$$( f(x,y_{1})-f(x,y_{2}),y_{1}-y_{2})  \leq-\beta|y_{1}-y_{2}|^{2},$$ which is important in proving the existence and uniqueness of the invariant measure  $ \mu^{x}(dy)$ for the frozen equation which is related to fast process $Y_{t}^{\varepsilon}$,$$	dY_{t}^{x,y}=f(x,Y_{t})dt+\delta_{2}(x,Y_{t})dB^{2}_{t},\ Y_{0}=y\in \mathbb{R}^{d_{2}},$$$x$ is fixed here, 
\begin{equation}\label{1}
	\left\{
	\begin{aligned}
		& dX_{t}^{\varepsilon}=b(t,X_{t}^{\varepsilon},Y_{t}^{\varepsilon})dt+\delta(t,X_{t}^{\varepsilon})dB^{1}_{t},\ X^{\varepsilon}_{0}=x\in \mathbb{R}^{d_{1} },\\
		&dY_{t}^{\varepsilon}=\frac{1}{\varepsilon}f(t,X_{t}^{\varepsilon},Y_{t}^{\varepsilon})dt+\frac{1}{\varepsilon^{\frac{1}{2}}}g(t,X_{t}^{\varepsilon},Y_{t}^{\varepsilon})dB^{2}_{t},\ Y^{\varepsilon}_{0}=y\in \mathbb{R}^{d_{2}},
	\end{aligned}
	\right.
\end{equation}
by dissipative condition of $f(t,x,y)$, a concept from dynamical system theory, $X_{t}^{\varepsilon}$ converges strongly as $\varepsilon \rightarrow0$ to averaged equation
$$d\bar{X}_{t}=\bar{b}(t,\bar{X}_{t})dt+\delta(t,\bar{X}_{t})dB^{1}_{t},\ X_{0}=x\in \mathbb{R}^{d_{1}},$$
where $\bar{b}(t,x)=\int_{\mathbb{R}^{d_{2}}}b(t,x,y)\mu^{t,x}(dy),$
%the dissipative condition is represented as  $\exists \lambda>0$, 
%$$2\langle f(t,x,y_{1})-f(t,x,y_{2}),y_{1}-y_{2}\rangle +\Arrowvert g(t,x,y_{1})-g(t,x,y_{2})\Arrowvert^{2}\leq-\lambda|y_{1}-y_{2}|^{2},$$
the dissipative condition of $f$ and $g$ where $t,x$  are fixed parameters enables existence and uniqueness of the invariant measure $\mu^{t,x}(dy)$ corresponding to the frozen equation,
$$
dY^{t,x}_{s}=f(t,x,Y^{t,x}_{s})ds+\delta_{2}(t,x,Y^{t,x}_{s})dB^{2}_{s},\ Y_{0}=y\in \mathbb{R}^{d_{2} }.
$$

Since the averaging principle shows the convergence of multiscale systems toward the averaged equation, we further aim to study the strong or weak convergence rates. In his seminal work \cite{CEB2}, C.-E. Bréhier employed Khasminskii's time discretization method to achieve strong convergence rates and asymptotic expansion of Kolmogorov equations for weak convergence rates. 
Subsequently, in \cite{CEB} on semilinear SPDEs with multiscale dynamics, by Poisson equation, he demonstrated that optimal strong and weak convergence orders are $\frac{1}{2}$ and $1$ respectively.Meanwhile, \cite{DIR,RZK,DL} also applied Khasminskii's time discretization to derive strong convergence rates and Kolmogorov equations for weak convergence rates. Compared to these two approaches, the Poisson equation method offers significant advantages in determining convergence rates, and Pardoux and Veretennikov developed this method in \cite{EY1,EY2,EY3}.
%\begin{align*}		\mathcal{L}u(x,y)+g(x,y)=0,
		%=\sum_{i,j=1}^{d_{2}}a_{i,j}(x,y)\frac{\partial^{2}}{\partial y_{i}\partial y_{j}}u(x,y)+\sum_{i=1}^{d_{2}}f_{i}(x,y)\partial _{y_{i}}u(x,y)+g(x,y) \end{align*}
%where $x\in \mathbb{R}^{d_{1}}$ is fixed and $y\in \mathbb{R}^{d_{2}}$, $	\mathcal{L}$ is the infinitesimal generator of frozen process, the probabilistic representation of solution $u(x,y)$ for $Y_{t}^{x,y} \in \mathbb{R}^{d_{2}}$ is
  %in bounded domain $D$ with a smooth boundary and zero boundary condition (Dirichlet boundary condition) is \begin{equation}	\begin{aligned}	u(x,y)=\int_{0}^{\tau}\mathbb{E}g(x,Y_{t}^{x,y})dt,\ \tau=inf\{t>0, Y_{t}^{x,y}\notin D\},	\end{aligned}\nonumber\end{equation}	\begin{align}\label{1-1}		u(x,y)=\int_{0}^{\infty}\mathbb{E}g(x,Y_{t}^{x,y})dt,	\end{align}then zero centering condition for $g(x,y)$ with respect to invariant measure of frozen process
%\begin{equation}\label{1-2}
%	\begin{aligned}	\bar{g}(x)=\int_{\mathbb{R}^{d_{2}}}g(x,y)\mu^{x}(dy)=0,		\end{aligned}\nonumber
%	\end{equation}and ergodicity of $Y_{t}^{x,y}$ are vital to enable the existence of $u(x,y)$ and its local boundedness, see \cite[Theorem 1]{EY1}. 
%\cite{CSS,DSXZ,FWL,PG,PXW,WC} investigated both strong and weak convergence rates of multiscale dispersive equations and hyperbolic equations. To be concluded, the Poisson equation framework guarantees  convergence rates of heterogeneous multiscale schemes efficitively.

In recent years, slow-fast systems driven by non-Gaussian noises have attracted considerable attention, since heavy-tailed random perturbations are often more suitable than Gaussian fluctuations for describing abrupt transitions and anomalous transport phenomena.  Slow-fast systems driven by jump processes also have seen fruitful results in recent years such as \cite{CSS,YXJ,ZHWWD}. X.-B. Sun et al \cite{SXX} studied a slow-fast system driven by independent additive noises $\alpha$-stable processes $L^{1}_{t}$ and $L^{2}_{t}$, where $\alpha\in(1,2)$,
\begin{equation}\label{3}
	\left\{
	\begin{aligned}
			& dX_{t}^{\varepsilon}=b(X_{t}^{\varepsilon},Y_{t}^{\varepsilon})dt+dL^{1}_{t},\ X^{\varepsilon}_{0}=x\in \mathbb{R}^{d_{1} },\\	&dY_{t}^{\varepsilon}=\frac{1}{\varepsilon}f(X_{t}^{\varepsilon},Y_{t}^{\varepsilon})dt+\frac{1}{\varepsilon^{\frac{1}{\alpha}}}dL^{2}_{t},\ Y^{\varepsilon}_{0}=y\in \mathbb{R}^{d_{2}},
			\end{aligned}	
			\right.
			\end{equation}
 they constructed nonlocal Poisson equations to show that the optimal strong convergence order of $X^{\varepsilon}_{t}$ is $1-\frac{1}{\alpha}$, and the weak convergence order is $1$. 
 
‌It is crucial to note‌ that all aforementioned references, as well as other existing studies on multiscale systems driven by $\alpha$-stable processes, exclusively consider models without jump coefficients, since the additive Lévy noises and dissipative drift can simplify the derivation of exponential ergodicity and regularity estimates for the frozen equation.

There are extensive references concerning the ergodicity of $\al$-stable processes. \cite{JW2} demonstrated exponential contractivity with respect to $L^p$-Wasserstein distance of SDE driven by $\al$-stable process with partially dissipative drift via coupling method.  \cite{DW} considered nonlocal operators with jump coefficients and variable orders, under uniform elliptic condition and continuity conditions, they established Hölder regularity by similar methods.  More ergodicity and regularity results can be found in \cite{BSWX, LW, MBM}.
\cite{VK} investigated asymptotic expansions for finite-dimensional symmetric stable distributions, %by which they derived the existence of continuous transition probability densities of stable-like jump-diffusions whose spherical measure was dependent on the process. 
%in particular, 
they obtained gradient estimates of probability densities, %constructed local multiplicative asymptotics,
and global two-sided estimates for these densities. Notably, the existence of transition density still holds even if some smoothness conditions are relaxed in \cite{ANK}, while gradient estimates do not. By applying the global two-sided estimates and gradient estimates to spatially periodic stable processes, and using Doeblin's celebrated result on invariant measures, \cite{BF} derived exponential ergodicity.

\subsection{Sketch of this paper}
In this paper we study strong and weak convergence rates of the following time-inhomogeneous multiscale system driven by $\al$-stable processes, which is widely used in  climate models  \cite{YK,MTV}, geophysical
fluid flows \cite{GD}, and other areas \cite{CK,PS}. For independent isotropic  $\alpha$-stable processes  $L^{1}_{t}$, $L^{2}_{t}$, with $1<\alpha_{1},\alpha_{2}<2$ and $\varepsilon \rightarrow 0$, the system is given by:
\begin{equation}\label{1.1}
\left\{
\begin{aligned}
& dX_{t}^{\varepsilon}=b(t,X_{t}^{\varepsilon},Y_{t}^{\varepsilon})dt+\delta_{1}(t,X_{t}^{\varepsilon},Y_{t}^{\varepsilon})dL^{1}_{t},\ X^{\varepsilon}_{0}=x\in \mathbb{R}^{d_{1} },\\
&dY_{t}^{\varepsilon}=\frac{1}{\varepsilon}f(X_{t}^{\varepsilon},Y_{t}^{\varepsilon})dt+\frac{1}{\varepsilon^{\frac{1}{\alpha_{2}}}}\delta_{2}(X_{t}^{\varepsilon},Y_{t}^{\varepsilon})dL^{2}_{t},\ Y^{\varepsilon}_{0}=y\in \mathbb{R}^{d_{2}},
\end{aligned}
\right.
\end{equation}
here $X_{t}^{\varepsilon}$  denotes the slow component, whose drift and jump coefficient are time dependent, whereas $Y_{t}^{\varepsilon}$ denotes the fast component, whose dynamics involve multiplicative stable noise through the coefficient $\delta_2$.
The main difficulties of the present work stem from the multiplicative jump coefficients $\delta_1$ and $\delta_2$, see Remark \ref{d1-d2} for a detailed discussion.

The frozen equation associated with the fast component is
$$	dY_{t}^{x,y}=f(x,Y_{t})dt+\delta_{2}(x,Y_{t})dL^{2}_{t},\ Y_{0}=y\in \mathbb{R}^{d_{2} },$$
where $x$ is fixed.
A major difficulty arises from the multiplicative L\'evy coefficient $\delta_2$, which makes the corresponding infinitesimal generator a stable-like operator with state-dependent jump kernel. In contrast to the additive case, both the exponential ergodicity and the regularity theory of the associated Poisson equation become substantially more complex.

To overcome these difficulties, we first construct a coupling process and establish exponential contractivity of the frozen semigroup in the $L^p$-Wasserstein distance under partially dissipative condition. This yields the existence and uniqueness of the invariant measure $\rho^x$ and exponential ergodicity.

Based on the Wasserstein contractivity of the frozen semigroup, we derive a gradient estimate for the nonlocal Poisson equation associated with the stable-like operator $\mathcal{L}_2(x,y)$, which can be regarded as the infinitesimal generator of $Y_t^{x,y}$,
$$
\mathcal{L}_2(x,y)u+g-\bar g=0,
$$
here $\bar{g}$ is the average of $g$ with respect to $\rho^x$, the gradient estimate constitutes the key analytical ingredient in the proofs of both the strong and weak averaging principles, see \eqref{5.4-1} in Theorem $\ref{T51-1}$.

We next establish strong convergence rate between  $X_{t}^{\varepsilon}$ and averaged process $\bar{X}_{t}$ in Theorem \ref{SCR}. In this case, we suppose that  $\delta_{1}(t,x,y)=\delta_{1}(t)$, then we obtain averaged process 
\begin{align*}
	d\bar{X}_{t}=\bar{b}(t,\bar{X}_{t})dt+\delta_{1}(t)dL_{t}^{1},
\end{align*}
the exponent $v\in((\alpha_{1}-\alpha_{2})^{+},\alpha_{1}]$, where $(a)^{+}=\text{max}\{a, 0\}$, governs the Hölder regularity of $b(t,x,y)$ and $\delta_{1}(t)$ with respect to $t$ and $x$, and  plays an important role  in our analysis. Significant simplifications of convergence orders emerge when  $v\geq1$,  
\begin{align*}
	\varepsilon^{\left[ \left( \frac{v}{\alpha_{2}}\right) \wedge \left( 1-\frac{1\vee(\alpha_{1}-v)}{\alpha_{2}}\right) \right] }=\varepsilon^{\left[1-\frac{1-(1\wedge v)}{\alpha_{2}}\right]}=\varepsilon^{1-\frac{1}{\alpha_{2}}},
\end{align*}
we highlight that  this result corresponds to the optimal strong convergence order $1-\frac{1}{\alpha_2}$ for \eqref{3} proposed in \cite{SXX}, see more details in Remark \ref{c7s}.

Weak convergence rate of $X_{t}^{\varepsilon}$ is studied in Theorem \ref{WCR}, the averaged equation is represented as 
\begin{align*}
	d\bar{X}_{t}=\bar{b}(t,\bar{X}_{t})dt+\bar{\delta}_{1}(t,\bar{X}_{t})dL_{t}^{1},
\end{align*}
when $v=\alpha_{1}=\alpha_{2}$, 
\begin{align*}
\varepsilon^{\frac{v}{\alpha_{2}}}=\varepsilon^{1-\frac{\alpha_{1}-v}{\alpha_{2}}}=\varepsilon,
\end{align*}
we observe that the above convergence order is consistent with weak convergence order 1 for system \eqref{3} stated in \cite{SXX}, see more discussions in Remark \ref{c7w}.

\begin{remark}\label{d1-d2}
We assume that $\delta_{1}(t,x,y)=\delta_{1}(t)$ in the sense of strong convergence to construct corrector equation. Although one may consider the formulation with $\delta_{1}(t,x,y)=\delta_{1}(t,x)$ proposed in \eqref{1} of \cite{LRSX}, this setting presents essential difficulties in our paper, as discussed in Remark $\ref{d1}$.

The difficulties arising from jump coefficient $\de_{2}(x,y)$ can be summarized as follows:
\begin{itemize}
    \item[(i)] Compared with models driven by additive Lévy noises, it is more challenging to establish the existence of  invariant measure and exponential ergodicity for the frozen equation.
We overcome these difficulties by the coupling method, see  Theorem $\ref{was-exp}$ and Section $\ref{inv-gen}$.
\item[(ii)] The derivation of the crucial gradient estimate for the Poisson equation associated with a state-dependent jump kernel is considerably more involved, see \eqref{5.4-1} in Theorem $\ref{T51-1}$ for details.
\end{itemize}
\end{remark}

Although sufficient  Hölder regularity conditions of drifts lead to the convergence rates obtained in the additive setting of \cite{SXX}, 
the analytical framework developed here is substantially different due to the presence of state-dependent jump kernels induced by multiplicative Lévy noises.

The main contributions of this paper can be summarized as follows.
\begin{itemize}
    \item We establish exponential contractivity of the frozen dynamics in the $L^p$-Wasserstein distance for stable-like processes with state-dependent jump coefficients.
    
    \item We derive a gradient estimate for the nonlocal Poisson equation associated with the stable-like operator with state-dependent jump kernel.
     
    \item We obtain strong and weak averaging rates for multiscale systems driven by multiplicative Lévy noises. In particular, under sufficient H\"{o}lder regularity of the time-dependent coefficients of the slow process, we ‌can yield optimal strong convergence rate of order $1-\frac{1}{\alpha_{2}}$.
\end{itemize}

\subsection{Organization of this paper}
The remainder of this paper is organized as follows. In Section \ref{setting}, we introduce the notation, assumptions, and state the main results concerning the strong and weak convergence rates, as well as the exponential contractivity of the frozen equation. Section \ref{moment} is devoted to the well-posedness and uniform moment estimates  of \eqref{1.1}. In Section \ref{inv}, we establish exponential contractivity in the $L^p$-Wasserstein distance and derive the existence and uniqueness of the invariant measure of the frozen equation associated with the fast equation by coupling method. In Section \ref{str-est}, we derive a crucial gradient estimate for the nonlocal Poisson equation, and the strong convergence result is constructed. In Section \ref{wea-est}, we study the weak convergence rates. Section \ref{main} contains the proofs of  Theorem \ref{SCR} and Theorem \ref{WCR} Finally, Appendix \ref{A} provides an auxiliary geometric identity used in the transformation of the state-dependent Lévy measure appearing in Section \ref{inv}.
%We begin by introducing some background on the multiscale system. In Section 2, we outline the key assumptions and state our main results. Section  \ref{moment} is devoted to the well‑posedness of \eqref{1.1}, with moment estimates for $(X^{\varepsilon}_{t},Y^{\varepsilon}_{t})$  presented in Theorem \ref{T32}. Section \ref{inv} is devoted to the invariant measure and exponential ergodicity of the frozen equation via coupling techniques. An auxiliary geometric identity required for the transformation of the state-dependent L\'evy measure is deferred to Appendix \ref{A}. Section \ref{str-est} establishes strong convergence of  $X^{\varepsilon}_{t}$ by constructing nonlocal Poisson equations,  where we also derive some $L^{p}$ estimates and regularity estimates. Section \ref{wea-est} is devoted to weak convergence estimates. Finally, the proofs of Theorem \ref{SCR}, Theorem \ref{WCR}, and Theorem \ref{tan} are presented in Section \ref{main}.

\section{Some settings and main results}\label{setting}

\subsection{Notations and assumptions}
We next give some notions and definitions concerning calculations in $d_{i}$-dimensional Euclidean space $\mathbb{R}^{d_{i} }(d_{i}\geq 1) $, we mention that  $\mathbb{R}^{d_{1} }$ and $\mathbb{R}^{d_{2}}$ are equipped with disadjoint orthogonal basis. 
$\langle \cdot \rangle $ denotes inner product, $|\cdot|$ denotes the Euclidean vector norm, $|x|=\sqrt{\sum_{i=1}^{d} x_i^{2}}$, $\|\cdot\|$ denotes the matrix norm and operator norm. 
Let $(\Omega,\mathcal{F},\mathbb{P})$ be the probability space  that describes random environments, denote by $ \mathbb{E}$ the expectation with respect to the probability measure $\mathbb{P}$. Denote $(a)^{+}=\text{max}\{a, 0\}$.

For any $ k\in \mathbb{N}^+$, $\delta\in (0,1)$, we define

$ C^{k}(\mathbb{R}^{d} )$=\{$u:\mathbb{R}^{d}\longrightarrow \mathbb{R}$: $u$ and all its partial derivatives up to order $k\geq0$ are continuous\}.

$ C^{k}_{b}(\mathbb{R}^{d} )$=\{$u\in  C^{k}(\mathbb{R}^{d})$: $u$ and all its partial derivatives up to order $k\geq0$ are bounded continuous\}.

$ C^{k+\delta}_{b}(\mathbb{R}^{d} )$=\{$u\in  C^{k}_{b}(\mathbb{R}^{d})$: $u$ and all its  partial derivatives up to order $k\geq0$ are $\delta$-H\"{o}lder continuous\}.

Then $ C_b^{k+\de}\subset C_b$ for $k\geq0,  \de\in(0,1)$. The spaces $ C^{k}_{b}$, $C^{k+\delta}_{b}$  equipped with $\Arrowvert \cdot\Arrowvert_{C^{k}_{b}}$ and $\Arrowvert \cdot\Arrowvert_{C^{k+\delta}_{b}}$ are Banach spaces. We emphasize that $u\in C^{k_{1}+\delta_{1},k_{2}+\delta_{2}}_{b}(\mathbb{R}^{d}\times\mathbb{R}^{d} )$ means that: 
(i). For $0<|\beta_{1}|<k_{1}$,  $0<|\beta_{2}|<k_{2}$, $\partial_{x}^{\beta_{1}}\partial_{y}^{\beta_{2}}u$ is bounded continuous; (ii).  $\partial_{x}^{k_{1}}u$ is $\delta_{1}$-H\"{o}lder continuous with respect to $x$ uniformly in $y$,   $\partial_{y}^{k_{2}}u$ is $\delta_{2}$-H\"{o}lder continuous with respect to $y$ uniformly in $x$. We denote that $f(\cdot,x,y)\in C_{b}^{v,\delta_{1},\delta_{2}}$ if $\forall (x,y)\in \mathbb{R}^{d_{1}+d_{2}}$, $f(\cdot,x,y)\in C_{b}^{v}(\mathbb{R_{+}})$, $f(t,\cdot,\cdot)\in C_{b}^{\delta_{1},\delta_{2}}(\mathbb{R}^{d_{1}+d_{2}})$. $X^{x,y}_{t}$  denotes the process $X_{t}$ starting from $(x,y)$. 

Given a
function space, the subscript $b$ will stand for boundness, while the subscript $p$ stands for polynomial growth in $y$ at order $1$, we mainly refer to  \cite[page.1209]{RX}.
%Additionally, we must introduce a function space $C^{\gamma,\eta,\delta}_p(\mb{R}^{+}\times\mathbb{R}^{d_{1}}\times\mathbb{R}^{d_{2}})$. $\forall \gamma, \eta, \delta\in(0,1)$,  $C^{\gamma,\eta,\delta}_p(\mb{R}^{+}\times\mathbb{R}^{d_{1}}\times\mathbb{R}^{d_{2}})$ consists of all functions which are local H\"{o}lder continuous and have polynomial growth in $y$ at most of order $1$ uniformly with respect to $t, x$.   
We start with defining  $C_p(\mb{R}^{+}\times\mathbb{R}^{d_{1}}\times\mathbb{R}^{d_{2}})$, i.e., for $f\in C_p(\mb{R}^{+}\times\mathbb{R}^{d_{1}}\times\mathbb{R}^{d_{2}})$, $\forall t\in \R^{+}$, $ x\in \R^{d_{1}}$,  $\exists C>0$, 
\begin{align*}
|f(t,x,y)|\leq C (1+|y|),
\end{align*}
\iffalse
then for $f\in C^{\gamma,0,0}_p(\mb{R}^{+}\times\mathbb{R}^{d_{1}}\times\mathbb{R}^{d_{2}})$, $\forall x\in \R^{d_{1}}$, $\exists C>0$,  
\begin{align*}
	\sup_{\substack{x\in \R^{d_{1}}}}|f(t_{1},x,y)-f(t_{2},x,y)|\leq C [|t_{1}-t_{2}|^{\ga}]\cdot(1+|y|),
\end{align*}

$ C^{0,\eta,0}_p(\mb{R}^{+}\times\mathbb{R}^{d_{1}}\times\mathbb{R}^{d_{2}})$ is defined in a similar way, 
\fi
for $f\in C^{0,\de}_p(\mathbb{R}^{d_{1}}\times\mathbb{R}^{d_{2}})$, $0<\delta\leq1$, $\forall x\in \R^{d_{1}}$,  $\exists C>0$, 
\begin{align*}
|f(x,y_{1})-f(x,y_{2})|\leq C [|y_{1}-y_{2}|^{\de}\wedge1]\cdot(1+|y_1|+|y_2|),
\end{align*}
similarly, for  $f\in C^{\eta,\delta}_p(\mathbb{R}^{d_{1}}\times\mathbb{R}^{d_{2}})$, $0< \eta, \delta\leq1$,  $\exists C>0$,
\begin{align}\label{cp}
	|f(x_{1},y_{1})-f(x_{2},y_{2})|\leq C [(|x_{1}-x_{2}|^{\eta}\wedge1)+(|y_{1}-y_{2}|^{\delta}\wedge1)]\cdot(1+|y_{1}|+|y_{2}|),
\end{align}
then we define the quasi-norm, 
\begin{align*}
	\|f\|_{C^{\eta,\delta}_p}=\sup_{|x_{1}-x_{2}|\leq1}\sup_{\substack{|y_1|\leq|1,y_2|\leq1, |y_{1}-y_{2}|\leq1}}\dfrac{	|f(x_{1},y_{1})-f(x_{2},y_{2})|}{|x_{1}-x_{2}|^\eta+|y_{1}-y_{2}|^\delta},
\end{align*}
then for $k_1,k_2\in \mathbb{N}^+$, $0<\eta,\delta\leq1$, $C^{k_1+\eta,k_2+\delta}_p(\mathbb{R}^{d_{1}}\times\mathbb{R}^{d_{2}})$=\{$u\in  C^{k_1,k_2}_{p}(\mathbb{R}^{d_{1}}\times\mathbb{R}^{d_{2}})$: $u$ and all its  partial derivatives up to order k are $C^{\eta,\delta}_p$\}. 

Finally, $C^{\gamma, \eta, \delta}_p(\mathbb{R}^{+}\times\mathbb{R}^{d_{1}}\times\mathbb{R}^{d_{2}})$  with $0<\gamma\leq1$ denotes the space of all functions $f$ such that for every fixed $t>0$, $f(t,\cdot,\cdot)\in C^{\eta, \delta}_p(\mathbb{R}^{d_{1}}\times\mathbb{R}^{d_{2}})$, $f(\cdot,x,y)\in C_b^\gamma(\mathbb{R}^{+})$, where $C_b^\gamma(\mathbb{R}^{+})$ is the usual bounded continuous Hölder space.
\begin{remark}
The order of $|y|$ on the right hand side is $1$ in \eqref{cp}, which is important in our analysis, due to we need $m\in[1, \alpha_{1}\wedge\alpha_{2})$ in moment estimates of $|Y_{t}^{\varepsilon}|^{m}$ in Theorem $\ref{T32}$, moreover, this is consistent with estimates in Lemma $\ref{L51}$, Theorem $\ref{T52}$, and Theorem $\ref{T62}$.
\end{remark}

Denote by $(P_t)_{t\ge0}$ the semigroup associated to
\eqref{4.1}, if the initial distribution of $Y_0=y$ is $\mu$,
then for any $t>0$, the distribution of $Y_t$ is $\mu P_t$. We next study the exponential contractivity of $\mu\rightarrow\mu P_t$ with respect to the standard $L^p$-Wasserstein
distance $W_p$, $\forall p\geq1$, which is defined as follows. For any two probability measures $\mu$ and $\nu$ on $\mathbb{R}^d$,
the standard
$L^p$-Wasserstein distance
$W_p$ for all $p\in[1,\infty)$ is 
\begin{align*}
	W_p(\mu,\nu)=\inf_{\Pi\in\mathcal{C}(\mu,\nu)} \Big(\int
	_{\mathbb{R}^d\times\mathbb{R}^d} |x-y|^pd\Pi(x,y)
	\Big)^{1/p},
\end{align*}
here $\mathcal{C}(\mu,\nu)$ is the space of all joint distributions with $\mu$ and $\nu$ as marginal distributions,
% the totality $\mathcal P_p(\mathbb{R}^d)$ of probabilitymeasures having finite moment of order $p$becomes complete metric space under $W_p$.
moreover,
\begin{align*}
	\mathbb{E}|X_t-Y_t|^p=\int
	_{\mathbb{R}^d\times\mathbb{R}^d} |x-y|^pd\Pi_t(x,y),
\end{align*}
where $\Pi_t\in\mathcal{C}(\delta_xP_t,\delta_yP_t)$, so we certainly have $W_p(\delta_xP_t,\delta_yP_t)^p\leq\mathbb{E}|X_t-Y_t|^p$.

Define $K_{t}$ as an $\mathbb{R}_{+}$-valued $\mathcal{F}_{t}$ adapted process such that  $\forall m\in [1,\alpha_{1}\wedge\alpha_{2})$ we suppose that 
\begin{align}\label{kt}
	\alpha_{\infty}=\int_{0}^{\infty}|K_{s}|^mds<\infty\ on \ \Omega,\  \mathbb{E}e^{m\alpha_{\infty}}<\infty,
\end{align}
and we assume that $\nu_{1}$ and $\nu_{2}$ are symmetric L\'{e}vy measures, i.e., $$\int_{\mathbb{R}^{d_{i} }}(|z|^{2}\wedge 1)\nu_{i}(dz)<\infty, \ i=1,2.$$

Define the nonlocal operator corresponding to $X_t^\varepsilon$ in \eqref{1.1} as follows
\begin{align*} 
	\mathcal{L}_{1}(t,x,y)u(x,y)&=	-(-\Delta_{x})^{\frac{\alpha_{1}}{2}}u(x,y)+b(t,x,y)\nabla_{x}u(x,y)\\
	&=P.V.\int_{\mathbb{R}^{d_{1}}}\Big(u(x+\delta_{1}(t,x,y)z,y)-u(x,y)-\langle\delta_{1}(t,x,y)z,\nabla_{x}u(x,y)\rangle I_{|z|\leq1}\Big)\nu_{1}(dz)\\
	&+b(t,x,y)\nabla_{x}u(x,y),
	\end{align*}
	here $\nu_{1}(dz)=\frac{c_{\alpha_{1},d_{1}}}{|z|^{d_{1}+\alpha_{1}}}dz$ is symmetric L\'{e}vy measure, $c_{\alpha_{1},d_{1}}>0$ is constant. The operator
	$ \mathcal{L}_{2}(x,y)$  associated with $Y_t^\varepsilon$  is given by a similar expression:
\begin{align}\label{l2}
	 \mathcal{L}_{2}(x,y)u(x,y)&=	-(-\Delta_{y})^{\frac{\alpha_{2}}{2}}u(x,y)+f(x,y)\nabla_{y}u(x,y)\nonumber\\
	 &=P.V.\int_{\mathbb{R}^{d_{2}}}\Big(u(x,y+\delta_{2}(x,y)z)-u(x,y)-\langle\delta_{2}(x,y)z,\nabla_{y}u(x,y)\rangle I_{|z|\leq1}\Big)\nu_{2}(dz)\nonumber\\
	 &+f(x,y)\nabla_{y}u(x,y),
	 \end{align}
where $\nu_{2}(dz)=\frac{c_{\alpha_{2},d_{2}}}{|z|^{d_{2}+\alpha_{2}}}dz$ is symmetric L\'{e}vy measure, $c_{\alpha_{2},d_{2}}>0$ is constant.

Then, we use the spherical coordinates $(|z|,\hat{z})=(|z|,\frac{z}{|z|})\in\mathbb{R}^{+}\times\s^{d_{2}-1}$ to define the $\sigma$-finite measure on $\mathbb{R}^{d_{2}}\setminus\{0\}$ as follows
$$\nu_{2}(dz)=\bar{\nu}_{2}(d\hat{z})\frac{d|z|}{|z|^{1+\alpha_{2}}},$$
here $\bar{\nu}_{2}$ is a symmetric finite Borel measure on $\mathbb{S}^{d_{2}-1}$.

We next impose some important assumptions.

\begin{assumption}[Uniformly elliptic conditions]\label{uni-ell}
	\
\begin{enumerate}
\item  $\forall x\in\R^{d_{1}}$, $ y\in\R^{d_{2}}$, $ \hat{\om}\in\s^{d_{2}-1}$, we have  %$\delta_{2}(x,y):\mb{R}^{d_{1}}\times\mb{R}^{d_{2}}\rightarrow GL(\mb{R}^{d_{2}})$, and 
$c_l\le|\de_2(x,y)\cdot \hat{\om}|\le c_u$;

%\noindent\textbf{Uniformly elliptic condition of $\de_2$}: 
\item  For any unit vector $\bar{p}$ in the direction of $p\in\mathbb{R}^{d_{2}}$, i.e., $\bar{p}=\frac{p}{|p|}$, there exist $C_{l},C_{u}>0$ s.t.
\begin{align*}
	C_{l}\leq\int_{\mathbb{S}^{d_{2}-1}}|\langle\bar{p},\hat{z}\rangle|^{\alpha_{2}}\bar{\nu}_{2}(d\hat{z})\leq C_{u}.
\end{align*}
\end{enumerate}
\end{assumption}
\begin{remark}
\textbf{Uniformly elliptic conditions} hold throughout this paper. They facilitate the well-posedness of the change of variable in \eqref{l2-2}. Moreover, the conditions %and the differentiability of $f$ and $\delta_2$ with respect to $y$
 prevent the stochastic process from degenerating to a drift-only process. %, but also plays a vital role in deriving uniform boundedness of the transition probability for \eqref{2.23} and the crucial regularity estimate \eqref{5.4-1} in strong convergence analysis.

Uniform ellipticity of $\de_{2}$ in Assumption $\ref{uni-ell}$-(1) is also necessary for \eqref{psi}, see Lemma $\ref{inv-gen-L}$.

%In periodic case, condition (i) and (ii) enable the probability density of  \eqref{2.23} to be bounded from below by a positive constant, which is crucial to derive exponential ergodicity, see Lemma $\ref{L41}$.
\end{remark}

\begin{assumption}[Partially dissipative condition]\label{par-dis}
 $\forall x\in \mathbb{R}^{d_{1}}$, $y\in \mathbb{R}^{d_{2}}$,  $t\geq0$, $\exists c, C>0$,
let 
\begin{align*}
	\sup_{x\in \mathbb{R}^{d_{1}}}f(x,0)<\infty,  \quad \sup_{t\geq0}\sup_{y\in \mathbb{R}^{d_{2}}}b(t,0,y)<\infty, 
\end{align*}
and we have
\begin{equation}\label{2.2}
	\begin{split}
&\langle b(\cdot,x_{1},\cdot)-b(\cdot,x_{2},\cdot), x_{1}-x_{2}\rangle  \leq c|x_{1}-x_{2}|^{2}I_{\{|x_{1}-x_{2}|\leq L_0\}}-C_B|x_{1}-x_{2}|^{2}I_{\{|x_{1}-x_{2}|>L_0\}},\\
	& \langle f(x,y_{1})-f(x,y_{2}), y_{1}-y_{2}\rangle  \leq c|y_{1}-y_{2}|^{2}I_{\{|y_{1}-y_{2}|\leq L_0\}}-C_F|y_{1}-y_{2}|^{2}I_{\{|y_{1}-y_{2}|>L_0\}},	\end{split}
\end{equation}	
\eqref{2.2} implies that there exists $ C_{1}>0$ s.t.,
\begin{equation}\label{2.4}
	\begin{split}
		&\langle f(x,y),y \rangle=\langle f(x,y)-f(x,0), y \rangle +\langle f(x,0),y\rangle   \leq C_{1}|y|-C_F|y|^{2},\\
			&\langle b(t,x,y),x \rangle =\langle  b(t,x,y)-b(t,0,y), x\rangle  + \langle b(t,0,y),x\rangle  \leq C_{1}|x|-C_B|x|^{2}.
	\end{split}
\end{equation}		
\end{assumption}

\begin{assumption}[Growth condition and boundedness condition]\label{gro-bou}
$\forall x\in \mathbb{R}^{d_{1}}$, $y\in \mathbb{R}^{d_{2}}$, $t\geq0$, $\exists C_{4},C_{5}>0$ s.t.,
\begin{align*}
	&|b(t,x,y)| \leq C_{4}(1+K_{t}),\ |f(x,y)| \leq C_{5}(1+|x|+|y|),\ \|\de_1(t,x,y)\|_{\infty} \leq C_{4},
\end{align*}
here we mention that uniform ellipticity of $\de_{2}$ in Assumption $\ref{uni-ell}$-(1) implies the boundedness of $\de_{2}$.
\end{assumption}
\begin{remark}
Due to the restrictions on the orders of moment estimates in Theorem $\ref{T32}$, the assumption \eqref{kt} and growth condition of $b$ are crucial to derive \eqref{5.50} in Theorem $\ref{T52}$.
\end{remark}

\begin{assumption}[Lipschitz condition]\label{lip}
 $\forall x\in \mathbb{R}^{d_{1}},\ y\in \mathbb{R}^{d_{2}}$, $t_{1},t_{2}\in [0,T],$  $ C_{T},C_{6}>0$ s.t.,
\begin{align*}
	&|b(t_{1},x_{1},y_{1})-b(t_{2},x_{2},y_{2})| \leq C_{T}(|t_{1}-t_{2}|+|x_{1}-x_{2}|+|y_{1}-y_{2}|),\\
	&\|\de_1(t_{1},x_{1},y_{1})-\de_1(t_{2},x_{2},y_{2})\| \leq C_{T}(|t_{1}-t_{2}|+|x_{1}-x_{2}|+|y_{1}-y_{2}|),\\
	&|f(x_{1},y_{1})-f(x_{2},y_{2})| \leq C_{6}(|x_{1}-x_{2}|+|y_{1}-y_{2}|),
	\\
	&\|\de_2(x_{1},y_{1})-\de_2(x_{2},y_{2})\| \leq C_{6}(|x_{1}-x_{2}|+|y_{1}-y_{2}|).
\end{align*}
\end{assumption}

%\noindent\textbf{Boundedness condition}: $\forall x\in \mathbb{R}^{d_{1}}$, $y\in \mathbb{R}^{d_{2}}$, $t\geq0$,
%\begin{equation}\nonumber\begin{split}			&\|\delta_{1}(t,x,y)\|_{\infty} < \infty,\  \|\delta_{2}(x,y)\|_{\infty} < \infty.		\end{split}\end{equation}
\iffalse
\begin{assumption}[Centering condition]\label{cen}
 $\forall t\geq0, x\in \mathbb{R}^{d_{1}}, y\in \mathbb{R}^{d_{2}}$, for  $g\in C_b(\mathbb{R}^{ d_{2}})$, %while in spatial periodic case we let $g^\Lambda\in C(\mathbb{R}^{ d_{2}}/\Lambda)$, then
\begin{equation}\label{2.13}
\int_{\mathbb{R}^{ d_{2}}}g(t,x,y)\rho^{x}(dy)=0, %\quad \text{or}\quad	\int_{\mathbb{R}^{ d_{2}}/\Lambda}g^\Lambda(t,x,y)\mu^{x}(dy)=0,	
\end{equation}
here $\rho^x$ is invariant measures of frozen equation \eqref{2.23}. % while $\mu^{x}$ is invariant measure defined by image of frozen process  \eqref{2.23}, which is $Y^{\Lambda,x,y}_t$ on $\mathbb{R}^{ d_{2}}/\Lambda$ in spatial periodic case.
\end{assumption}
\fi

\subsection{Main results}
  Next we state the main results of this paper.

\begin{theorem}[Strong convergence rate]\label{SCR}
Assume that $\delta_{1}(t,x,y)=\delta_{1}(t)$, let Assumption $\ref{uni-ell}$-Assumption $\ref{lip}$ all hold,  $b(\cdot,\cdot,\cdot)\in C_{p}^{\frac{v}{\alpha_{1}},v,2+\gamma}(\mb{R}^{+}\times\mathbb{R}^{d_{1}}\times\mathbb{R}^{d_{2}})\cap C_b^{2}(\mathbb{R}^{d_{2}})$, $\delta_{1}(t)\in C_{p}^{\frac{v}{\alpha_{1}}}$,  $f(\cdot,\cdot)\in C_{b}^{v, 2+\gamma}$, $\de_{2}(\cdot,\cdot)\in C_{b}^{v, 2+\gamma}$,  $v\in((\alpha_{1}-\alpha_{2})^{+},\alpha_{1}]$,  $\gamma\in (0,1)$, for any initial data $x\in \mathbb{R}^{d_{1}}$, $y\in \mathbb{R}^{d_{2}}$, $T>0$, $t\in[0,T]$, $m\in [1,\alpha_{1}\wedge\alpha_{2})$, %for spatial periodic case  we let \textbf{A1, AC, AS, A2, A3, A4, A5} hold, 
we have:
\begin{align*}
\mathbb{E}\left( \sup_{t\in [0,T]}|X_{t}^{\varepsilon}-\bar{X}_{t}|^{m}\right) \leq   C_{T,m,x,y}\cdot\left(\varepsilon^{m\left[ \left( \frac{v}{\alpha_{2}}\right) \wedge \left( 1-\frac{1\vee(\alpha_{1}-v)}{\alpha_{2}}\right) \right] }+\varepsilon^{m(1-\frac{1-(1\wedge v)}{\alpha_{2}})}\right),
\end{align*}
here 
\begin{align*}
	d\bar{X_{t}}=\bar{b}(t,\bar{X}_{t})dt+\delta_{1}(t)dL_{t}^{1},
\end{align*}
we define the averaged coefficient $\bar{b}(t,x)$ as,  %or  in spatial periodic case by \eqref{pr} as follows
\begin{equation}\label{b}
\bar{b}(t,x)=\int_{\mathbb{R}^{d_{2}}}b(t,x,y)\rho^{x}(dy), %\quad\text{or}\quad  \bar{b}(t,x)=\int_{\mathbb{R}^{d_{2}}/\Lambda}b^{\Lambda}(t,x,y)\mu^{x}(dy),
\end{equation}
where $\rho^x$ is invariant measures of frozen process \eqref{2.23}, %in general case, however, in spatial periodic case $\mu^{x}(dy)$ is the invariant measure for transition semigroup of $Y^{\Lambda,x,y}_{t}$ on $\mathbb{R}^{d_{2}}/\Lambda$, which is the image of frozen process  $Y^{x,y}_{t}$ with respect to the projection $R_{\Lambda}:\mathbb{R}^{d_{2}}\rightarrow\mathbb{R}^{d_{2}}/\Lambda$, 
here $Y^{x,y}_{t}$ is the frozen equation, 
\begin{equation}\label{2.23}
	dY^{x,y}_{t}=f(x,Y_{t})dt+\delta_{2}(x,Y_{t})dL^{2}_{t},\ Y_{0}=y\in \mathbb{R}^{d_{2}}.
\end{equation}
\end{theorem}

%\begin{remark}
%We remind that the gradient estimate $|\nabla_{y}u|$,  derived from heat kernel asymptotic expansion of transition density where regularities of $f$ and $\de_2$ with respect to $y$ are important, is crucial in strong convergence analysis, see \eqref{5.4-1} in Theorem $\ref{T51-1}$, \eqref{5.54} in Theorem $\ref{T52}$.
%\end{remark}

The following theorem is about the weak convergence rate.

\begin{theorem}[Weak convergence rate]\label{WCR}
Let Assumption $\ref{uni-ell}$-Assumption $\ref{lip}$ all hold, $\forall x\in \mathbb{R}^{d_{1}}$, $y\in \mathbb{R}^{d_{2}}$,   $T>0$, $t\in[0,T]$, 
$b(\cdot,\cdot,\cdot)\in C_{p}^{\frac{v}{\alpha_{1}},v,2+\gamma}\cap C_{b}^{1,1+\gamma,2}$, $\delta_{1}(\cdot,\cdot,\cdot)\in C_{p}^{\frac{v}{\alpha_{1}},v,2+\gamma}\cap C_{b}^{1,1+\gamma,2}$, $f(\cdot,\cdot)\in C_{b}^{v,2+\gamma}$, 
$\delta_{2}(\cdot,\cdot)\in C_{b}^{v,2+\gamma}$, and $v\in((\alpha_{1}-\alpha_{2})^{+},\alpha_{1}]$,      $\gamma\in (0,1)$,  $\forall \phi(x)\in C^{2+\gamma}_{b}$,  %for spatial periodic case  we let \textbf{A1, AC, AS, A2, A3, A4, A5} hold, 
we have  
\begin{align*}
	\sup_{t\in [0,T]}|\mathbb{E}\phi(X_{t}^{\varepsilon})-\mathbb{E}\phi(\bar{X_{t}})|\leq  C_{T,x,y}\cdot \left(\varepsilon^{\frac{v}{\alpha_{2}}}+\varepsilon^{1-\frac{\alpha_{1}-v}{\alpha_{2}}}\right),
\end{align*}
where
\begin{align*}
d\bar{X}_{t}=\bar{b}(t,\bar{X}_{t})dt+\bar{\delta}_{1}(t,\bar{X}_{t})dL_{t}^{1},
\end{align*}
the averaged coefficient $\bar{b}(t,x)$ is defined as \eqref{b},  $\bar{\de}_1(t,x)$ is represented as 
\begin{equation}\label{d}
	\bar{\delta}_{1}(t,x)=\int_{\mathbb{R}^{d_{2}}}\de_1(t,x,y)\rho^{x}(dy), %\quad\text{or}\quad	\bar{\delta}_{1}(t,x)=\int_{\mathbb{R}^{d_{2}}/\Lambda}\de_1^{\Lambda}(t,x,y)\mu^{x}(dy),
\end{equation} 
here $\rho^x$ is invariant measures defined as in strong convergence result.
\end{theorem}
\iffalse
\begin{remark}
Exponential ergodicity, gradient estimate are derived by H\"{o}lder regularity and boundedness of functions, see Lemma $\ref{L41}$, Theorem $\ref{T51-1}$ and Theorem $\ref{T61-1}$ for details, while for molification estimates we employ $C_p$ spcaes defined above in Lemma $\ref{L51}$, so we take intersections of $C_p$ and $C_b$ in above theorems.
\end{remark}
\fi
We have the following theorem about exponential contractivity of frozen equation \eqref{2.23} with respect to $L^p$-Wasserstein distance, see Section \ref{exp-con} for details.

\begin{theorem}[Exponential contractivity of frozen equation]\label{was-exp}
Let Assumption $\ref{par-dis}$ holds, for \eqref{2.23}, $\forall y_1,y_2\in \mathbb{R}^{d_{2}}$, $p\geq1$, $t>0$, we have $ C(p)>0$, $\beta>0$ s.t.
\begin{align*}
W_p(\delta_{y_1}P_t,\delta_{y_2}P_t)\leq C(p)e^{-\frac{\beta t}{p}}|y_1-y_2|.
\end{align*}
\end{theorem}

\section{Well-posedness and some moment estimates of $(X_{t}^{\varepsilon},Y_{t}^{\varepsilon})$ }\label{moment}

Recall that $ L_{t}^{i},\ i=1,2,$ denote the isotropic $\alpha$-stable processes associated with $X_{t}^{\varepsilon}$ and $Y_{t}^{\varepsilon}$ respectively, the corresponding Poission random measures are defined by \cite{DA},
$$ N^{i}(t,A)=\sum_{s\leq t}1_{A}(L_{s}^{i}-L_{s-}^{i}),\ \forall A\in \mathcal{B}(\mathbb{R}^{d_{i} }),$$
then compensated Poisson measures will be 
$$ \tilde{N}^{i}(t,A)=N^{i}(t,A)-t\nu_{i}(A),$$
where $\nu_{i}(dz)=\frac{c_{\alpha_{i},d_{i}}}{|z|^{d_{i}+\alpha_{i}}}dz$ is symmetric L\'{e}vy measure, $c_{\alpha_{i},d_{i}}>0$ is constant. By L\'{e}vy-It\^{o} decomposition and symmetry of $\nu_{i}(dz)$, we have
\begin{equation}\label{3.1}
	 L^{i}_{t}=\int_{|z|\leq1}z \tilde{N}^{i}(t,dz)+\int_{|z|>1}zN^{i}(t,dz),
\end{equation}	 
so \eqref{1.1} with initial data  $X^{\varepsilon}_{0}=x\in \mathbb{R}^{d_{1} }$, $ Y^{\varepsilon}_{0}=y\in \mathbb{R}^{d_{2} }$ can be rewritten in the form of Poisson processes
\begin{equation}\label{3.2}
	\left\{
\begin{aligned}
	&dX_{t}^{\varepsilon}=b(t,X_{t}^{\varepsilon},Y_{t}^{\varepsilon})dt+\int_{|z|\leq1}\delta_{1}(t^{-},X_{t^{-}}^{\varepsilon},Y_{t^{-}}^{\varepsilon})z \tilde{N}^{1}(dt,dz)+\int_{|z|>1}\delta_{1}(t^{-},X_{t^{-}}^{\varepsilon},Y_{t^{-}}^{\varepsilon})zN^{1}(dt,dz),\\
	&dY_{t}^{\varepsilon}=\frac{1}{\varepsilon}f(X_{t}^{\varepsilon},Y_{t}^{\varepsilon})dt+\frac{1}{\varepsilon^{\frac{1}{\alpha_{2}}}}\left( \int_{|z|\leq1}\delta_{2}(X_{t^{-}}^{\varepsilon},Y_{t^{-}}^{\varepsilon})z \tilde{N}^{2}(dt,dz)+\int_{|z|>1}\delta_{2}(X_{t^{-}}^{\varepsilon},Y_{t^{-}}^{\varepsilon})zN^{2}(dt,dz)\right).
\end{aligned}
\right.
\end{equation}

\begin{theorem}[Well-posedness of \eqref{1.1}]\label{T31}
Let  Assumption $\ref{uni-ell}$-(1), Assumption $\ref{gro-bou}$, Assumption $\ref{lip}$ hold, $\forall \varepsilon >0$, given any initial data $x\in \mathbb{R}^{d_{1} }$, $y \in \mathbb{R}^{d_{2} }$, there exists unique strong solution $(X_{t}^{\varepsilon},Y_{t}^{\varepsilon})$ to \eqref{1.1}.
\end{theorem}
Under Assumption \ref{gro-bou}, Assumption \ref{lip} on $b$ , $f$, $\delta_{1}$ and $\delta_{2}$,	the well-posedness of \eqref{3.2} can be  established by the same procedures outlined in  \cite[Theorem 6.2.3, Theorem 6.2.9, Theorem 6.2.11]{DA}, which lead to well-posedness of \eqref{1.1}.
\begin{theorem}\label{T32}
Let  Assumption $\ref{uni-ell}$-(1), Assumption $\ref{par-dis}$ and  Assumption $\ref{gro-bou}$ hold.	For any solution $(X_{t}^{\varepsilon},Y_{t}^{\varepsilon})$ to \eqref{1.1}, $\forall m\in [1,\alpha_{1}\wedge\alpha_{2})$,  $t\geq0$, $\exists C_{m}>0$ s.t.,
	\begin{equation}\label{3.4}
		\sup_{\varepsilon\in (0,1)}\sup_{t\geq0} \mathbb{E}|X_{t}^{\varepsilon}|^{m}\leq C_{m}(1+|x|^{m}),
	\end{equation}
	\begin{equation}\label{3.6}
		\sup_{\varepsilon\in (0,1)}\sup_{t\geq0}\mathbb{E} |Y_{t}^{\varepsilon}|^{m}\leq  C_{m}(1+|y|^{m}).
	\end{equation}
\end{theorem}
\begin{proof}
Our proof refers to \cite{LRSX},  \cite[Lemma A.1]{SXX}, especially \cite[Theorem 3.2]{KY}.	We observe that for $X_{t}^{\varepsilon} $,
	\begin{align*}
			X_{t}^{\varepsilon}=&x+\int_{0}^{t}b(s,X_{s}^{\varepsilon},Y_{s}^{\varepsilon})ds+\int_{0}^{t}\delta_{1}(s^{-},X_{s^{-}}^{\varepsilon},Y_{s^{-}}^{\varepsilon})\left( \int_{|z|\leq 1}z \tilde{N}^{1}(ds,dz)+\int_{|z|> 1}zN^{1}(ds,dz)\right),
	\end{align*}
due to the fact that $m<\alpha_{1}<2$, here we do not use It\^{o}'s formula directly, however,  we can deduce that $ |x|^{2\cdot \frac{m}{m}}<(|x|+1)^{2\cdot \frac{m}{m}}<(|x|^{2}+1)^{\frac{m}{2}}$, so we define  %$ |y|^{2\cdot \frac{p}{p}}<(|y|+1)^{2\cdot \frac{p}{p}}<(|y|^{2}+1)^{\frac{p}{2}}$, 
	\begin{equation}
		\begin{split}
			U(t,x)&=e^{-\frac{m}{2}\alpha_{t}}(|x|^{2}+1)^{\frac{m}{2}},
			%\	U(y)=(|y|^{2}+1)^{\frac{p}{2}},
		\end{split}\nonumber
	\end{equation}
	we can see that  $U(t,x)>0$, and %$U(y)>0$, 
	\begin{equation}\label{du}
		\begin{split}
			&|DU(t,x)|=\left| e^{-\frac{m}{2}\alpha_{t}} \frac{mx}{(|x|^{2}+1)^{1-\frac{m}{2}}}\right|\leq C_{m} e^{-\frac{m}{2}\alpha_{t}}|x|^{m-1},\\
			%& |DU(y)|=\left|  \frac{py}{(|y|^{2}+1)^{1-\frac{p}{2}}}\right|\leq C_{p} |y|^{p-1},
				\end{split}
		\end{equation}
	\begin{equation}\label{d2u}
	\begin{split}
			&|D^{2}U(t,x)|=\left| e^{-\frac{m}{2}\alpha_{t}}\left(  \frac{mI_{d_{2}\times d_{2} }}{(|x|^{2}+1)^{1-\frac{m}{2}}}-\frac{m(m-2)x\otimes x}{(|x|^{2}+1)^{2-\frac{m}{2}}}\right) \right|\leq \frac{C_{m}e^{-\frac{m}{2}\alpha_{t}}}{(|x|^{2}+1)^{1-\frac{m}{2}}}\leq C_{m}e^{-\frac{m}{2}\alpha_{t}}.\\
			%&|D^{2}U(y)|=\left|   \frac{pI_{d_{2}\times d_{2} }}{(|y|^{2}+1)^{1-\frac{p}{2}}}-\frac{p(p-2)y\otimes y}{(|y|^{2}+1)^{2-\frac{p}{2}}} \right|\leq \frac{C_{p}}{(|y|^{2}+1)^{1-\frac{p}{2}}}\leq C_{p}.
		\end{split}
	\end{equation}
	
Applying It\^{o}'s formula, and taking expectation on both sides, with the fact that $\mathbb{E}\tilde{N}^{1}(ds,dz)=0$,	\begin{equation}\label{3.9}
	\begin{split}
	&\frac{d\mathbb{E}U(t,X_{t}^{\varepsilon})}{dt}=-\frac{m}{2}\mathbb{E}K_{t}U(t,X_{t}^{\varepsilon})+\mathbb{E}\langle  b(t,X_{t}^{\varepsilon},Y_{t}^{\varepsilon}),DU(t,X_{t}^{\varepsilon})\rangle  \\
	&+\mathbb{E} \int_{|z|\leq 1}\big( U(t,X_{t}^{\varepsilon}+\delta_{1}\cdot z)-U(t,X_{t}^{\varepsilon})-\langle  DU(t,X_{t}^{\varepsilon}), \delta_{1}\cdot z\rangle   \big) \nu_{1}(dz)\\
	&+\mathbb{E} \int_{|z|> 1}\big( U(t,X_{t}^{\varepsilon}+\delta_{1}\cdot z)-U(t,Y_{t}^{\varepsilon}) \big) \nu_{1}(dz)\\
	&\leq \mathbb{E} \langle b(t,X_{t}^{\varepsilon},Y_{t}^{\varepsilon}),DU(t,X_{t}^{\varepsilon})\rangle  +\mathbb{E} \int_{|z|> 1}\big( U(t,X_{t}^{\varepsilon}+\delta_{1}\cdot z)-U(t,X_{t}^{\varepsilon}) \big) \nu_{1}(dz)\\
	&+\mathbb{E} \int_{|z|\leq 1}\big( U(t,X_{t}^{\varepsilon}+\delta_{1}\cdot z)-U(t,X_{t}^{\varepsilon})- \langle DU(t,X_{t}^{\varepsilon}), \delta_{1}\cdot z\rangle   \big) \nu_{1}(dz)=I_{1}+I_{2}+I_{3}.
	\end{split}
	\end{equation}

For $I_{1}$, by  Assumption \ref{par-dis} of $b$ in  \eqref{2.4},
	\begin{equation}\label{3.10}
		\begin{split}
			I_{1}&=\mathbb{E} \langle  b(t,X_{t}^{\varepsilon},Y_{t}^{\varepsilon}),DU(t,X_{t}^{\varepsilon})\rangle \\
			& \leq \mathbb{E} e^{-\frac{m}{2}\alpha_{t}} \frac{\langle  b(t,X_{t}^{\varepsilon},Y_{t}^{\varepsilon})-b(t,0,Y_{t}^{\varepsilon}),mX_{t}^{\varepsilon}\rangle  +\langle  b(t,0,Y_{t}^{\varepsilon}),mX_{t}^{\varepsilon}\rangle }{(|X_{t}^{\varepsilon}|^{2}+1)^{1-\frac{m}{2}}}\\
			&\leq C_{m}\mathbb{E}e^{-\frac{m}{2}\alpha_{t}} \frac{C_{3}|X_{t}^{\varepsilon}|-C_{B}|X_{t}^{\varepsilon}|^{2}}{(|X_{t}^{\varepsilon}|^{2}+1)^{1-\frac{m}{2}}}\\
			&\leq C_{m,C_{B}} \mathbb{E}\left( 1-(|X_{t}^{\varepsilon}|^{2}+1)^{\frac{m}{2}}\right)=C_{m,C_{B}}-C_{m,C_{B}}\mathbb{E}U(t,X_{t}^{\varepsilon}), 
		\end{split}
	\end{equation}
	thus for $I_{2}$, by \eqref{du} and  Assumption \ref{gro-bou} of $ \delta_{1}(t,x,y)$,
	\begin{equation}\label{3.11}
		\begin{split}
			&I_{2}=\mathbb{E} \int_{|z|> 1}\big( U(t,X_{t}^{\varepsilon}+\delta_{1}\cdot z)-U(t,X_{t}^{\varepsilon}) \big) \nu_{1}(dz)\\
			& \leq C_{m}\mathbb{E}e^{-\frac{m}{2}\alpha_{t}}\int_{|z|> 1}\left(|X_{t}^{\varepsilon}|^{p-1}+|z|^{m-1} \right) \nu_{1}(dz) \leq C_{m} +C_{m}\mathbb{E}U(t,X_{t}^{\varepsilon}),
		\end{split}
	\end{equation}
	we derive the last inequality from $1\leq m<\alpha_1$ and H\"{o}lder inequality. 
    
    For $I_3$, by Taylor’s formula, for some $\theta\in(0,1)$,
$$
U(t,X_t^\varepsilon+\delta_1\cdot z)-U(t,X_t^\varepsilon)-\langle DU(t,X_t^\varepsilon),\delta_1\cdot z\rangle
=\frac{1}{2}
(\delta_1\cdot z)^{T}
D^2U(t,X_t^\varepsilon+\theta\delta_1\cdot z)
(\delta_1\cdot z),
$$
then by the boundedness of $\delta_1$ in Assumption \ref{gro-bou} and the estimate \eqref{d2u}, we obtain
$$
\left|
U(t,X_t^\varepsilon+\delta_1\cdot z)-U(t,X_t^\varepsilon)
-\langle DU(t,X_t^\varepsilon),\delta_1\cdot z\rangle
\right|
\le
C_p |z|^2,
$$
therefore, 
\begin{equation}\label{3.10-1}
I_3\le C_m\int_{|z|\le1}|z|^2\nu_1(dz)
\le C_m.
\end{equation}
the last step holds by the definition of L\'evy measure $\int_{|z|\le1}|z|^2\nu_1(dz)<\infty$.
	
Combining \eqref{3.9}-\eqref{3.10-1}, take $C_B$ in \eqref{2.4} large enough, from analysis in \eqref{3.10}, we obtain
	\begin{align*}
			&\frac{d\mathbb{E}U(t,X_{t}^{\varepsilon})}{dt}\leq C_{m}-C_{m,C_B}\mathbb{E}U(t,X_{t}^{\varepsilon}),
	\end{align*}
	by Gronwall's inequality we have 
	\begin{align*}
		\mathbb{E}U(t,X_{t}^{\varepsilon})\leq e^{-C_{m}t}(|x|^{2}+1)^{\frac{m}{2}}+C_{m}\int^{t}_{0}e^{-C_{m}(t-s)}ds,
	\end{align*}
	which means
	$$\mathbb{E}(|X_{t}^{\varepsilon}|^{2}+1)^{\frac{m}{2}}\leq \mathbb{E}e^{-C_{m}t}(|x|^{2}+1)^{\frac{m}{2}}+\mathbb{E}(1-e^{-C_{m}t}),$$
	so we yield, 
	\begin{align*}
			\sup\limits_{\varepsilon\in (0,1)}\sup\limits_{t\geq0} \mathbb{E}\left(|X_{t}^{\varepsilon}|^{m}\right)\leq C_{m}(1+|x|^{m}),
	\end{align*}
	we get \eqref{3.4}. Next we need to estimate $\sup\limits_{\varepsilon\in (0,1)}\sup\limits_{t\geq0}\mathbb{E}\left(|Y_{t}^{\varepsilon}|^{m}\right) $.

	From \eqref{3.2} we also deduce that 
	\begin{align*}
			Y_{t}^{\varepsilon}=&y+\int_{0}^{t}\frac{1}{\varepsilon}f(X_{s}^{\varepsilon},Y_{s}^{\varepsilon})ds+\int_{0}^{t}\frac{1}{\varepsilon^{\frac{1}{\alpha_{2}}}}\delta_{2}(X_{s^{-}}^{\varepsilon},Y_{s^{-}}^{\varepsilon})\left( \int_{|z|\leq \varepsilon^{\frac{1}{\alpha_{2}}}}z \tilde{N}^{2}(ds,dz)+\int_{|z|> \varepsilon^{\frac{1}{\alpha_{2}}}}zN^{2}(ds,dz)\right),
	\end{align*}
	applying It\^{o}'s formula and taking expectation on both sides, with $\mathbb{E}\tilde{N}^{2}(ds,dz)=0$ we derive,
	\begin{align*}
	&\frac{d\mathbb{E}U(Y_{t}^{\varepsilon})}{dt}=\mathbb{E} \frac{1}{\varepsilon}\langle  f(X_{t}^{\varepsilon},Y_{t}^{\varepsilon}),DU(t,Y_{t}^{\varepsilon})\rangle  \\
		&+\mathbb{E} \int_{|z|\leq \varepsilon^{\frac{1}{\alpha_{2}}}}\left( U(Y_{t}^{\varepsilon}+\varepsilon^{-\frac{1}{\alpha_{2}}}\delta_{2}\cdot  z)-U(t,Y_{t}^{\varepsilon})-\langle  DU(Y_{t}^{\varepsilon}), \varepsilon^{-\frac{1}{\alpha_{2}}}\delta_{2}\cdot z\rangle  \right) \nu_{2}(dz)\\
		&+\mathbb{E} \int_{|z|> \varepsilon^{\frac{1}{\alpha_{2}}}}\left( U(Y_{t}^{\varepsilon}+\varepsilon^{-\frac{1}{\alpha_{2}}}\delta_{2}\cdot z)-U(Y_{t}^{\varepsilon}) \right) \nu_{2}(dz)\\
		&\leq \mathbb{E} \frac{1}{\varepsilon}\langle  f(X_{t}^{\varepsilon},Y_{t}^{\varepsilon}),DU(Y_{t}^{\varepsilon})\rangle  +\mathbb{E} \int_{|z|> \varepsilon^{\frac{1}{\alpha_{2}}}}\left( U(Y_{t}^{\varepsilon}+\varepsilon^{-\frac{1}{\alpha_{2}}}\delta_{2}\cdot z)-U(Y_{t}^{\varepsilon}) \right) \nu_{2}(dz)\\
		&+\mathbb{E} \int_{|z|\leq \varepsilon^{\frac{1}{\alpha_{2}}}}\left( U(Y_{t}^{\varepsilon}+\varepsilon^{-\frac{1}{\alpha_{2}}}\delta_{2}\cdot z)-U(Y_{t}^{\varepsilon})-\langle DU(Y_{t}^{\varepsilon}), \varepsilon^{-\frac{1}{\alpha_{2}}}\delta_{2}\cdot z \rangle  \right) \nu_{2}(dz)=I_{1}+I_{2}+I_{3},
	\end{align*}
from Assumption \ref{uni-ell}-(1) of $\de_2(x,y)$, the proof of \eqref{3.6} is quite similar to  \cite[Lemma A.1]{SXX} and \cite[Theorem 3.2]{KY},	for $I_{1}$, by  \eqref{2.2} in Assumption \ref{par-dis} , \eqref{2.4}, 
\begin{align*}
I_{1}&=\mathbb{E} \frac{1}{\varepsilon}\langle  f(X_{t}^{\varepsilon},Y_{t}^{\varepsilon}),DU(Y_{t}^{\varepsilon})\rangle \\
& \leq\frac{1}{\varepsilon} \mathbb{E}  \frac{\langle  f(X_{t}^{\varepsilon},Y_{t}^{\varepsilon})-f(X_{t}^{\varepsilon},0),mY_{t}^{\varepsilon}\rangle  +\langle  f(X_{t}^{\varepsilon},0),mY_{t}^{\varepsilon}\rangle }{(|Y_{t}^{\varepsilon}|^{2}+1)^{1-\frac{m}{2}}}\\
&\leq  \frac{C_{m}}{\varepsilon}\mathbb{E} \frac{C_{3}|Y_{t}^{\varepsilon}|-C_{F}|Y_{t}^{\varepsilon}|^{2}}{(|Y_{t}^{\varepsilon}|^{2}+1)^{1-\frac{m}{2}}}\\
&\leq \frac{C_{m,C_{F}}}{\varepsilon}\mathbb{E}\left( 1-(|Y_{t}^{\varepsilon}|^{2}+1)^{\frac{m}{2}}\right)=\frac{C_{m,C_{F}}}{\varepsilon}-\frac{C_{m,C_F}\mathbb{E}U(t,Y_{t}^{\varepsilon})}{\varepsilon},  
\end{align*}
	in addition, taking $y=\varepsilon^{-\frac{1}{\alpha_{2}}}z$,  we obtain 
\begin{equation*}
		\nu_{2}(dz)=\frac{c}{|z|^{d_2+\alpha_{2}}}dz=\frac{c}{|\varepsilon^{\frac{1}{\alpha_{2}}}y|^{d_2+\alpha_{2}}}(\varepsilon^{\frac{1}{\alpha_{2}}})^{d_2}dy=\frac{1}{\varepsilon}\frac{c}{|y|^{d_2+\alpha_{2}}}dy=\frac{1}{\varepsilon}\nu_{2}(dy),		
\end{equation*}	
thus for $I_{2}$, similar to \eqref{3.11}, and Assumption \ref{uni-ell}-(1) of $ \delta_{2}(x,y)$,
\begin{align*}
		&I_{2}=\frac{1}{\varepsilon}\mathbb{E} \int_{|y|> 1}\big( U(Y_{t}^{\varepsilon}+\de_2\cdot y)-U(Y_{t}^{\varepsilon}) \big) \nu_{2}(dy) \\
		&\leq  \frac{C_{m}}{\varepsilon}\mathbb{E}\int_{|y|> 1}\left(|Y_{t}^{\varepsilon}|^{m-1}+|y|^{m-1} \right) \nu_{2}(dy) \leq\frac{C_{m}}{\varepsilon}+\frac{C_{m}\mathbb{E}U(Y_{t}^{\varepsilon})}{\varepsilon},  
\end{align*}
then 
\begin{equation*}
I_{3}=\frac{1}{\varepsilon}\mathbb{E} \int_{|y|\leq 1}\big( U(Y_{t}^{\varepsilon}+\de_2\cdot y)-U(Y_{t}^{\varepsilon})-\langle  DU(Y_{t}^{\varepsilon}),\de_2\cdot y\rangle  \big) \nu_{2}(dy)\leq  \frac{C_{m}}{\varepsilon},  
\end{equation*}
	combining above estimates, take $C_F$ in \eqref{2.4} large enough, we derive
\begin{equation*}
\frac{d\mathbb{E}U(Y_{t}^{\varepsilon})}{dt}\leq\frac{C_{m}}{\varepsilon}-\frac{C_{m,C_F}\mathbb{E}U(Y_{t}^{\varepsilon})}{\varepsilon},
\end{equation*}
so that by Gronwall's inequality we have 
\begin{align*}
	\mathbb{E}U(Y_{t}^{\varepsilon})\leq e^{-C_{m}\frac{t}{\varepsilon}}(|y|^{2}+1)^{\frac{m}{2}}+\frac{C_{m}}{\varepsilon} \int^{t}_{0}e^{-\frac{C_{m}}{\varepsilon}(t-s)}ds,
\end{align*}
which means
$$\mathbb{E}(|Y_{t}^{\varepsilon}|^{2}+1)^{\frac{m}{2}}\leq \mathbb{E}e^{-C_{m}\frac{t}{\varepsilon}}(|y|^{2}+1)^{\frac{m}{2}}+\mathbb{E}(1-e^{-C_{m}\frac{t}{\varepsilon}}),$$
so that,
\begin{equation*}
	\sup\limits_{\varepsilon\in (0,1)}\sup\limits_{t\geq0} \mathbb{E}\left(|Y_{t}^{\varepsilon}|^{m}\right)\leq C_{m}(1+|y|^{m}),
\end{equation*}
proof is complete.
\end{proof}

\section{The frozen equation for \eqref{1.1}}\label{inv}
In this section, we investigate the invariant measure and exponential ergodicity of the frozen equation associated with $Y^\varepsilon_t$. For any fixed $x\in\mathbb{R}^{d_1}$, consider
\begin{equation}\label{4.1}
dY_t=f(x,Y_t),dt+\delta_2(x,Y_t),dL_t^2,
\qquad
Y_0=y\in\mathbb R^{d_2},
\end{equation}
whose infinitesimal generator is given by $\mathcal{L}_2(x,y)$ in \eqref{l2}.
The multiplicative Lévy coefficient $\delta_2(x,y)$ introduces substantial difficulties compared with the additive case. Indeed, the jump mechanism depends on the current state $y$, and consequently the corresponding nonlocal operator is no longer translation invariant. 

As a result, many standard techniques based on homogeneous stable kernels and explicit Fourier representations are no longer available. In particular, establishing exponential ergodicity of the frozen equation becomes significantly more involved.
To overcome this difficulty, we first rewrite the jump part of the generator in a form suitable for coupling analysis. The key observation is that the multiplicative coefficient $\delta_2(y)$ induces a state-dependent transformation of the underlying Lévy measure. This naturally leads to a family of state-dependent spherical measures, which allows us to represent the generator as a stable-like operator with variable jump kernel.
Throughout this section, $x\in\mathbb{R}^{d_1}$ is fixed. For simplicity of notation, we write $\delta_2(x,y)$ as $\delta_2(y)$. Recall that the Lévy measure of the isotropic $\alpha_2$-stable process admits the spherical decomposition
$$
\nu_2(dz)=
\bar{\nu}_2(d\hat z)\frac{dr}{r^{1+\alpha_2}},
\qquad
z=r\hat z,
$$
where $\hat z=z/|z|\in S^{d_2-1}$, $r=|z|>0$, and $\bar\nu_2$ is a symmetric finite Borel measure on the unit sphere.
To derive the transformed spherical measure, we introduce the change of variables
$$
\omega=\delta_2(y)z,
$$
the corresponding transformation on the unit sphere generates a nonlinear immersion whose Jacobian determinant is computed in Appendix \ref{A}. This geometric identity yields the change-of-variable formula required for the transformed Lévy measure and leads to the state-dependent spherical measure $\pi(y)$ introduced below.
Using this representation, we shall reformulate the generator $\mathcal{L}_2$ as a stable-like operator with state-dependent jump kernel. This formulation enables us to employ coupling techniques for nonlocal operators and establish exponential contractivity of the frozen semigroup in the $L^p$-Wasserstein distance, which in turn yields exponential ergodicity of the frozen equation.

Denote $|z|=r$, $\hat{z}=\frac{z}{|z|}\in\s^{d_{2}-1}$, we have
\begin{align*}
	\nu_{2}(M)=\int_{0}^{\infty}\int_{\s^{d_{2}-1}}I_{M}(r\hat{z})\bar{\nu}_{2}(d\hat{z})\frac{dr}{r^{1+\al_{2}}},
\end{align*}
then
\begin{align*}
	\nu_{2}(\de^{-1}_2M)=\int_{0}^{\infty}\int_{\s^{d_{2}-1}}I_{M}(r\de_2\hat{z})\bar{\nu}_{2}(d\hat{z})\frac{dr}{r^{1+\al_{2}}},
\end{align*}
let $\om=\de_2 \cdot z =r\de_2 \cdot \hat{z}$,  $\hat{\om}=\frac{\om}{|\om|}=\frac{\de_2 \cdot \hat{z}}{|\de_2 \cdot \hat{z}|}$, $s=r|\de_2 \cdot \hat{z}|=|\om|$, so that $ds=|\de_2 \cdot \hat{z}|dr$, and we have nonlinear immersion  $F(\hat{\om})=\hat{z}=\frac{\delta^{-1}_2(y) \hat{\omega}}{|\delta^{-1}_2(y) \hat{\omega}|}$, $1=|\hat{z}|=|\de_2^{-1}\hat{\om}|\cdot|\de_2 \hat{z}|$, we thus have 
\begin{align*}
	\nu_{2}(\de^{-1}_2M)
	&=\int_{0}^{\infty}\int_{\s^{d_{2}-1}}I_{M}\left(s\frac{\de_2\hat{z}}{|\de_2\hat{z}|}\right)|\de_2\hat{z}|^{\al_{2}}\bar{\nu}_{2}(d\hat{z})\frac{ds}{s^{1+\al_{2}}}\\
	&=\int_{0}^{\infty}\int_{\s^{d_{2}-1}}I_{M}(s\hat{\om})\left| \de_2\frac{\delta^{-1}_2(y) \hat{\omega}}{|\delta^{-1}_2(y) \hat{\omega}|}\right| ^{\al_{2}}|J_F(\hat{\om})|\bar{\nu}_{2}(d\hat{\om})\frac{ds}{s^{1+\al_{2}}}\\
	&=\int_{0}^{\infty}\int_{\s^{d_{2}-1}}I_{M}(s\hat{\om})\frac{1}{|\delta^{-1}_2(y) \hat{\omega}|^{\al_{2}}} |J_F(\hat{\om})|\bar{\nu}_{2}(d\hat{\om})\frac{ds}{s^{1+\al_{2}}},
\end{align*}
where we used the fact that $|\hat{\om}|=1$ in the last step, the Jacobian determinant $J_F(\hat{\om})$ with respect to $F(\hat{\om})$ is actually determinant of tangent map $dF(\hat{\om})$, especially from Lemma \ref{jac} in Appendix we have,
\begin{align*}
	|J_F(\hat{\om})|=|det(dF(\hat{\om}))|=\dfrac{|det(\de_2^{-1})|}{|\de_2^{-1}\hat{\om}|^{d_{2}}}=\big(|det(\de_2)|\cdot|\de_2^{-1}\hat{\om}|^{d_{2}} \big)^{-1},
\end{align*}
here $det(\de_2^{-1})$, $det(\de_2)$ are the determinants of $\de_2^{-1}$ and $\de_2$ respectively, hence
\begin{align*}
	\nu_{2}(\de^{-1}_2M)=\int_{0}^{\infty}\int_{\s^{d_{2}-1}}I_{M}(s\hat{\om})\big(|det(\de_2(y))|\cdot|\de_2^{-1}(y)\hat{\om}|^{\al_{2}+d_{2}}\big)^{-1} \bar{\nu}_{2}(d\hat{\om})\frac{ds}{s^{1+\al_{2}}},
\end{align*}
then we define the $y$-dependent spherical measure $\pi(y)$ on $\s^{d_{2}-1}$ as
\begin{align}\label{pi-H}
	\pi(y)(d\hat{\om})=H(y,\hat{\om})\bar{\nu}_{2}(d\hat{\om})=\big(|det(\de_2(y))|\cdot|\de_2^{-1}(y)\hat{\om}|^{\al_{2}+d_{2}} \big)^{-1}\bar{\nu}_{2}(d\hat{\om}),
\end{align}
from  Assumption \ref{uni-ell} of $\de_2(y)$ and $\bar{\nu}_2$ we deduce that $\pi(y)$ also satisfies Assumption \ref{uni-ell} as that of $\bar{\nu}_{2}$. 
%obviously differentiability of $\pi(y)$ with respect to $y$ is determined by that of $H(y,\hat{\om})$. since we additionally assume that $\delta_{2}(\cdot)\in C_{b}^{2+\ga}$, the uniform boundedness for the partial derivatives of $\pi(y)$ up to order 2 are obtained in Lemma \ref{der},  together with $f(x,\cdot)\in C^{2+\ga}_{b}(\R^{d_{2}})$, from \cite[743 Theorem 3.1]{VK}, we can derive the existence of transition probability density $p(t,y,Y)$, and $p(t,y,Y)$ is twice differentiable with respect to $y$.

Nevertheless,  for the infinitesimal generator $\mathcal{L}_2$ defined in \eqref{l2}, we have an equivalent variant form by above analysis
\begin{align}\label{l2-2}
	\mathcal{L}_{2}(x,y)u(x,y)
	&=P.V.\int_{\mathbb{R}^{d_{2}}}\Big(u(x,y+z)-u(x,y)-\langle z,\nabla_{y}u(x,y)\rangle I_{|z|\leq1}\Big)H(\hat{z},y)\nu_{2}(dz)\nonumber\\
	&+f(t,x,y)\nabla_{y}u(x,y),
\end{align}
where $H(y,\hat{z})=\big(|det(\de_2(y))|\cdot|\de_2^{-1}(y)\hat{z}|^{\al_{2}+d_{2}} \big)^{-1}.$

Next, under  Assumption \ref{uni-ell},  we apply coupling method for non-local operators to obtain exponential contractivity with respect to $L^p$-Wasserstein distance,  then the exponential ergodicity follows.

%In the other way we consider the celebrated Doeblin's results of exponential ergodicity, so we impose compactness and spatial periodicty on coefficients to prove exponential ergodicity of projected process $Y_t^{\Lambda,x,y}$ on $\mathbb{R}^{d_{2} }/\Lambda$, which is the image of frozen process $Y_t^{x,y}$, and come back to periodic settings in $\mathbb{R}^{d_{2} }$ by projection, see Section \ref{inv-per}.

\subsection{Exponential decay with respect to $L^p$-Wasserstein distance}\label{exp-con}
The idea is based on \cite{DW,JW2}. We introduce a Markov coupling operator for $\mathcal{L}_2$ in terms of \eqref{l2-2}, $\forall y_1,y_2,z\in\R^{d_2}$, let
$$\varphi_{y_1,y_2}(z)=
\begin{cases}
	z-\frac{2\langle y_1-y_2, z\rangle}{|y_1-y_2|^2}(y_1-y_2), & y_1\neq y_2;\\
	-z, &y_1=y_2.\\
\end{cases}
$$
then $\varphi_{y_1,y_2}(z):\R^{d_2}\to\R^{d_2}$ enjoys the following properties:
\begin{itemize}
	\item[(A1)] $\varphi_{y_1,y_2}(z)=\varphi_{y_2,y_1}(z)$ and $\varphi^2_{y_1,y_2}(z)=z$, i.e.\
	$\varphi_{y_1,y_2}^{-1}(z)=\varphi_{y_1,y_2}(z)$;
	\item[(A2)] $|\varphi_{y_1,y_2}(z)|=|z|$;
	\item[(A3)] $(z-\varphi_{y_1,y_2}(z))\parallel(y_1-y_2)$ and $(z+\varphi_{y_1,y_2}(z))\perp (y_1-y_2).$
\end{itemize}
For convenient we denote $\varphi_{y_1,y_2}(z)$ by $\varphi(z)$.  $\forall y_1,y_2,z\in\R^{d_2}$, let
\begin{align*}
\widetilde{H}(y_1,y_2,z)=H(y_1,z)\wedge H(y_2,z)\wedge H(y_1,\varphi(z))\wedge H(y_2,\varphi(z)),
\end{align*}
and $\forall F\in C^2_b(\R^{2d_2})$,
$$\nabla_{y_1}F(y_1,y_2):=\left(\frac{\partial F(y_1,y_2)}{\partial y_1^i}\right)_{1\le i\le d_2},\quad \nabla_{y_2}F(y_1,y_2):=
 \left(\frac{\partial F(y_1,y_2)}{\partial y_2^i}\right)_{1\le i\le d_2},$$
when $|y_1-y_2|\leq L_0$, for $a\in (0, \frac{1}{2})$ we can define for $ z\in\R^{d_2}\setminus\{0\}$
\begin{align*}
&\widetilde{\mathcal{L}}_1F(y_1,y_2)
=\frac{1}{2}\bigg[\int_{\left\{|z|\le a|y_1-y_2| \right\}} \Big(F(y_1+z,y_2+\varphi(z))-F(y_1,y_2)-\langle \nabla_{y_1}F(y_1,y_2),z\rangle I_{\{|z|\le1\}}\\
&-\langle \nabla_{y_2}F(y_1,y_2),\varphi(z)\rangle  I_{\{|z|\le1\}}\Big)\frac{\widetilde{H}(y_1,y_2,z)}{|z|^{d_2+\alpha_2}}\,dz\\
&+\int_{\left\{|z|\le  a|y_1-y_2|\right\}}\!\! \Big(F(y_1+\varphi(z),y_2+z)-F(y_1,y_2)-\langle \nabla_{y_2}F(y_1,y_2),z\rangle I_{\{|z|\le1\}}\\
&-\langle \nabla_{y_1}F(y_1,y_2), \varphi(z)\rangle  I_{\{|z|\le1\}}\Big)\frac{\widetilde{H}(y_1,y_2,z)}{|z|^{d_2+\alpha_2}}\,dz\bigg]\\
&+\int_{\left\{|z|\le  a|y_1-y_2|\right\}}\Big(F(y_1+z,y_2+z)-F(y_1,y_2)-\langle \nabla_{y_1}F(y_1,y_2), z\rangle  I_{\{|z|\le1\}}-\langle \nabla_{y_2}F(y_1,y_2), z\rangle  I_{\{|z|\le1\}}\Big)\\
&\times\frac{H(y_1,z)\wedge H(y_2,z)-\widetilde{H}(y_1,y_2,z)}{|z|^{d_2+\alpha_2}}dz\\
&+\int_{\left\{|z|> a|y_1-y_2|\right\}}\Big(F(y_1+z,y_2+z)-F(y_1,y_2)-\langle \nabla_{y_1}F(y_1,y_2), z\rangle I_{\{|z|\le1\}}-\langle \nabla_{y_2}F(y_1,y_2), z\rangle  I_{\{|z|\le1\}}\Big)\\
&\times\frac{H(y_1,z)\wedge H(y_2,z)}{|z|^{d_2+\alpha_2}}dz+\langle \nabla_{y_1}F(y_1,y_2),f(y_1)\rangle +\langle \nabla_{y_2}F(y_1,y_2),f(y_2)\rangle,
\end{align*}
as for $|y_1-y_2|> L_0$,  $ z\in\R^{d_2}\setminus\{0\}$
\begin{align*}
\widetilde{\mathcal{L}}_2F(y_1,y_2)=&\int_{\R^{d_2}}\Big(F(y_1+z,y_2+z)-F(y_1,y_2)-\langle \nabla_{y_1}F(y_1,y_2), z\rangle I_{\{|z|\le1\}}-\langle \nabla_{y_2}F(y_1,y_2), z\rangle  I_{\{|z|\le1\}}\Big)\\
&\times\frac{H(y_1,z)\wedge H(y_2,z)}{|z|^{d_2+\alpha_2}}dz+\langle \nabla_{y_1}F(y_1,y_2),f(y_1)\rangle +\langle \nabla_{y_2}F(y_1,y_2),f(y_2)\rangle.
\end{align*}

From above definition we have the following lemma.
\begin{lemma}\label{cou-oper}
$\forall z\in \R^{d_2}\setminus \{0\}$,  $\widetilde{\mathcal{L}}=\widetilde{\mathcal{L}}_1+\widetilde{\mathcal{L}}_2$ is coupling operator of $\mathcal{L}_2$, i.e., $\forall F(y_1)\in C^2_b(\R^{d_2})$
\begin{align}\label{cou}
\widetilde{\mathcal{L}}F(y_1,y_2)=\mathcal{L}_2F(y_1)+\mathcal{L}_2F(y_2).
\end{align}
where $F(y_1,y_2)=F(y_1)+F(y_2)$.
\end{lemma}
\begin{proof}
	The proof refers to \cite[Lemma 2.2]{DW}. Here
$F(y_1,y_2)$ can be is regarded as a bivariate function on $\R^{2d_2}$, it suffices to verify that $\widetilde{\mathcal{L}}F(y_1)=\mathcal{L}_2F(y_1)$.
For the case $|y_1-y_2|\leq L_0$, we have 
\begin{align*}
	&\widetilde{\mathcal{L}}_1F(y_1)
	=\frac{1}{2}\bigg[\int_{\left\{|z|\le a|y_1-y_2| \right\}} \Big(F(y_1+z)-F(y_1)-\langle \nabla_{y_1}F(y_1),z\rangle I_{\{|z|\le1\}}\Big)\frac{\widetilde{H}(y_1,y_2,z)}{|z|^{d_2+\alpha_2}}\,dz\\
	&+\int_{\left\{|z|\le  a|y_1-y_2|\right\}}\!\! \Big(F(y_1+\varphi(z))-F(y_1)-\langle \nabla_{y_1}F(y_1), \varphi(z)\rangle  I_{\{|z|\le1\}}\Big)\frac{\widetilde{H}(y_1,y_2,z)}{|z|^{d_2+\alpha_2}}\,dz\bigg]\\
	&+\int_{\left\{|z|\le  a|y_1-y_2|\right\}}\Big(F(y_1+z)-F(y_1)-\langle \nabla_{y_1}F(y_1), z\rangle  I_{\{|z|\le1\}}\Big)\frac{H(y_1,z)-\widetilde{H}(y_1,y_2,z)}{|z|^{d_2+\alpha_2}}dz\\
	&+\int_{\left\{|z|> a|y_1-y_2|\right\}}\Big(F(y_1+z)-F(y_1)-\langle \nabla_{y_1}F(y_1), z\rangle I_{\{|z|\le1\}}\Big)\frac{H(y_1,z)}{|z|^{d_2+\alpha_2}}dz+\langle \nabla_{y_1}F(y_1),f(y_1)\rangle,
\end{align*}
according to the definition of $\varphi$, and $\frac{\widetilde{H}(y_1,y_2,z)}{|z|^{d_2+\alpha_2}}dz$ is  rotationally invariant under the transformation $z\rightarrow\varphi$,  we have 
\begin{align*}
\widetilde{\mathcal{L}}_1F(y_1)
&=\int_{\left\{|z|\le  a|y_1-y_2|\right\}}\Big(F(y_1+z)-F(y_1)-\langle \nabla_{y_1}F(y_1), z\rangle  I_{\{|z|\le1\}}\Big)\frac{H(y_1,z)}{|z|^{d_2+\alpha_2}}dz\\
&+\int_{\left\{|z|> a|y_1-y_2|\right\}}\Big(F(y_1+z)-F(y_1)-\langle \nabla_{y_1}F(y_1), z\rangle I_{\{|z|\le1\}}\Big)\frac{H(y_1,z)}{|z|^{d_2+\alpha_2}}dz\\
&+(\nabla_{y_1}F(y_1),f(y_1))=\mathcal{L}_2F(y_1)\cdot I_{\{|y_1-y_2|\leq L_0\}},
\end{align*}
when $|y_1-y_2|> L_0$, it is easy to verify that $\widetilde{\mathcal{L}}_2F(y_1)=\mathcal{L}_2F(y_1)\cdot I_{\{|y_1-y_2|> L_0\}}$, the proof is complete.
\end{proof}

Then, we need to construct a coupling process associated with $\widetilde{\mathcal{L}}$, this method is based on \cite[Section 2.2]{JW2}.  $\forall y_1,y_2\in\R^{d_2}$, $ z\in \R^{d_2}\setminus \{0\}$,  $A\in \mathscr{B}(\R^{2d})$, let 
\begin{align*}
{\mu}(y_1,y_2,A)&=\frac{1}{2}\int_{\left\{(z,\varphi(z))\in A,|z|\le a|y_1-y_2|,|y_1-y_2|\leq L_0 \right\}}\frac{\widetilde{H}(y_1,y_2,z)}{|z|^{d_2+\alpha_2}}dz\\
&+\frac{1}{2}\int_{\left\{(\varphi(z),z)\in
A,|z|\le a|y_1-y_2|,|y_1-y_2|\leq L_0\right\}}\frac{\widetilde{H}(y_1,y_2,z)}{|z|^{d_2+\alpha_2}}dz\\
&+\int_{\left\{(z,z)\in A,|z|\leq a|y_1-y_2|,|y_1-y_2|\leq L_0\right\}}
\frac{H(y_1,z)\wedge H(y_2,z)-\widetilde{H}(y_1,y_2,z)}{|z|^{d_2+\alpha_2}}dz\\
&+\int_{\left\{(z,z)\in A,|z|> a|y_1-y_2|,|y_1-y_2|\leq L_0\right\}\cup \left\{(z,z)\in A,|y_1-y_2|> L_0\right\}}\frac{H(y_1,z)\wedge H(y_2,z)}{|z|^{d_2+\alpha_2}}dz,
\end{align*}
then $\forall F\in C^2_b(\R^{2d_2})$, 
\begin{align*}
\widetilde{\mathcal{L}}F(y_1,y_2)=&\int_{\R^{2d_2}}\Big(F\big((y_1,y_2)+(u_1,u_2)\big)-F(y_1,y_2)-\big\langle (\nabla_{y_1}F(y_1,y_2),\nabla_{y_2}F(y_1,y_2)), (u_1,u_2)\big\rangle\\
&\times I_{\{|u_1|\le1, |u_2|\le1\}}\Big)\mu(y_1,y_2,u_1,u_2)+\langle\nabla_{y_1}F(y_1,y_2),f(y_1)\rangle+\langle\nabla_{y_2}F(y_1,y_2),f(y_2)\rangle,
\end{align*}
and $\forall h\in C_b(\R^{2d})$, $u=(z,\varphi(z))$, by (A2) property of $\varphi(z)$,
\begin{align*}
\int_{\R^{2d_2}} h(u)\frac{|u|^2}{1+|u|^2}\,{\mu}(y_1,y_2,du)&=\int_{\left\{|z|\le a|y_1-y_2|, |y_1-y_2|\leq L_0\right\}}h(z,\varphi(z))\frac{|z|^2}{1+2|z|^2}
\frac{\widetilde{H}(y_1,y_2,z)}{|z|^{d_2+\alpha_2}}dz\\
&+\int_{\left\{|z|\le a|y_1-y_2|, |y_1-y_2|\leq L_0\right\}}h(\varphi(z),z)\frac{|z|^2}{1+2|z|^2}
\frac{\widetilde{H}(y_1,y_2,z)}{|z|^{d_2+\alpha_2}}dz\\
&+2\int_{\left\{|z|\leq a|y_1-y_2|,|y_1-y_2|\leq L_0\right\}}h(z,z)
\frac{H(y_1,z)\wedge H(y_2,z)-\widetilde{H}(y_1,y_2,z)}{|z|^{d_2+\alpha_2}}dz\\
&+2\int_{\left\{|z|> a|y_1-y_2|,|y_1-y_2|\leq L_0\right\}\cup \left\{|y_1-y_2|> L_0\right\}}h(z,z)\frac{H(y_1,z)\wedge H(y_2,z)}{|z|^{d_2+\alpha_2}}dz,
\end{align*}
which means that $(y_1,y_2)\rightarrow\int_{\R^{2d_2}} h(u)\frac{|u|^2}{1+|u|^2}\,{\mu}(y_1,y_2,du)$ is continuous on $\R^{2d_2}$, we also have the fact that $f(y)$ is continuous on $\R^{d_2}$.  

\iffalse
Recall the martingale problem for $(\mathcal{L}_2, C^2_b)$ has the following description, a probability measure $\mathbb{P}^y$ on the Skorokhod space $ D([0,\infty);\R^{d_2})$ is solution to the martingale problem for $(\mathcal{L}_2, C^2_b)$, if $\forall f\in C_b^2(\R^{d_2}), y\in\R^{d_2}$, $\mathbb{P}(Y_0=y)=1$,
\begin{align*}
	f(Y_t)-f(y)-\int_0^{t} \mathcal{L}_2f(Y_s)ds,\ t\ge 0,
\end{align*}
is  $\mathbb{P}^y$ martingale. \fi

Then similar to analysis in \cite{DW} and \cite[Section 2.2]{JW2}, %by\cite[Theorem 2.2]{DWS},
there exists a martingale solution with respect to $\widetilde{\mathcal{L}}$, i.e., there exist a probability space
$(\widetilde{\Omega}, \widetilde{\mathscr{F}},
(\widetilde{\mathscr{F}}_t)_{t\ge0}, \widetilde{\mathbb{P}})$ and
$\R^{2d_2}$-valued process $(\widetilde{Y}_t)_{t\ge0}$ which is
$(\widetilde{\mathscr{F}}_t)_{t\ge0}$-progressively measurable, and $\forall f\in C_b^2(\R^{2d_2})$,
\begin{align*}
f(\widetilde{Y}_t)-f(\widetilde{Y}_0)-\int_0^{t\wedge \tau} \widetilde{\mathcal{L}}f(\widetilde{Y}_s)ds,\ t\ge 0,
\end{align*}
is $(\widetilde{\mathscr{F}}_t)_{t\ge0}$-local martingale, where
$\tau$ is the explosion time of $\widetilde{Y}_t$, 
$$\tau=\liminf_{n\to\infty}\{t\ge0: |\widetilde{Y}_t|\ge n\},$$
here $\widetilde{Y}_t=(Y'_t,Y''_t)$, then $Y'_t$ and $Y''_t$ are two stochastic processes
on $\R^{d_2}$. Since $\widetilde{\mathcal{L}}$ is  coupling operator of $\mathcal{L}_2$,
the infinitesimal generators of processes $Y'_t$ and
$Y''_t$ are $\mathcal{L}_2$, then both processes are
solutions to the martingale problem of $\mathcal{L}_2$. From Assumption \ref{gro-bou} and Assumption \ref{lip} we can derive the well-posedness of the pathwise unique strong solution to \eqref{4.1}, so the well-posedness of weak solution follows, and the well-posedness of  martingale problem to $(\mathcal{L}_2, C^2_b)$. We conclude  that $Y'_t$ and $Y''_t$ are well-posed, and $\tau=\infty$, $\widetilde{\mathcal{L}}$ generates a non-explosive process $\widetilde{Y}_t$.
Let $T$ be the coupling time of $Y'_t$ and $Y''_t$, i.e.
$$T=\inf\{t\ge0: Y'_t=Y''_t\},$$
so $T$ is an $(\widetilde{\mathscr{F}}_t)_{t\geq 0}$-stopping time. Define
a new process $(Z'_t)_{t\ge0}$ as follows
$$Z_t'=
\begin{cases}
	Y''_t, & t< T;\\
	Y'_t, & t\ge T,\\
\end{cases}
$$
from \cite[Section 2.2]{JW2}, we conclude that
$(Y'_t,Z'_t)_{t\geq0}$ is also non-explosive coupling process of
$Y_t$ s.t. $Z_t'=Y'_t$ $\forall t\ge T$,
the generator of $(Y_t',Z'_t)_{t\ge0}$ before the coupling time $T$
is  the coupling operator $\widetilde{\mathcal{L}}$ defined in \eqref{cou}, and we know that
$\forall y_1,y_2\in\R^{d_2}$, $f\in C_b(\R^{d_2})$,
$$P_t f(y_1)={\mathbb{E}^{y_1}f(Y_t')}=\widetilde{{\mathbb{E}}}^{(y_1,y_2)}f(Y_t'),\ \ P_t f(y_2)={\mathbb{E}^{y_2}f(Z_t')}=\widetilde{{\mathbb{E}}}^{(y_1,y_2)}f(Z_t').$$

Then we discuss the exponential contractivity of $Y_t^{x,y}$ with respect to $L^p$-Wasserstein distance.
$\forall r>0$, we define
$$\psi(r)=
\begin{cases}
	1-e^{-c_1r}, & \quad r\in(0,2L_0];\\
	Ae^{c_2(r-2L_0)}+B(r-2L_0)^2+ \bigl(1-e^{-2c_1L_0}-A
	\bigr), & \quad r\in [2L_0,\infty),\\
\end{cases}$$
$\psi(2L_0)=1-e^{-2c_1L_0}$, where
$A= \frac{c_1}{c_2}\mathrm{e}^{-2L_0c_1}>0$, $B= -\frac{(c_1+c_2)c_1}{2}\mathrm{e}^{-2L_0c_1}<0,$
$c_1,c_2>0$, $c_2\ge20c_1$ so that
$\log\frac{2(c_1+c_2)}{c_2}\le\log2.1,$ obviously  $\psi\in
C^2\left([0,\infty)\right)$. 

\begin{lemma}[Lyaponov type estiamte]\label{inv-gen-L} $\forall y_1, y_2\in\mathbb{R}^d$,  $ z\in\R^{d_2}\setminus\{0\}$, $\exists \beta>0$,
\begin{align*}	\widetilde{\mathcal{L}}\psi(|y_1-y_2|)\leq-\beta\psi(|y_1-y_2|).
\end{align*}
\end{lemma}
\begin{proof}
This Lyapunov type estimate is based on \cite[Proposition 3.1]{JW2}. We start with the case when $|y_1-y_2|\le L_0$. $\forall y_1,y_2,\in\mathbb{R}^{d_2}, z\in\mathbb{R}^{d_2}\setminus \{0\}$, by (A3) property of $\varphi$, we have
$
\langle y_1-y_2,z+\varphi(z)\rangle=0,
$
then
\begin{align*}
\langle\nabla_{y_1}\psi\bigl(|y_1-y_2|\bigr), z+\varphi(z)\rangle=0 \quad\mbox{and}\quad \langle\nabla_{y_2}\psi\bigl(|y_1-y_2|\bigr), z+
	\varphi_{x,y}(z)\rangle=0.
\end{align*}
additionally, by Lemma \ref{cou-oper},
\begin{align*}
\widetilde{\mathcal{L}}\psi(|y_1-y_2|)
&=\frac{1}{2} \bigg[\int_{ \{|z|\le a{|y_1-y_2|} \}} \Big(\psi (|y_1-y_2+(z-\varphi(z))| )+\psi( |y_1-y_2-(z-
	\varphi(z))|)\\
	&-2\psi(|y_1-y_2|) \Big)\frac{\widetilde{H}(y_1,y_2,z)}{|z|^{d_2+\alpha_2}}dz \bigg]+\psi'(|y_1-y_2|) \frac{\langle  f(y_1)-f(y_2),y_1-y_2\rangle }{|y_1-y_2|},
\end{align*}
since $\psi\in C^3([0,2L_0))$, $\psi'>0$,
$\psi''<0$, $\psi'''>0$ on $[0,2L_0)$, then for $0\le
\delta< r \le L_0$, we have 
\begin{align*}
\psi(r+\delta)+\psi(r-\delta)-2\psi(r)=\int_r^{r+\delta}
ds\int_{s-\delta}^s \psi''(u)
du\leq\psi''(r+\delta)\delta^2,
\end{align*}
 for $a\in(0,\frac{1}{2})$, $|z|\leq a|y_1-y_2|$, by \cite[p.1068, (3.1)]{JW2}
\begin{align*}
\psi(|y_1-y_2+(z-\varphi(z))|)+&\psi(|y_1-y_2-(z-	\varphi(z))|)-2\psi(|y_1-y_2|)\\
%&=\psi\Big(|y_1-y_2|+\frac{2\langle y_1-y_2,z\rangle }{|y_1-y_2|}\Big)+\psi \Big(|y_1-y_2|-\frac{2\langle y_1-y_2,z\rangle }{|y_1-y_2|} \Big)-2\psi\bigl(|y_1-y_2|)	\\
&\leq4\psi''((1+2a)|y_1-y_2|)\frac{\langle y_1-y_2,z\rangle ^2}{|y_1-y_2|^2},
\end{align*}
then $\forall y_1,y_2,\in\mathbb{R}^{d_2}, z\in\mathbb{R}^{d_2}\setminus \{0\}$, since $\psi''<0$, by uniform ellipticity of $\widetilde{H}(y_1,y_2,z)$ deduced from Lemma \ref{jac}, let $ M=\widetilde{H}(y_1,y_2,z)$,  we have
\begin{align}
\widetilde{\mathcal{L}}
&\psi(|x-y|)\le 2\psi''((1+2a)|y_1-y_2|)\int_{ \{|z|\le a{|y_1-y_2|} \}}\frac{|\langle y_1-y_2,z\rangle |^2}{|y_1-y_2|^2}\frac{\widetilde{H}(y_1,y_2,z)}{|z|^{d_2+\alpha_2}}dz\nonumber\\
&+\psi'(|y_1-y_2|) \frac{\langle f(y_1)-f(y_2),y_1-y_2\rangle }{|y_1-y_2|}\nonumber\\
&= 2\psi''((1+2a)|y_1-y_2|)\int
_{\{|z|\le a{|y_1-y_2|} \}}{|z_1|^2}\frac{\widetilde{H}(y_1,y_2,z)}{|z|^{d_2+\alpha_2}}dz+\psi'(|x-y|) \frac{\langle f(y_1)-f(y_2),y_1-y_2\rangle }{|y_1-y_2|}\nonumber\\
&= \frac{2M}{d_2}\psi''((1+2a)|y_1-y_2|)
\int_{ \{|z|\le a{|y_1-y_2|} \}}{|z|^2}\frac{1}{|z|^{d_2+\alpha_2}} dz+\psi'(|y_1-y_2|)\frac{\langle f(y_1)-f(y_2),y_1-y_2\rangle }{|y_1-y_2|}\nonumber\\
&\leq \biggl[-\frac{2M L_0^{1-\alpha_2}}{d_2(2-\alpha_2)}c_1a^{2-\alpha_2}e^{-2c_1a L_0}+c \biggr] c_1e^{-c_1|y_1-y_2|}|y_1-y_2|,\label{psi}
\end{align}
where we used the fact that $\frac{\widetilde{H}(y_1,y_2,z)}{|z|^{d_2+\alpha_2}}dz$ is  rotationally invariant under the transformation $z\rightarrow\varphi$ in the second equality, the Assumption \ref{par-dis} of $f$ in \eqref{2.2}, $c$ is the constant in \eqref{2.2}, $ 1<\al_{2}<2$ are applied in the last inequality.  Let $K=\frac{2M L_0^{1-\alpha_2}}{d_2(2-\alpha_2)}, c_1=(2c/K)^{1/(\alpha-1)}
\mathrm{e}^{2L_0/(\alpha-1)}+2,  a=1/c_1,$ then similar to \cite[Proposition 3.1]{JW2},
since $\psi(0)=0$, $\psi''\leq0$,
$\psi(r)\leq r\psi'(r)= c_1\mathrm{e}^{-c_1 r} r,\ r
\in[0,L_0],$  $\exists\beta>0$ s.t.,
\begin{align*}
	\widetilde{\mathcal{L}}\psi(|y_1-y_2|)\leq-\beta \psi(|y_1-y_2|).
\end{align*}

Next, we consider the case that $|y_1-y_2|>L_0$, $\forall y_1,y_2,\in\mathbb{R}^{d_2}, z\in\mathbb{R}^{d_2}\setminus \{0\}$ and $L_0<
|y_1-y_2|<2L_0$, by Assumption \ref{par-dis} of $f$ and $\psi'>0$, from \eqref{psi} we have
\begin{align*}
\widetilde{\mathcal{L}}\psi(|y_1-y_2|)\leq-c_1\cdot Ce^{-c_1|y_1-y_2|}|y_1-y_2|,
\end{align*}
then for $|y_1-y_2|\geq2L_0$,
\begin{align*}
\widetilde{\mathcal{L}}\psi(|y_1-y_2|)\leq-C \bigl[Ac_2e^{c_2(|y_1-y_2|-2L_0)}+2B(|y_1-y_2|-2L_0)\bigr]|y_1-y_2|,
\end{align*}
consider the function
\begin{align}\label{r-2l}
g(r)=\frac{1}{2}Ac_2e^{c_2(r-2L_0)}+2B(r-2L_0),\ r\in [2L_0,\infty),
\end{align}
by the definitions of $A$ and
$B$, there exist unique $r_1\in[2L_0,\infty)$ s.t. $g'(r_1)=0$, and 
$
g(r_1)=\frac{-2B}{c_2} \big[1-\log\frac{-4B}{Ac_2^2} \big]=\frac
{-2B}{c_2} \big[1-\log \frac{2(c_1+c_2)}{c_2} \big],
$
let $c_2$ large enough so that $\log \frac{2(c_1+c_2)}{c_2}\le\log2.1,$ recall that $B<0$, then $g(r_1)>0$, which means $g(r)>0$ when $r\in [2L_0,\infty)$, we duduce that 
\begin{align*}
	\widetilde{\mathcal{L}}\psi(|y_1-y_2|)\leq-\frac{1}{2} CAc_2e^{c_2(|y_1-y_2|-2L_0)}|y_1-y_2|,
\end{align*}
combining above analysis, we have $\exists\beta>0$ s.t.,
\begin{align*}
	\widetilde{\mathcal{L}}\psi(|y_1-y_2|)\leq-\beta \psi(|y_1-y_2|),
\end{align*}
the proof is complete.
\end{proof}

%We mention that according to the definition of $\psi(r)$,  $\forall p\geq1$, $\exists 0<c(p)<\infty$ s.t. 
%\begin{align}\label{rp}
%|r|^p\leq c(p)\psi(r),\ r>0,
%\end{align}
%and $\exists C>0$ s.t. $\psi(r)\leq Cr$ when $r\in (0, 2L_0]$, see Lemma \ref{cp-psi} in Appendix.

Next we prove Theorem \ref{was-exp}.
\begin{proof}[\textbf{Proof of Theorem $\ref{was-exp}$}]
We apply the coupling process $(Y'_t,Z'_t)$ constructed
in Section \ref{exp-con}. Taking the similar procedure in \cite[p.1012, Proof of Theorem 1.2]{JW2}, we have
%Denote by $\widetilde{\mathbb
%	{P}}^{(y_1,y_2)}$ and
%$\widetilde{\mathbb{E}}^{(y_1,y_2)}$ the distribution and the expectation of
%$(Y'_t,Z'_t)$ starting from $(y_1,y_2)$, respectively. 
%$\forall t>0$ let $r_t=|Y_t'-Z_t'|$, and  $\forall %n\geq1$ define the stopping time
%$$T_n=\inf\{t>0: r_t\notin[1/n, n]\},$$
%$\forall y_1,y_2\in\mathbb{R}^d$ with $|y_1-y_2|>0$,
%we take $n$ large enough s.t. $1/n<|y_1-y_2|<n$. For $\psi(r)$ we have 
%\begin{align*}
%\widetilde{\mathbb{E}}^{(y_1,y_2)}\psi(|Y'_{t\wedge T_{n}}-Z'_{t\wedge	T_{n}}|)&=\psi(|y_1-y_2|)+\widetilde{\mathbb{E}}^{(y_1,y_2)} \int_0^{t\wedge T_{n}} \widetilde{\mathcal{L}} \psi (|Y'_{s}-Z'_{s}|)ds\\
%&\leq\psi(|y_1-y_2|)-\lambda\widetilde{\mathbb{E}}^{(y_1,y_2)}\int_0^t \psi (|Y'_{{s\wedge T_{n}}}-Z'_{{s\wedge T_{n}}}|)ds,
%\end{align*}
%then by Gronwall ineqyality we have 
%$\mathbb{E}\psi(r_{t\wedge T_n})\leq\psi(r_0)e^{-\beta t},$
%since  $(Y'_t,Z'_t)$ is non-explosive, then $T_n\uparrow T$ a.s. as $n\to\infty$,
%here $T$ is the coupling time of the process $(Y_t',Z_t')$, from Fatou's lemma, let $n\to\infty$ we have
%$\mathbb{E}\psi(r_{t\wedge T})\leq\psi(r_0)e^{-\beta t}.$
%Since $Z'_t=Y'_t$ for $t\geq T$, we have
%$r_t=0$ for all $t\geq T$, therefore
%$\mathbb{E}\psi(r_t)\leq\psi(r_0)e^{-\beta t},$which means that 
\begin{align*}
	\mathbb{E}\psi(|Y_t'-Z_t'|)\leq\psi(|y_1-y_2|)e^{-\beta t},
\end{align*}
when $|y_1-y_2|\leq L_0$, $\forall p\ge1$, $t>0$, from \eqref{cp-psi-2} in Lemma \ref{cp-psi}, $\exists 0<C(p)<\infty$ s.t. 
\begin{align}\label{e-p}
W_p(\delta_{y_1}P_t,\delta_{y_2}P_t)^p\leq\mathbb{E}|Y_t'-Z_t'|^p\le C(p)\mathbb{E}\psi(|Y_t'-Z_t'|) \leq C(p)e^{-\beta t}|y_1-y_2|, 
\end{align}
the second inequality follows from definition of $\psi$ in Lemma \ref{inv-gen-L}.

For $|y_1-y_2|\geq L_0$, the second inequality in \eqref{e-p} does not hold since $\psi$ grows exponentially on $r\in[2L_0,\infty)$, so we split the distance of $|y_1-y_2|$ into pieces and employ \eqref{e-p} to derive the result.  Let $n=\lfloor\frac{|y_1-y_2|}{L_0}\rfloor+ 1\geq2$, then
%\begin{align}\label{n-pi}
$\frac{n}{2}\le n-1\leq\frac{|y_1-y_2|}{L_0}\leq n,$
%\end{align}
let $y^i=y_1+i\frac{y_2-y_1}{n}$ for $i=0,1,\ldots, n$, and $y^0=x$, $y^n=y_2$, we have $\forall i=1,2,\ldots,n$, $|y^{i-1}-y^{i}|=\frac{|y_1-y_2|}{n}\le L_0$,
thus 
\begin{align}\label{e-p-2}
W_p(\delta_{y_1}P_t,\delta_{y_2}P_t)&\leq \sum_{i=1}^n
W_{p}(\delta _{y^{i-1}} P_t,\delta_{y^{i}}P_t)\leq\sum_{i=1}^n(\mathbb{E}|Y_t'(y^{i-1})-Z_t'(y^{i})|)^{1/p} \leq C(p) e^{-\frac{\beta t}{p}}\sum_{i=1}^n|y^{i-1}-y^{i}|^{1/p}\nonumber\\
&\leq C(p)e^{-\frac{\beta t}{p}} nL_0^{1/p}\leq 2C_pL_0^{1/p-1}e^{-\frac{\beta t}{p}}|y_1-y_2|\leq C(p)e^{-\frac{\beta t}{p}}|y_1-y_2|,
\end{align}
we used \eqref{e-p} in third inequality, the proof is complete.
\end{proof}

\subsection{Invariant measure of \ref{4.1}}\label{inv-gen}
\begin{lemma}\label{L-inv-gen}
	Suppose that $ f(x,\cdot)\in C^{1}_{b}$,  Assumption $\ref{par-dis}$ and  Assumption $\ref{lip}$ are valid, for any fixed  $ x\in \mathbb{R}^{d_{1}}$, $\forall  t\geq0, y_{1}, y_{2}\in \mathbb{R}^{d_{2}}$, we have $\exists \beta>0$ s.t.
	$$ 	\mathbb{E}| Y_{t}^{ x, y_{1}}-Y_{t}^{x, y_{2}}|\leq C\cdot e^{-\beta t}| y_{1}- y_{2}|.$$
\end{lemma}
\begin{proof}
From %the calculations in \eqref{e-p} and \eqref{e-p-2} in  
Theorem \ref{was-exp}, and take $p=1$, $\forall |y_{1}- y_{2}|>0$, we can derive this estimate.
\end{proof}

For any bounded measurable function $g:\mathbb{R}^{d_{2}}\rightarrow \mathbb{R}$, denote $ g(y)$, we have
$ P_{t}^{x}g(y)=\mathbb{E}g(Y_{t}^{ x,y})$,  $t\geq 0,$  $y\in \mathbb{R}^{d_{2}}.$ From Theorem \ref{was-exp}, we deduce that there exists unique invariant measure $\rho^x$ for \eqref{4.1}, then we define the average with respect to $\rho^x$ as
\begin{align*}
	\bar{g}=\rho^x(g)=\int_{\R^{d_2}}g(y)\rho^x(dy),
\end{align*}
 the following lemma is crucial in our analysis.
\begin{lemma}\label{L-gen-exp}
	Suppose that $ f(x,\cdot)\in C^{1}_{b}$, Assumption $\ref{par-dis}$ is valid. $\forall  t\geq0$,  let $\forall g(y)\in C^{1}_{b}$, for fixed $x\in \mathbb{R}^{d_{1}},  y_{1}, y_{2}\in \mathbb{R}^{d_{2}}$, $\exists \beta>0$ s.t., for any $y\in \mathbb{R}^{d_{2}}$, we have 
	$$ \sup_{x\in \mathbb{R}^{d_{1}}}| P_{t}^{x}g(x,y)-\bar{g}(x)|\leq C\cdot Lip(g)e^{-\beta t}(1+|y|),$$
	here $ Lip(g)=\sup_{x\neq y}\frac{|g(x)-g(y)|}{|x-y|}.$
\end{lemma}
\begin{proof}
By definition of invariant mesure, Lemma \ref{L-inv-gen}, we have
\begin{align*}
|\mathbb{E}g(Y_t^{x,y})-\rho^x(g)|&=
\left|\mathbb{E}g(Y_t^{x,y})-\int_{\R^{d_2}}g(z)\rho^x(dz)\right|\leq\left|\int_{\R^{d_2}}\mathbb{E}g(Y_t^{x,y})-\mathbb{E}g(Y_t^{x,z})\rho^x(dz)\right|\\
&\leq Lip(g) \left|\int_{\R^{d_2}}\mathbb{E}|Y_t^{x,y}-Y_t^{x,z}|\rho^x(dz)\right|\leq C\cdot Lip(g)e^{-\beta t}(1+|y|),
\end{align*}
the proof is complete.
\end{proof}

\subsection{Moment estimates of $ Y_{t}^{ x,y}$}
%We remind that the following results are only used for strong convergence analysis in general case.
\begin{proposition}\label{P41}
	Let  Assumption $\ref{par-dis}$ holds, we have for $m\in [1,\alpha_{1}\wedge\alpha_{2})$, for $T\geq 1$,
	\begin{equation}\label{4.4}
	\sup_{t\geq0} \mathbb{E}|Y_{t}^{ x,y}|^{m}\leq C_{m}(1+|y|^{m}),
	\end{equation}
	\begin{equation}\label{4.5}
		\mathbb{E}\left(\sup_{t\in [0,T]} |Y_{t}^{ x,y}|^{m}\right)\leq C_{m}(T^{\frac{m}{2}}+|y|^{m}).
	\end{equation}
\end{proposition}
\begin{proof}
The proof follows from \cite[Lemma A.1]{SXX}, \cite[Theorem 4.1]{KY},	\eqref{4.4} can be derived from \eqref{3.6} and we shall prove \eqref{4.5}.
	We define 
	\begin{align*}
		U_{T}(y)=(|y|^{2}+T)^{\frac{m}{2}},
	\end{align*}
	so that similar to \eqref{du} and \eqref{d2u},
	\begin{align*}
		\begin{split}
				|DU_{T}(y)|&=\left|  \frac{my}{(|y|^{2}+T)^{1-\frac{m}{2}}}\right|\leq C_{m} |y|^{m-1},\\
		|D^{2}U_{T}(y)|&=\left|   \frac{mI_{d_{2}\times d_{2} }}{(|y|^{2}+T)^{1-\frac{m}{2}}}-\frac{m(m-2)y\otimes y}{(|y|^{2}+T)^{2-\frac{m}{2}}} \right|\leq C_{m}T^{\frac{m-2}{2}},
		\end{split}
	\end{align*}
by It\^{o}'s formula,
	\begin{align*}
			&U_{T}(Y_{t}^{ x,y})=U_{T}(y)+ \int_{0}^{t} \langle f(x,Y_{r}^{ x,y}),DU_{T}(Y_{r}^{x,y})\rangle  dr \\
			&+ \int_{0}^{t}\int_{|z|\leq T^{\frac{1}{2}}}\Big( U_{T}(Y_{r}^{ x,y}+\delta_{2}(Y_{r}^{ x,y})z)-U_{T}(Y_{r}^{ x,y})-(DU_{T}(Y_{r}^{ x,y}), \delta_{2}(Y_{r}^{ x,y})z)  \Big) \tilde{N}^{2}(dr,dz)\\
			&+\int_{0}^{t} \int_{|z|> T^{\frac{1}{2}}}\Big( U_{T}(Y_{r}^{x,y}+\delta_{2}(Y_{r}^{ x,y})z)-U_{T}(Y_{r}^{ x,y}) \Big) \nu_{2}(dz)dr\\
			&\leq  \int_{0}^{t} \langle f(x,Y_{r}^{x,y}),DU_{T}(Y_{r}^{ x,y})\rangle  dr+\mathbb{E} \int_{0}^{t}\int_{|z|> T^{\frac{1}{2}}}\Big( U_{T}(Y_{r}^{x,y}+\delta_{2}(Y_{r}^{ x,y})z)-U_{T}(Y_{r}^{ x,y}) \Big) \nu_{2}(dz)dr\\
			&+\int_{0}^{t} \int_{|z|\leq T^{\frac{1}{2}}}\Big( U_{T}(Y_{r}^{x,y}+\delta_{2}(Y_{r}^{ x,y})z)-U_{T}(Y_{r}^{ x,y})- \langle DU_{T}(Y_{r}^{x,y}), \delta_{2}(Y_{r}^{ x,y})z\rangle \Big) \nu_{2}(dz)dr\\
			&+\int_{0}^{t} \int_{|z|> T^{\frac{1}{2}}}\Big( U_{T}(Y_{r}^{x,y}+\delta_{2}(Y_{r}^{ x,y})z)-U_{T}(Y_{r}^{ x,y}) \Big) N^{2}(dr,dz)+U_{T}(y)=\hat{I}_{1}+\hat{I}_{2}+\hat{I}_{3}+\hat{I}_{4}+U_{T}(y),
	\end{align*}
	so by Assumption \ref{par-dis} of $f(x,y)$ in  \eqref{2.4} and $T\geq 1$, obviously we have for $\hat{I}_{1}$,
	\begin{align*}
		\mathbb{E}\left(\sup_{t\in [0,T]} |\hat{I}_{1}(t)|\right)\leq \int_{0}^{T} \frac{C_{m}}{(|Y_{r}^{ x,y}|^{2}+T)^{1-\frac{m}{2}}}dr\leq C_{m}T^{\frac{m}{2}-1+1}\leq C_{m}T^{\frac{m}{\alpha_{2}}},
	\end{align*}
%for $\hat{I}_{3}$, from boundedness condition of $\delta_{2}(x,y)$ we have	\begin{equation}\label{4.10}	\begin{split}&\mathbb{E}\left(\sup_{t\in [0,T]} |\hat{I}_{3}(t)|\right)\leq C_{p}T^{\frac{p}{2}-1}\int_{0}^{T}\int_{|z|\leq T^{\frac{1}{2}}}|\delta_{2}(Y_{r}^{ x,y})z|^{2}\nu_{2}(dz)dr\leq C_{p}T^{\frac{p}{2}},\\\end{split}\end{equation}
for $\hat{I}_{2}$,  from Assumption \ref{gro-bou} of $\delta_{2}(x,y)$, by Burkholder-Davies-Gundy's inequality and \eqref{4.4},
	\begin{align*}
			&\mathbb{E}\left(\sup_{t\in [0,T]} |\hat{I}_{2}(t)|\right)\leq \mathbb{E}\left[ \int_{0}^{T} \int_{|z|\leq T^{\frac{1}{2}}}\left|  U_{T}(Y_{r}^{ x,y}+\delta_{2}(Y_{r}^{ x,y})z)-U_{T}(Y_{r}^{ x,y}) \right|^{2}  N_{2}(dz)dr\right] ^{\frac{1}{2}}\\
			&\leq \mathbb{E}\left[ \int_{0}^{T} \int_{|z|\leq T^{\frac{1}{2}}}\left( \left| Y_{r}^{ x,y}  \right|^{2m-2}|\delta_{2}(Y_{r}^{ x,y})z|^{2}+|\delta_{2}(Y_{r}^{ x,y})z|^{2m} \right)  \nu_{2}(dz)dr\right] ^{\frac{1}{2}}\\
			&\leq \frac{1}{4}\mathbb{E} \left(\sup_{r\in [0,T]} |Y_{r}^{ x,y}|^{m}\right)+C_{m}\left(\int_{0}^{T} \int_{|z|\leq T^{\frac{1}{2}}} |\delta_{2}(Y_{r}^{ x,y})z|^{2}  \nu_{2}(dz)dr\right) ^{m}+\int_{0}^{T} \int_{|z|\leq T^{\frac{1}{2}}} |\delta_{2}(Y_{r}^{ x,y})z|^{m}  \nu_{2}(dz)dr\\
			&\leq \frac{1}{4}\mathbb{E} \left(\sup_{r\in [0,T]} |Y_{r}^{ x,y}|^{m}\right)+C_{m}T^{\frac{m}{2}},
	\end{align*}
	for $\hat{I}_{3}$,
\begin{align*}
	&\mathbb{E}\left(\sup_{t\in [0,T]} |\hat{I}_{3}(t)|\right)\leq C_{m}T^{\frac{m}{\alpha_{2}}-\frac{2}{\alpha_{2}}}\int_{0}^{T}\int_{|z|\leq T^{\frac{1}{\alpha_{2}}}}|\de_2\cdot  z|^{2}\nu_{2}(dz)dr\leq C_{m}T^{\frac{m}{\alpha_{2}}},
\end{align*}
and	for $\hat{I}_{4}$,
\begin{align*}
	&\mathbb{E}\left(\sup_{t\in [0,T]} |\hat{I}_{4}(t)|\right)\leq \mathbb{E}\left[ \int_{0}^{T} \int_{|z|> T^{\frac{1}{\alpha_{2}}}}\left|  U_{T}(Y_{r}^{ x,y}+\de_2\cdot z)-U_{T}(Y_{r}^{ x,y}) \right|  N_{2}(dz)dr\right] \\
	&\leq \mathbb{E}\left[ \int_{0}^{T} \int_{|z|>T^{\frac{1}{\alpha_{2}}}}\left( \left| Y_{r}^{ x,y}  \right|^{m-1}|\de_2\cdot z|+|\de_2\cdot z|^{m} \right)  \nu_{2}(dz)dr\right]\\
	&\leq \frac{1}{4}\mathbb{E} \left(\sup_{r\in [0,T]} |Y_{r}^{ x,y}|^{m}\right)+C_{p}\left(\int_{0}^{T} \int_{|z|>T^{\frac{1}{\alpha_{2}}}} |\de_2\cdot z|^{2}  \nu_{2}(dz)dr\right) ^{m}+\int_{0}^{T} \int_{|z|> T^{\frac{1}{\alpha_{2}}}} |\de_2\cdot z|^{m}  \nu_{2}(dz)dr\\
	&\leq \frac{1}{4}\mathbb{E} \left(\sup_{r\in [0,T]} |Y_{r}^{ x,y}|^{m}\right)+C_{m}T^{\frac{m}{\alpha_{2}}},
\end{align*}
combining above estimayes  we derive \eqref{4.5}.
\end{proof}

\begin{lemma}\label{L43}
$\forall t\in [0,T]$,  $T\geq1$, we have
\begin{align*}
	& \mathbb{E}\left(\sup_{t\in [0,T]} |Y_{t}^{\varepsilon}|^{m}\right)\leq C_{T,m}\left( \varepsilon^{-\frac{m}{\alpha_{2}}}+|y|^{m}\right).
	\end{align*}
\end{lemma}
\begin{proof}
Denote $\tilde{L}^{2}_{t}=\frac{1}{\varepsilon^{\frac{1}{\alpha_{2}}}}L^{2}_{t\varepsilon}$, so that 
\begin{align*}
\tilde{Y}_{t}^{\varepsilon}&=y+\frac{1}{\varepsilon}\int^{t\varepsilon}_{0}f(X_{s}^{\varepsilon},\tilde{Y}_{s}^{\varepsilon})ds+\frac{\de_{2}(X_{s}^{\varepsilon},\tilde{Y}_{s}^{\varepsilon})}{\varepsilon^{\frac{1}{\alpha_{2}}}}L^{2}_{t\varepsilon}=y+\int ^{t}_{0}f(X_{s\varepsilon}^{\varepsilon},\tilde{Y}_{s\varepsilon}^{\varepsilon})ds+\de_{2}(X_{s}^{\varepsilon},\tilde{Y}_{s}^{\varepsilon})\tilde{L}^{2}_{t},
\end{align*}
we can see that $\tilde{Y}_{t}^{\varepsilon}$ and $Y_{t}^{\varepsilon}$ have the same law, since $\forall x\in \mathbb{R}^{d_{1}}$, $y\in \mathbb{R}^{d_{2}}$, $\|\de_{2}(x,y)\|_{\infty}<\infty$, and we have $\sup_{x\in \mathbb{R}^{d_{1}}}f(x,0)<\infty$, then similar to the proof of \eqref{4.5}, we have
\begin{align*}
& \mathbb{E}\left(\sup_{t\in [0,T]} |\tilde{Y}_{t}^{\varepsilon}|^{m}\right)\leq C_{m}\left( T^{\frac{m}{\alpha_{2}}}+|y|^{m}\right),
\end{align*}
from \eqref{2.2}, \eqref{2.4}, and \eqref{4.5}, for any $T\geq 1$,
\begin{align*}
&\mathbb{E}\left(\sup_{t\in [0,T]} |Y_{t}^{\varepsilon}|^{m}\right)= \mathbb{E}\left(\sup_{t\in [0,\frac{T}{\varepsilon}]} |\tilde{Y}_{t}^{\varepsilon}|^{m}\right)\leq C_{m}\left( \left( \frac{T}{\varepsilon}\right) ^{\frac{m}{\alpha_{2}}}+|y|^{m}\right) \leq C_{T,m}\left( \varepsilon^{-\frac{m}{\alpha_{2}}}+|y|^{m}\right).
\end{align*}
\end{proof}
\begin{remark}
Lemma $\ref{L43}$ is only used, but very important, for strong convergence estiamtes, see \eqref{5.49} in Theorem $\ref{T52}$ for details.
% while \eqref{4.4} is applied in weak convergence.
\end{remark}

\section{Strong convergence estimates for \ref{1.1}}\label{str-est}
We begin with introducing  mollification of functions which will be used to tackle the difficulities related to time derivative and slow component. Let $\rho_{1}:\mathbb{R}\rightarrow [0,1]$, $\rho_{2}: \mathbb{R}^{ d_{1}}\rightarrow [0,1]$ be two nonnegative smooth mollifiers s.t. 

(1). $\rho_{1}\in C_{0}^{\infty}(\mathbb{R})$, supp $  \rho_{1}\subset \overline{B_{1}(0)}=\left\lbrace t\in \mathbb{R}:|t|\leq 1 \right\rbrace $, and $ \rho_{2}\in C_{0}^{\infty}(\mathbb{R}^{ d_{1}})$, supp $  \rho_{2}\subset \overline{B_{1}(0)}=\left\lbrace x\in \mathbb{R}^{ d_{1}}:|x|\leq 1 \right\rbrace $;

(2). $\int_{\mathbb{R}}\rho_{1}(t)dt=\int_{\mathbb{R}^{ d_{1}}}\rho_{2}(x)dx=1$;

(3). $\forall k\geq 0$, $\exists C_{k}>0$ s.t. $|\nabla^{k}\rho_{1}(t)|\leq C_{k}\rho_{1}(t)$, $|\nabla^{k}\rho_{2}(x)|\leq C_{k}\rho_{2}(x)$.

Then for any $n\in  \mathbb{N}^{+}$, let $\rho_{1}^{n}(t)=n^{\alpha_{1}}\rho_{1}(n^{\alpha_{1}}t)$, $\rho_{2}^{n}(x)=n^{d_{1}}\rho_{2}(nx)$, then for $g(t,x,y)$, mollification of $g(t,x,y)$ in $t$ and $x$ is defined by
\begin{align}\label{mol}
	g_{n}(t,x,y)=g\ast\rho_{2}^{n}\ast\rho_{1}^{n}=\int_{\mathbb{R}^{ d_{1}+1}}g(t-s, x-z, y)\rho_{2}^{n}(z)\rho_{1}^{n}(s)dzds,
\end{align} 
in addition we define the fractional Laplacian operator  $-(-\Delta_{x})^{\frac{\alpha}{2}}f(x)$, $x,z\in \mathbb{R}^{d_{1}}$,  $0<\alpha<2$, as follows 
\begin{equation}\label{5.56-1}
	-(-\Delta_{x})^{\frac{\alpha}{2}}f(x)=P.V.\int_{\mathbb{R}^{d_{1}}}(u(x+\delta_{1}(x)z)-u(x)-\langle \delta_{1}(x)z,\nabla_{x}u(x)\rangle I_{|z|\leq1})\nu(dz),
\end{equation} 
where $\nu(dz)=\frac{c}{|z|^{d_{1}
		+\alpha}}dz$ is symmetric L\'{e}vy measure.
We mention that by mollification method we have $g_{n}(\cdot,x,y)\in  C_{0}^{\infty}(\mathbb{R})$, $g_{n}(t,\cdot,y)\in  C_{0}^{\infty}(\mathbb{R}^{ d_{1}})$, so we can get higher regularity estimates of $g_{n}(\cdot,\cdot,y)$ with respect to $t$ and $x$, thus we have the following lemma, which is analogous to \cite[Lemma 5.1]{KY}.

\begin{lemma}\label{L51}
	Let $g(t,x,y)\in C_{p}^{\frac{v}{\alpha_{1}}, v,\sigma }$ with $0<v\leq\alpha_{1}$, $0<\sigma<1$, and 
		 define $g_{n}$ by \eqref{mol}, then we have 
	\begin{equation}\label{5.57}
		\|g(\cdot,\cdot,y)-g_{n}(\cdot,\cdot,y)\|_{\infty}\leq C\cdot n^{-v}(1+|y|),
	\end{equation} 
	\begin{equation}\label{5.58-1}
		\|\partial_{t}g_{n}(\cdot,\cdot,y)\|_{\infty}\leq C\cdot n^{\alpha_{1}-v}(1+|y|), 
	\end{equation} 
	\begin{equation}\label{5.58-4}
		\|(-\Delta_{x})^{\frac{\alpha_{1}}{2}}g_{n}(\cdot,\cdot,y)\|_{\infty}\leq C\cdot n^{\alpha_{1}-v}(1+|y|),
	\end{equation} 
%	\iffalse	\begin{equation}\label{5.58-2}		\|\nabla_{x}^{2}g_{n}(\cdot,\cdot,y)\|_{\infty}\leq C\cdot n^{2-v}(1+|y|),	\end{equation} 	\fi
	\begin{equation}\label{5.58-3}
		\|\nabla_{x}g_ {n}(\cdot,\cdot,y)\|_{\infty}\leq C\cdot n^{1-(1\wedge v)}(1+|y|).
	\end{equation} 
%	we can further estimate that $ 	 \|\nabla_{x}^{2}g_{n}(\cdot,\cdot,y)\|_{\infty}\leq C\cdot n^{2-v}(1+|y|).$
\end{lemma}
\begin{proof}
The proof mainly refers to \cite[Lemma 4.1]{RX} and \cite[Lemma 5.1]{KY}. By definition of  H\"{o}lder derivative and a change of variable, for $0<v\leq1$,
 taking $y=nz$, from definintion of $\nu(dz)$ in \eqref{5.56-1}, we observe that 
\begin{align*} \nu(dz)=\frac{c}{|z|^{d_{1}+\alpha}}dz=\frac{c}{|n^{-1}y|^{d_{1}+\alpha}}(n^{-1})^{d_{1}}dy=n^{\alpha}\frac{c}{|y|^{d_{1}+\alpha}}dy=n^{\alpha}\nu(dy),	
	\end{align*}
therefore,
\begin{align}
\left|(-\Delta_{x})^{\frac{\alpha}{2}}\rho_{2}^{n}(x)\right|& =c\left| \int_{\mathbb{R}^{d_{1}}}\left( n^{d_{1}}\rho_{2}(nx+n\delta_{1}(x)z)-n^{d_{1}}\rho_{2}(nx)-\langle n\delta_{1}(x)z,\nabla_{x}n^{d_{1}}\rho_{2}(nx)\rangle I_{|nz|\leq1}\right) \nu(dz)\right|\nonumber \\
&=c\cdot n^{\alpha}\cdot n^{d_{1}}\left| \int_{\mathbb{R}^{d_{1}}}\left( \rho_{2}(nx+\delta_{1}(x)y)-\rho_{2}(nx)-\langle \delta_{1}(x)y,\nabla_{x}\rho_{2}(nx)\rangle I_{|y|\leq1}\right) \nu(dy)\right|\nonumber\\
& \leq C_{\alpha}\cdot n^{\alpha}\cdot n^{d_{1}}\rho_{2}(nx)\leq C_{\alpha}\cdot n^{\alpha}\rho_{2}^{n}(x),\label{del-lap}
\end{align}
we used definition in \eqref{5.56-1} and the fact that $\forall k\geq 0$, $\exists C_{k}>0$ s.t. $|\nabla^{k}\rho_{2}(x)|\leq C_{k}\rho_{2}(x)$ in first inequality. Consequently, by definition in \eqref{cp}, \eqref{del-lap}, since $g(t,x,y)\in C_{p}^{\frac{v}{\alpha_{1}}, v,\sigma }$, 
\begin{align*}
|(-\Delta_{x})^{\frac{\alpha_{1}}{2}}g_{n}(\cdot,\cdot,y)|&\leq \int_{\mathbb{R}^{ d_{1}+1}}	|g(t-s,x-z,y)-g(t-s,x,y)|\rho_{1}^{n}(s)|(-\Delta_{z})^{\frac{\alpha_{1}}{2}}\rho_{2}^{n}(z)|dzds\\
&\leq C\cdot n^{\alpha_{1}}\int_{\mathbb{R}^{ d_{1}+1}}	|z|^{v}(1+|y|)\rho_{1}^{n}(s)\rho_{2}^{n}(z)dzds\leq C\cdot n^{\alpha_{1}-v}(1+|y|),
	\end{align*}
we derive \eqref{5.58-4}, other results can be derived in the same approaches as in \cite[Lemma 5.1]{KY},
\begin{align*}
	&|g(t,x,y)-g_{n}(t,x,y)|\leq \int_{\mathbb{R}^{ d_{1}+1}}	|g(t,x,y)-g(t-s,x-z,y)|\rho_{1}^{n}(s)\rho_{2}^{n}(z)dzds\\
	&\leq C\cdot\int_{\mathbb{R}^{ d_{1}+1}}	(|s|^{\frac{v}{\alpha_{1}}}+|z|^{v})(1+|y|)\rho_{1}^{n}(s)\rho_{2}^{n}(z)dzds\leq C\cdot n^{-v}(1+|y|),
\end{align*}
\begin{align*}
	|\nabla_{x}^{2}g_{n}(\cdot,\cdot,y)|&\leq \int_{\mathbb{R}^{ d_{1}+1}}	|g(t-s,x-z,y)-g(t-s,x,y)|\rho_{1}^{n}(s)|\nabla^{2}_{z}\rho_{2}^{n}(z)|dzds\\
	&\leq C\cdot n^{2}\int_{\mathbb{R}^{ d_{1}+1}}	|z|^{v}(1+|y|)\rho_{1}^{n}(s)\rho_{2}^{n}(z)dzds\leq C\cdot n^{2-v}(1+|y|),
\end{align*}
\begin{align*}
	|\nabla_{x}g_{n}(\cdot,\cdot,y)|&\leq \int_{\mathbb{R}^{ d_{1}+1}}	|g(t-s,x-z,y)-g(t-s,x,y)|\rho_{1}^{n}(s)|\nabla_{z}\rho_{2}^{n}(z)|dzds\\
	&\leq C\cdot n\int_{\mathbb{R}^{ d_{1}+1}}	|z|^{v}(1+|y|)\rho_{1}^{n}(s)\rho_{2}^{n}(z)dzds\leq C\cdot n^{1-v}(1+|y|),
\end{align*}
\begin{equation}\label{5.L5}
		\begin{split}
	|\partial_{t}g_{n}(t,x,y)|&\leq \int_{\mathbb{R}^{ d_{1}}+1}	|g(t-s,x-z,y)-g(t,x-z,y)|\partial_{s}\rho_{1}^{n}(s)|\rho_{2}^{n}(z)dzds\\
	&\leq C\cdot n^{\alpha_{1}}\cdot\int_{\mathbb{R}^{ d_{1}}+1}|s|^{\frac{v}{\alpha_{1}}}\rho_{1}^{n}(s)\rho_{2}^{n}(z)(1+|y|)dzds\leq C\cdot n^{\alpha_{1}-v}(1+|y|),
	\end{split}
\end{equation}
for $1<v\leq\alpha_{1}$,
\begin{align*}
	|g(t,x,y)-g_{n}(t,x,y)|	&\leq \int_{\mathbb{R}^{ d_{1}}+1}	|g(t-s,x+z,y)+g(t-s,x-z,y)-2g(t,x,y)|\rho_{1}^{n}(s)\rho_{2}^{n}(z)dzds\\
	&\leq C\cdot\int_{\mathbb{R}^{ d_{1}}+1}	(|s|^{\frac{v}{\alpha_{1}}}+|z|^{v})(1+|y|)\rho_{1}^{n}(s)\rho_{2}^{n}(z)dzds\leq C\cdot n^{-v}(1+|y|),
	\end{align*} 
applying \eqref{del-lap} we have
\begin{align*}
	|(-\Delta_{x})^{\frac{\alpha_{1}}{2}}g_{n}(\cdot,\cdot,y)|&\leq \int_{\mathbb{R}^{ d_{1}+1}}	|\nabla_{x}g(t-s,x-z,y)-\nabla_{x}g(t-s,x,y)|\rho_{1}^{n}(s)|(-\Delta_{z})^{\frac{\alpha_{1}-1}{2}}\rho_{2}^{n}(z)|dzds\\
	&\leq C\cdot n^{\alpha_{1}-1}\int_{\mathbb{R}^{ d_{1}+1}}	|z|^{v-1}(1+|y|)\rho_{1}^{n}(s)\rho_{2}^{n}(z)dzds\leq C\cdot n^{\alpha_{1}-v}(1+|y|),
\end{align*}
	%\begin{equation}	\begin{split}	|\nabla_{x}^{2}g_{n}(\cdot,\cdot,y)|&\leq \int_{\mathbb{R}^{ d_{1}+1}}	|\nabla_{x}g(t-s,x-z,y)-\nabla_{x}g(t-s,x,y)|\rho_{1}^{n}(s)|\nabla_{z}\rho_{2}^{n}(z)|dzds\\	&\leq C\cdot n\int_{\mathbb{R}^{ d_{1}+1}}	|z|^{v-1}(1+|y|)\rho_{1}^{n}(s)\rho_{2}^{n}(z)dzds\leq C\cdot n^{2-v}(1+|y|),	\end{split}\nonumber	\end{equation}
\begin{equation}\label{5.58-32}
\begin{split}
	|\nabla_{x}g_{n}(\cdot,\cdot,y)|&\leq \int_{\mathbb{R}^{ d_{1}+1}}	|\nabla_{x}g(t-s,x-z,y)|\rho_{1}^{n}(s)|\rho_{2}^{n}(z)|dzds\\
	&\leq C\cdot \int_{\mathbb{R}^{ d_{1}+1}}	(1+|y|)\rho_{1}^{n}(s)\rho_{2}^{n}(z)dzds\leq C\cdot (1+|y|),
	\end{split}
\end{equation}
the proof of estimate $\partial_{t}g_{n}(t,x,y)$ when $1<v\leq\alpha_{1}$ is obtained in \eqref{5.L5}.
\end{proof}
\begin{remark}
Above results claim that these estimates we need are uniformly bounded both in $t$ and $x$, and  bounded from above by $|y|$ of order $1$, this conculsion plays important role in strong and weak convergence estimates, which is also consistent with moment estimates in Theorem $\ref{T32}$, ordre $p$ in Theorem $\ref{T52}$, and Theorem $\ref{T62}$ where orders of $|y|$ are $1$.
\end{remark}

Insipred by \cite{RX}, \cite{SXX} and \cite{KY}, we consider the nonlocal Poisson equation as follows, %which can be regarded as corrector equation to eliminate the effects of inhomogeneous terms by  generator of $Y_{t}$.Let $g(t,x,y)$  satisfy Lipschitz condition, growth condition, partially dissipative condition, 
\begin{equation}\label{5.1-1}
	\mathcal{L}_{2}(x,y)u(t,x,y)+g(t,x,y)-\bar{g}(t,x)=0,
\end{equation}
here we define $ 
\bar{g}(t,x)=\int_{\mathbb{R}^{d_{2}}}g(t,x,y)\rho^{x}(dy),$
%$ while in spatial periodic case we have$\bar{g}(t,x)=\int_{\mathbb{R}^{d_{2}}/\Lambda}g^{\Lambda}(t,x,y)\mu^{x}(dy),$  
and we have the following regularity estimates.

\begin{theorem}\label{T51-1}
Let $g$ satisfies Assumption $\ref{par-dis}$,  Assumption $\ref{gro-bou}$,  Assumption $\ref{lip}$,	$\forall x\in \mathbb{R}^{d_{1}}, $  $y\in \mathbb{R}^{d_{2}} $,  $t\in [0,T]$,  $g(t,x,\cdot)\in C_{b}^{ 2}(\mathbb{R}^{d_{2}})$, we define 
	\begin{equation}\label{5.2-1}
		\begin{split}
			u(t,x,y)=\int^{\infty}_{0} \big[\mathbb{E}g(t,x,Y_{s}^{x,y})-\bar{g}(t,x)\big]ds,
		\end{split}
	\end{equation}
	then $u(t,x,y)$  is a solution of \eqref{5.1-1}. And %we have $\rho^{x}(u)=0$,  %and  $u(t,\cdot,y)\in C_{p}^{v}(\mathbb{R}^{d_{1}})$, $u(t,x,\cdot)\in C_{p}^{2}(\mathbb{R}^{d_{2}})$,  
	there exists $ C>0$ s.t.,
	\begin{equation}\label{5.3-1}
		\sup_{t\in [0,T]}\sup_{x\in \mathbb{R}^{d_{1}}}|u(t,x,y)|\leq C_T(1+|y|).
	\end{equation}
Additionally, we have the gradient estimate, %which holds for both general and periodic cases, 
\begin{equation}\label{5.4-1}
		\sup_{t\in [0,T]} \sup_{\substack{x\in \mathbb{R}^{d_{1}}, y\in \mathbb{R}^{d_{2}}}}|\nabla_{y}u(t,x,y)|\leq C_T.
	\end{equation}
%In general case, we can further derive above gradient estimate by an easier method
%\begin{align}\label{5.4-2}
%	\sup_{t\in [0,T]} \sup_{\substack{x\in \mathbb{R}^{d_{1}}, y\in \mathbb{R}^{d_{2}}}}\mathbb{E}|\nabla_{y}u(t,x,y)|\leq C_{T}.
%\end{align}
\end{theorem}
\begin{proof}
Our proof refers to \cite[Proposition 3.3]{SXX} and \cite[Theorem 5.2]{KY}. From It\^{o}'s formula,  we can deduce that $u(t,x,y)$ is a solution of \eqref{5.1-1}. %and properties of  $u(t,\cdot,y)\in C^{v}_{p}(\mathbb{R}^{d_{1}})$, $u(t,x,\cdot)\in C_{p}^{2+\gamma}(\mathbb{R}^{d_{2}})$ are deduced from regularities of $g(t,x,y)$,
%$\rho^{x}(g-\bar{g})=0$ leads to $\rho^{x}(u)=0$. 

 $\forall g(t,x,\cdot)\in C_{b}^{ 2}(\mathbb{R}^{d_{2}})$, by Lemma \ref{L-gen-exp},
\begin{align*}
\sup_{t\in [0,T]}\sup_{x\in \mathbb{R}^{d_{1}}} |u(t,x,y)|\leq \int_{0}^{\infty}|\mathbb{E}g(t,x,Y_{s}^{x,y})-\bar{g}(t,x)|ds&\leq  C\int_{0}^{\infty}e^{-\frac{\beta s}{2}}(1+|y|)ds\leq  C(1+|y|),
\end{align*}
we obtain \eqref{5.3-1}.

Next we deal with gradient estimate \eqref{5.4-1}. From Leibniz chain rule,
\begin{equation}\label{nab-y-u}
		\nabla_{y}u(t,x,y)= \int_{0}^{\infty}\mathbb{E}\nabla_{y}g(t,x,Y_{s}^{x,y})\nabla_{y}Y_{s}^{x,y}ds,
\end{equation}
here $ \nabla_{y}Y_{s}^{x,y}$ satisfies
\begin{equation}
	\left\{
	\begin{aligned}
		&d\nabla_{y}Y_{s}^{x,y}= \nabla_{y}f(t,x,Y_{s}^{x,y})\cdot \nabla_{y}Y_{s}^{x,y}ds,\\
		&\nabla_{y}Y_{0}^{x,y}=\frac{Y_{0}^{x,y_{1}}-Y_{0}^{x,y_{2}}}{y_{1}-y_{2}}=\frac{y_{1}-y_{2}}{y_{1}-y_{2}}=I,
	\end{aligned}\nonumber
	\right.
\end{equation}
since by Lemma \ref{L-inv-gen} we have 
\begin{equation*}
 \sup_{\substack{x\in \mathbb{R}^{d_{1}}, y\in \mathbb{R}^{d_{2}}}}\mathbb{E}|\nabla_{y}Y_{s}^{x,y}|\leq C_{T}e^{-\frac{\beta s}{2}},\ s>0,
\end{equation*}
substitute this into \eqref{nab-y-u}, with the boundedness of $\nabla_{y}g(t,x,y)$, we can deduce that $\exists C_{T}>0 $ s.t.,
\begin{align*}
\sup_{t\in [0,T]} \sup_{\substack{x\in \mathbb{R}^{d_{1}}, y\in \mathbb{R}^{d_{2}}}}|\nabla_{y}u(t,x,y)|\leq C_{T},
\end{align*}
which asserts \eqref{5.4-1},
proof is complete.
\end{proof}

\iffalse
\begin{remark}\label{nab-y}
In most existing works about multiscale systems driven by $\alpha$-stable processes, the coefficients of $L_t$ are constants, which facilitates the existence of invariant measure and exponential ergodicity with the help of partially dissipative conditions of drifts, particularly the crucial gradient estimate \eqref{5.4-1} used to deal with martingale term of fast component $Y_t^{\va}$, see in \eqref{5.54}.

However, in our model this estimate is complicated since we have jump coefficients $\de_{2}$, so we start with the existence of transition probability density $p(s,y,Y)$ of $Y_{s}^{x,y}$ in Lemma $\ref{L41}$ by \cite[743, Theorem 3.1]{VK}, then we employ asymptotic expansion of transition density with respect to constant coefficients process \cite[739, Proposition 3.1; 743, Theorem 3.1]{VK} to derive \eqref{asy} and \eqref{part-q}, together with the uniform boundenness of transition semigroup $P^{x}_{s}$ in $C^{2}(\mb{R}^{d_{2}})$ \cite[743, Theorem 3.1; 749, Proposition 3.4]{VK}, \eqref{s<=1} is proved. We must emphasize that regularities of drift $f$ and spherical measure $\pi(y)$ with respect to $y$ are necessary to enable differentiability of $p(s,y,Y)$ and global two-sided heat kernel estimates.
 \end{remark}
\fi

Next, we deal with the difficulty arising from $b(t,x,y)-\bar{b}(t,x)$, which has zero mean with respect to $\rho^x$, %and $\mu^{x}$, i.e.,
 $\int_{\mathbb{R}^{d_{2}}}b(t,x,y)-\bar{b}(t,x)\rho^{x}(dy)=0$, %and $\int_{\mathbb{R}^{d_{2}}/\Lambda}b(t,x,y)-\bar{b}(t,x)\mu^{x}(dy)=0$ respectively, 
 here the averaged coefficients $\bar{b}$ is defined as
\begin{align*}
\bar{b}(t,x)=\rho^{x}(b)=\int_{\mathbb{R}^{d_{2}}}b(t,x,y)\rho^{x}(dy),%\quad\text{or}\quad \bar{b}(t,x)=\mu^{x}(b^{\Lambda}\circ R^{\Lambda})=\int_{\mathbb{R}^{d_{2}}/\Lambda}b^{\Lambda}(t,x,y)\mu^{x}(dy),
\end{align*}
then we have the following theorem. Recall that $(a)^{+}=\text{max}\{a, 0\}$.

\begin{theorem}\label{T52}
Suppose that $b(\cdot,\cdot,\cdot)\in C_{p}^{\frac{v}{\alpha_{1}},v,2+\gamma}(\mb{R}^{+}\times\mathbb{R}^{d_{1}}\times\mathbb{R}^{d_{2}})\cap C_b^{2}(\mathbb{R}^{d_{2}})$, $v\in((\alpha_{1}-\alpha_{2})^{+},\alpha_{1}]$, let Assumption $\ref{par-dis}$ and Assumption $\ref{gro-bou}$  hold, for $m\in [1,\alpha_{1}\wedge\alpha_{2})$, then we have 
\begin{equation}\label{5.44}
	\mathbb{E}\left( \sup_{t\in [0,T]}\left| \int_{0}^{t}\left( b(s,X_{s}^{\varepsilon},Y_{s}^{\varepsilon})-\bar{b}(s,X_{s}^{\varepsilon})\right)ds\right|^{m} \right)  \leq C_{T,m}\left(\varepsilon^{m\left[ \left( \frac{v}{\alpha_{2}}\right) \wedge \left( 1-\frac{1\vee(\alpha_{1}-v)}{\alpha_{2}}\right) \right] }+\varepsilon^{m(1-\frac{1-(1\wedge v)}{\alpha_{2}})}\right).
\end{equation}
%this estimate holds for both general and periodic cases.
\end{theorem}
\begin{proof}
From Theorem \ref{T51-1} we know that there exist  $ u(\cdot,\cdot,\cdot)\in C_{b}^{\frac{v}{\alpha_{1}},v,2+\gamma}(\mb{R}^{+}\times\mathbb{R}^{d_{1}}\times\mathbb{R}^{d_{2}})\cap C_b^{2}(\mathbb{R}^{d_{2}})$ s.t.
\begin{equation}\label{5.45}
	\mathcal{L}_{2}(x,y)u(t,x,y)+b(t,x,y)-\bar{b}(t,x)=0.
\end{equation}

Set $u_{n}$ be the mollifyer of $u$, which is solution of \eqref{5.45}, by It\^{o}'s formula we deduce that
\begin{align*}
	u_{n}(t,X_{t}^{\varepsilon},Y_{t}^{\varepsilon})
	&=u_{n}(x,y)+\int_{0}^{t}\partial_{s}u_{n}(s,X_{s}^{\varepsilon},Y_{s}^{\varepsilon})ds+\int_{0}^{t}\mathcal{L}_{1}(s,x,y)u_{n}(s,X_{s}^{\varepsilon},Y_{s}^{\varepsilon})ds\\
	&+\frac{1}{\varepsilon}\int_{0}^{t}\mathcal{L}_{2}(x,y)u_{n}(s,X_{s}^{\varepsilon},Y_{s}^{\varepsilon})ds+M_{n,t}^{1,\varepsilon}+M_{n,t}^{2,\varepsilon},
\end{align*}
here $M_{n,t}^{1,\varepsilon}$, $M_{n,t}^{2,\varepsilon}$ are two $ \mathcal{F}_{t}$ martingales defined as
\begin{align*} M_{n,t}^{1,\varepsilon}&=\int_{0}^{t}\int_{\mathbb{R}^{d_{1}}}\big(u_{n}(s-,X_{s-}^{\varepsilon}+\de_1\cdot  x,Y_{s-}^{\varepsilon})-u_{n}(s-,X_{s-}^{\varepsilon},Y_{s-}^{\varepsilon})\big)\tilde{N}^{1}(ds,dx),\\ M_{n,t}^{2,\varepsilon}&=\int_{0}^{t}\int_{\mathbb{R}^{d_{2}}}\big(u_{n}(s-,X_{s-}^{\varepsilon},Y_{s-}^{\varepsilon}+\varepsilon^{-\frac{1}{\alpha_{2}}}\de_2\cdot  y)-u_{n}(s-,X_{s-}^{\varepsilon},Y_{s-}^{\varepsilon})\big)\tilde{N}^{2}(ds,dy).
\end{align*}
where $\tilde{N}^{1}$, $\tilde{N}^{2}$ are compensated Poisson measures.

Above calculations lead us to 
\begin{align*}
\int_{0}^{t}\mathcal{L}_{2}(x,y)u_{n}(s,X_{s}^{\varepsilon},Y_{s}^{\varepsilon})ds	
&=-\varepsilon\bigg[ u_{n}(x,y)-u_{n}(s,X_{t}^{\varepsilon},Y_{t}^{\varepsilon})+\int_{0}^{t}\partial_{s}u_{n}(s,X_{s}^{\varepsilon},Y_{s}^{\varepsilon})ds\\
&+\int_{0}^{t}\mathcal{L}_{1}(s,x,y)u_{n}(s,X_{s}^{\varepsilon},Y_{s}^{\varepsilon})ds +M_{n,t}^{1,\varepsilon}+M_{n,t}^{2,\varepsilon}\bigg],
\end{align*}
in addition from the non-local Poisson equation \eqref{5.45},
\begin{align}
	&\mathbb{E}\left( \sup_{t\in [0,T]}\left| \int_{0}^{t} b(s,X_{s}^{\varepsilon},Y_{s}^{\varepsilon})-\bar{b}(s,X_{s}^{\varepsilon})ds\right| ^{m}\right)\leq  \mathbb{E}\left(  \int_{0}^{T}\left|\mathcal{L}_{2}(x,y)u_{n}(s,X_{s}^{\varepsilon},Y_{s}^{\varepsilon})-\mathcal{L}_{2}(x,y)u(s,X_{s}^{\varepsilon},Y_{s}^{\varepsilon})\right|^{m}ds \right)\nonumber\\
	&+ C_{T,m}\cdot \varepsilon^{m}\left[ \mathbb{E}\left( \sup_{t\in [0,T]}|u_{n}(x,y)-u_{n}(t,X_{t}^{\varepsilon},Y_{t}^{\varepsilon})|^{m}\right) +\mathbb{E}\left( \int_{0}^{T}|\mathcal{L}_{1}(s,x,y)u_{n}(s,X_{s}^{\varepsilon},Y_{s}^{\varepsilon})|^{m}ds\right)\nonumber\right.	\\
	&\left.+\mathbb{E}\left( \sup_{t\in [0,T]}|M_{n,t}^{1,\varepsilon}|^{m}\right)+\mathbb{E}\left( \sup_{t\in [0,T]}|M_{n,t}^{2,\varepsilon}|^{m}\right) + \mathbb{E}\left(  \int_{0}^{T}\left|\partial_{s}u_{n}(s,X_{s}^{\varepsilon},Y_{s}^{\varepsilon})\right|^{m}ds
	\right)\right] \nonumber\\
	&=I_{0}+C_{T,m}\cdot \varepsilon^{m}\left( I_{1}+I_{2}+I_{3}+I_{4}+I_{5}\right),\label{5.48}
\end{align}
we will estiamte the above terms respectively. 
	
For $I_{0}$, since $2+\gamma> \delta$, by \eqref{3.6}, \eqref{5.57} in Lemma \ref{L51},  similar to \cite[Lemma 4.2]{RX} and \cite[Theorem 5.5]{KY},
\begin{align}
	I_{0}&=\mathbb{E}\left(  \int_{0}^{T}\left|\mathcal{L}_{2}(x,y)u_{n}(s,X_{s}^{\varepsilon},Y_{s}^{\varepsilon})-\mathcal{L}_{2}(x,y)u(s,X_{s}^{\varepsilon},Y_{s}^{\varepsilon})\right|^{m}ds \right)\nonumber\\
	&\leq C_{T,m}\cdot n^{-mv}\mathbb{E}\int_{0}^{T}(\left|1+|Y_{s}^{\varepsilon}|^{m}\right|)ds\leq C_{T,m}(1+|y|^{m})n^{-mv}.\label{5.49-1}
\end{align}

For $I_1$, by definition of $u_{n}$, \eqref{5.3-1} in Theorem \ref{T51-1}, and Lemma \ref{L43}
\begin{align}
		I_{1}&=\mathbb{E}\left( \sup_{t\in [0,T]}|u_{n}(x,y)-u_{n}(t,X_{t}^{\varepsilon},Y_{t}^{\varepsilon})|^{m}\right)\leq \mathbb{E}\left( \sup_{t\in [0,T]}|u(x,y)-u(t,X_{t}^{\varepsilon},Y_{t}^{\varepsilon})|^{m}\right)\nonumber \\
		&\leq C_{T,m}(1+|y|^{m})+\mathbb{E}\left( \sup_{t\in [0,T]}|Y_{t}^{\varepsilon}|^{m}\right)
		\leq C_{T,m}\cdot\varepsilon^{-\frac{m}{\alpha_{2}}}(1+|y|^{m}).\label{5.49}
\end{align}
\iffalse
while in spatial periodic case, by definition of $u_{n}$ and \eqref{5.3-2} in Theorem \ref{T51-1} directly, we have 
\begin{align}
	I_{1}&=\mathbb{E}\left( \sup_{t\in [0,T]}|u_{n}(x,y)-u_{n}(t,X_{t}^{\varepsilon},Y_{t}^{\varepsilon})|^{p}\right)\leq \mathbb{E}\left( \sup_{t\in [0,T]}|u(x,y)-u(t,X_{t}^{\varepsilon},Y_{t}^{\varepsilon})|^{p}\right)\leq C_{T,p}.\label{5.49-2}
\end{align}
\fi

For $I_{2}$, since we have Assumption \ref{gro-bou}, i.e., $|b(t,x,y)| \leq C_{4}(1+K_{t})$, by \eqref{kt}, \eqref{5.58-4} and \eqref{5.58-3} in Lemma \ref{L51}, moment estimates \eqref{3.6},
\begin{align}
	I_{2}&=\mathbb{E}\left(\int_{0}^{T}|\mathcal{L}_{1}(s,x,y)u_{n}(s,X_{s}^{\varepsilon},Y_{s}^{\varepsilon})|^{m}ds\right)\leq C_{T,m}  \mathbb{E}\left( \int_{0}^{T}|(b(s,X_{s}^{\varepsilon},Y_{s}^{\varepsilon}),\nabla_{x}u_{n}(s,X_{s}^{\varepsilon},Y_{s}^{\varepsilon}))|^{m}ds\right)\nonumber
		\\
	&+ C_{T,m} \mathbb{E}\left(\int_{0}^{T}|-(-\Delta_{x})^{\frac{\alpha_{1}}{2}}u_{n}(s,X_{s}^{\varepsilon},Y_{s}^{\varepsilon})|^{m}ds\right)\leq C_{T,m}n^{m(\alpha_{1}-v)}(1+|y|^{m}).\label{5.50}
\end{align}
	\iffalse
	for  $I_{3}$, from growth condition, \eqref{3.4}, \eqref{3.6}, \eqref{5.58-3},
	\begin{equation}\label{5.51}
		\begin{split}
			I_{3}&=\mathbb{E}\left(\frac{1}{\gamma_{\varepsilon}^{p}}\int_{0}^{T}|H(s,X_{s}^{\varepsilon},Y_{s}^{\varepsilon})\nabla_{x}u_{n}(s,X_{s}^{\varepsilon},Y_{s}^{\varepsilon})|^{p}ds\right)\leq  \frac{C_{T,p}}{\gamma_{\varepsilon}^{p}}n^{1-(1\wedge v)}  \mathbb{E}\left(\int_{0}^{T}\left( 1+|X_{s}^{\varepsilon}|+|Y_{s}^{\varepsilon}|\right)^{p}ds\right)\\
			&\leq \frac{C_{T,p}}{\gamma_{\varepsilon}^{p}}n^{1-(1\wedge v)} (1+|x|^{p}+|y|^{p}),
		\end{split}
	\end{equation}
	and for $I_{4}$,
	\begin{equation}\label{5.52}
		\begin{split}
			I_{4}&=\mathbb{E}\left(\frac{1}{\beta_{\varepsilon}^{p}}\int_{0}^{T}|\mathcal{L}_{4}(x,y)u_{n}(s,X_{s}^{\varepsilon},Y_{s}^{\varepsilon})|^{p}ds\right)\leq  \frac{C_{T,p}}{\beta_{\varepsilon}^{p}} (1+|x|^{p}+|y|^{p}).
		\end{split}
	\end{equation}
	\fi
	
We can deduce from Burkholder-Davies-Gundy's inequality, Assumption \ref{gro-bou}  of $\de_1$, \eqref{3.6}, \eqref{5.58-3}, 	\begin{align}
	&\mathbb{E}\left( \sup_{t\in [0,T]}|M_{n,t}^{1,\varepsilon}|^{m}\right)\nonumber\leq C_{T,m}
	\mathbb{E}\left(\sup_{t\in [0,T]}\left| \int_{0}^{t}\left(\int_{|z|\leq 1}u_{n}(s,X_{s}^{\varepsilon}+\de_1\cdot z,Y_{s}^{\varepsilon})-u_{n}(s,X_{s}^{\varepsilon},Y_{s}^{\varepsilon})\tilde{N}_{1}(ds,dz) \right)\right| ^{m}\right)\nonumber\\
	&+ C_{T,m}
	\mathbb{E}\left(\sup_{t\in [0,T]}\left| \int_{0}^{t}\left(\int_{|z|> 1}u_{n}(s,X_{s}^{\varepsilon}+\de_1\cdot  z,Y_{s}^{\varepsilon})-u_{n}(s,X_{s}^{\varepsilon},Y_{s}^{\varepsilon})\tilde{N}_{1}(ds,dz) \right)\right| ^{m}\right)\nonumber\\
	&\leq  C_{T,m}  \int_{0}^{T}\mathbb{E}\left[\left(\int_{|z|\leq 1}|\de_1\cdot z\nabla_{x} u_{n}(s,X_{t}^{\varepsilon},Y_{t}^{\varepsilon})|^{2}\nu_{1}(dz) \right)^{\frac{m}{2}}+\int_{|z|> 1}|\de_1\cdot z\nabla_{x} u_{n}(s,X_{t}^{\varepsilon},Y_{t}^{\varepsilon})|^{m}\nu_{1}(dz)\right]ds\nonumber \\
	&\leq C_{T,m}n^{1-(1\wedge v)}  \int_{0}^{T}\mathbb{E}\left[\left(\int_{|z|\leq 1}|\de_1\cdot z|^{2}(1+|Y_{s}^{\varepsilon}|^{2})\nu_{1}(dz) \right)^{\frac{m}{2}}+\int_{|z|> 1}|\de_1\cdot z|^{m}(1+|Y_{s}^{\varepsilon}|^{m})\nu_{1}(dz) \right]ds\nonumber\\
	& \leq C_{T,m}n^{1-(1\wedge v)}(1+|y|^{m}),\label{5.53}
\end{align}
then, by $ \nabla_{y}u_{n}=(\nabla_{y}u)\ast\rho_{2}^{n}\ast\rho_{1}^{n}$ again, Assumption \ref{uni-ell}-(1)  of $\de_2$,  \eqref{5.4-1} in Theorem \ref{T51-1},
\begin{align}
&\mathbb{E}\left( \sup_{t\in [0,T]}|M_{n,t}^{2,\varepsilon}|^{m}\right)\leq C_{T,m}
\mathbb{E}\left(\sup_{t\in [0,T]}\left| \int_{0}^{t}\left(\int_{|z|\leq 1}u_{n}(s,X_{s}^{\varepsilon},Y_{s}^{\varepsilon}+\varepsilon^{-\frac{1}{\alpha_{2}}}\de_2\cdot z)-u_{n}(s,X_{s}^{\varepsilon},Y_{s}^{\varepsilon})\tilde{N}_{2}(ds,dz) \right)\right| ^{m}\right) \nonumber\\	&+ C_{T,m}\mathbb{E}\left(\sup_{t\in [0,T]}\left| \int_{0}^{t}\left(\int_{|z|> 1}u_{n}(s,X_{s}^{\varepsilon},Y_{s}^{\varepsilon}+\varepsilon^{-\frac{1}{\alpha_{2}}}\de_2\cdot z)-u_{n}(s,X_{s}^{\varepsilon},Y_{s}^{\varepsilon})\tilde{N}_{2}(ds,dz) \right)\right| ^{m}\right) \nonumber\\
&\leq  C_{T,m}\cdot \varepsilon^{-\frac{m}{\alpha_{2}}} \int_{0}^{T}\mathbb{E}\left[\left(\int_{|z|\leq 1}| z\nabla_{y} u_{n}(s,X_{t}^{\varepsilon},Y_{t}^{\varepsilon})|^{2}\nu_{2}(dz) \right)^{\frac{m}{2}}+\int_{|z|> 1}| z\nabla_{y} u_{n}(s,X_{t}^{\varepsilon},Y_{t}^{\varepsilon})|^{m}\nu_{2}(dz) \right]ds \nonumber\\
&\leq C_{T,m}\cdot\varepsilon^{-\frac{m}{\alpha_{2}}}  \int_{0}^{T}\left[\left(\int_{|z|\leq 1}| z|^{2}\nu_{2}(dz) \right)^{\frac{m}{2}}+\int_{|z|> 1}|z|^{m}\nu_{2}(dz) \right]ds \leq C_{T,m}\cdot\varepsilon^{-\frac{m}{\alpha_{2}}} ,\label{5.54}
\end{align}
for $I_{5}$, by \eqref{5.58-1}  in Lemma \ref{L51},
\begin{equation}\label{5.55-1}
	\mathbb{E}\left(  \int_{0}^{T}\left|\partial_{s}u_{n}(s,X_{s}^{\varepsilon},Y_{s}^{\varepsilon})\right|^{m}ds
	\right)  \leq C_{T,m}n^{m(\alpha_{1}-v)}(1+|y|^{m}),
	\end{equation}
combining \eqref{5.49-1}-\eqref{5.55-1} together, take $n=\varepsilon^{-\frac{1}{\alpha_{2}}}$,
\begin{align*}
\mathbb{E}\left( \sup_{t\in [0,T]}\left| \int_{0}^{t}\left( b(s,X_{s}^{\varepsilon},Y_{s}^{\varepsilon})-\bar{b}(s,X_{s}^{\varepsilon})\right)ds\right|^{m} \right)
&\leq C_{T,m,x,y}\left( \varepsilon^{m(1-\frac{1}{\alpha_{2}})}+\varepsilon^{m(1-\frac{\alpha_{1}-v}{\alpha_{2}})}+\varepsilon^{\frac{mv}{\alpha_{2}}}+\varepsilon^{m(1-\frac{1-(1\wedge v)}{\alpha_{2}})} \right)\\ 
&\leq  C_{T,m,x,y}\left(\eta_{\varepsilon}^{\frac{mv}{\alpha_{2}}}+\eta_{\varepsilon}^{m(1-\frac{1\vee(\alpha_{1}-v)}{\alpha_{2}})}+\varepsilon^{m(1-\frac{1-(1\wedge v)}{\alpha_{2}})}\right)\\
&\leq  C_{T,m,x,y}\left(\varepsilon^{m\left[ \left( \frac{v}{\alpha_{2}}\right) \wedge \left( 1-\frac{1\vee(\alpha_{1}-v)}{\alpha_{2}}\right) \right] }+\varepsilon^{m(1-\frac{1-(1\wedge v)}{\alpha_{2}})}\right),
\end{align*}
proof is complete.
\end{proof}

\begin{remark}\label{d1}
We need $\delta_{1}(t,x,y)=\delta_{1}(t)$ for strong convergence estimates in the following sense. 
	
As is shown above, we construct the corrector equation to eliminate the difference of drifts $b-\bar{b}$, provided that $\delta_{1}(t,x,y)= \delta_{1}(t,x)$, there will have
\begin{align*}
	X_{t}^{\varepsilon}-\bar{X_{t}}=\int_{0}^{t}\left( b(s,X_{s}^{\varepsilon},Y_{s}^{\varepsilon})-\bar{b}(s,\bar{X}_{s})\right)ds+\int_{0}^{t}(\delta_{1}(s,X_{s}^{\varepsilon})-\delta_{1}(s,\bar{X}_{s}))dL_{s}^{1},
\end{align*}
the term $\int_{0}^{t}(\delta_{1}(s,X_{s}^{\varepsilon})-\delta_{1}(s,\bar{X}_{s}))dL_{s}^{1}$ cannot be dealt with Lipschitz condition, comparison theorem, Grownall’s inequality, or in any other methods,  consequently, the corrector equation cannot be constructed.
\end{remark}

\section{Weak convergence estimates for \ref{1.1}}\label{wea-est}

Next we consider the following Kolmogorov equation
\begin{equation}\label{6.10}
	\left\{
	\begin{aligned}
		&\partial_{t}u(t,x)=	-(-\Delta_{x})^{\frac{\alpha_{1}}{2}}u(t,x)+( \bar{b}(t,x),\nabla_{x}u(t,x)) ,\ t\in[0,T],\\
		&u(0,x)=\phi(x),
	\end{aligned}
	\right.
\end{equation}
here we let  $\phi(x)\in C^{2+\gamma}_{b}(\mathbb{R}^{d_{1}})$, and the averaged coefficients are defined as follows, 
\begin{align*}
	\bar{b}(t,x)=\int_{\mathbb{R}^{d_{2}}}b(t,x,y)\rho^{x}(dy), \quad \bar{\delta}_{1}(t,x)=\int_{\mathbb{R}^{d_{2}}}\delta_{1}(t,x,y)\rho^{x}(dy),
\end{align*}
$\mathcal{\bar{L}}$ can be regarded as the infinitesimal generator of transition semigroup associated with the averaged process $\bar{X_{t}}$,  which takes the form as $$d\bar{X}_{t}=\bar{b}(t,\bar{X}_{t})dt+\bar{\delta}_{1}(t,\bar{X}_{t})dL_{t}^{1},$$by classical parabolic PDE theory, there exists a unique solution 
\begin{equation}\label{6.11}
	u(t,x)=\mathbb{E}\phi(\bar{X}_{t}(x)),\ t\in[0,T],
\end{equation}
so that $u(t,\cdot)\in C_{b}^{2+\gamma}(\mathbb{R}^{d_{1}})$, $\nabla_{x}u(t,\cdot)\in C_{b}^{1+\gamma}(\mathbb{R}^{d_{1}})$, $\nabla_{x}u(\cdot,x)\in C^{1}([0,T])$, and $\exists C_{T}>0$ s.t.,
\begin{equation}\label{6.12}
	\sup_{t\in [0,T]}\Arrowvert u(t,\cdot)\Arrowvert_{C_{b}^{2+\gamma}(\mathbb{R}^{d_{1}})}\leq C_{T},\  
	\sup_{t\in [0,T]}\Arrowvert \nabla_{x}u(t,\cdot)\Arrowvert_{C_{b}^{1+\gamma}(\mathbb{R}^{d_{1}})}\leq C_{T},\
	 \sup_{t\in [0,T]}\Arrowvert \partial_{t}(\nabla_{x}u(\cdot,x))\Arrowvert\leq C_{T}.
\end{equation}

For any fixed $t>0$, let $\hat{u}_{t}(s,x)=u(t-s,x),\ s\in [0,t]$, by It\^{o}'s formula, 
\begin{equation}\label{6.13}
	\begin{split}
		&\hat{u}_{t}(t,X_{t}^{\varepsilon})=\hat{u}_{t}(0,x)+\int_{0}^{t}\partial_{s}\hat{u}_{t}(s,X_{s}^{\varepsilon})ds+\int_{0}^{t}\mathcal{L}_{1}\hat{u}_{t}(s,X_{s}^{\varepsilon})ds+\hat{M}^{1}_{t},
	\end{split}
\end{equation}
where $$\hat{M}^{1}_{t}=\int_{0}^{t}\int_{\mathbb{R}^{d_{1}}}\Big( \hat{u}_{t}(s,X_{s^{-}}^{\varepsilon}+x)-\hat{u}_{t}(s,X_{s^{-}}^{\varepsilon})\Big)\tilde{N}^{1}(ds,dx), $$
observe that $ \mathbb{E}\hat{M}^{1}_{t}=0,$ $\hat{u}_{t}(t,X_{t}^{\varepsilon})=u(0,X_{t}^{\varepsilon})=\phi(X_{t}^{\varepsilon}),$ $\hat{u}_{t}(0,x)=u(t,x)=\mathbb{E}\phi(\bar{X}_{t}(x)),$ and
\begin{align*}
\partial_{s}\hat{u}_{t}(s,X_{s}^{\varepsilon})&=\partial_{s}u(t-s,X_{s}^{\varepsilon})=-\mathcal{\bar{L}}u_{t}(s,X_{s}^{\varepsilon})=(-\Delta_{x})^{\frac{\alpha_{1}}{2}}\hat{u}_{t}(s,X_{s}^{\varepsilon})-\langle  \bar{b}(s,X_{s}^{\varepsilon}),\nabla_{x}\hat{u}_{t}(s,X_{s}^{\varepsilon})\rangle \\
&=P.V.\int_{\mathbb{R}^{d_{1}}}\Big(\hat{u}_{t}(s,X_{s}^{\varepsilon}+\bar{\delta}_{1}(t,X_{s}^{\varepsilon})z)-\hat{u}_{t}(s,X_{s}^{\varepsilon})-\langle \bar{\delta}_{1}(t,X_{s}^{\varepsilon})z,\nabla_{x}\hat{u}_{t}(s,X_{s}^{\varepsilon})\rangle I_{|z|\leq1}\Big)\nu_{1}(dz)\\
&-\langle  \bar{b}(s,X_{s}^{\varepsilon}),\nabla_{x}\hat{u}_{t}(s,X_{s}^{\varepsilon})\rangle ,
\end{align*}
let $I(\delta)=1$ be the indicator function, thus we have by definition
\begin{align*}
&(-\Delta_{x})^{\frac{\alpha_{1}}{2}}\hat{u}_{t}(s,X_{s}^{\varepsilon})I(\delta_{1})-(-\Delta_{x})^{\frac{\alpha_{1}}{2}}\hat{u}_{t}(s,X_{s}^{\varepsilon})I(\bar{\delta}_{1})\\
&=P.V.\int_{\mathbb{R}^{ d_{1}}}\Big(\hat{u}_{t}(s,X_{s}^{\varepsilon}+\delta_{1}\cdot z)-\hat{u}_{t}(s,X_{s}^{\varepsilon}+\bar{\delta}_{1}\cdot z)-\langle \delta_{1}-\bar{\delta}_{1}\cdot z,\nabla_{x}\hat{u}_{t}(s,X_{s}^{\varepsilon})\rangle I_{|z|\leq1} \Big)\nu_{1}(dz)\\
&=(\delta_{1}-\bar{\delta}_{1})P.V.\int_{\mathbb{R}^{ d_{1}}}\left(\int_{0}^{1}\nabla_{x}\hat{u}_{t}(s,X_{s}^{\varepsilon}+h(\delta_{1}-\bar{\delta}_{1})\cdot z)zdh-\langle  z,\nabla_{x}\hat{u}_{t}(s,X_{s}^{\varepsilon})\rangle I_{|z|\leq1} \right)\nu_{1}(dz),
\end{align*}
then we get from \eqref{6.13},
\begin{align}
	&\mathbb{E}\phi(X_{t}^{\varepsilon})-\mathbb{E}\phi(\bar{X}_{t})
	= \mathbb{E}\int_{0}^{t}-\mathcal{\bar{L}}\hat{u}_{t}(s,X_{s}^{\varepsilon})+\mathcal{L}_{1}\hat{u}_{t}(s,X_{s}^{\varepsilon})ds\nonumber\\
	&=\mathbb{E}(\delta_{1}-\bar{\delta}_{1})\int_{0}^{t}P.V.\int_{\mathbb{R}^{ d_{1}}}\left(\int_{0}^{1}\nabla_{x}\hat{u}_{t}(s,X_{s}^{\varepsilon}+h(\delta_{1}-\bar{\delta}_{1})\cdot z)zdh-\langle  z,\nabla_{x}\hat{u}_{t}(s,X_{s}^{\varepsilon})\rangle I_{|z|\leq1} \right)\nu_{1}(dz)ds\nonumber\\
	&+ \mathbb{E}\int_{0}^{t}\langle b(s,X_{s}^{\varepsilon},Y_{s}^{\varepsilon})- \bar{b}(s,X_{s}^{\varepsilon}),\nabla_{x}\hat{u}_{t}(s,X_{s}^{\varepsilon} )\rangle ds=\mathbb{E}\int_{0}^{t}(A_{1}+A_{2})ds,\label{6.14}
\end{align}
$\forall s\in [0,T]$, $x\in \mathbb{R}^{d_{1}} $, define 
\begin{align}\label{6.15}
\check{b}_{t}(s,x,y)&=\langle  b(s,x,y),\nabla_{x}\hat{u}_{t}(s,x)\rangle, %\quad	\check{b}_{t}^{\Lambda}(s,x,y)=\langle  b^{\Lambda}(s,x,y),\nabla_{x}\hat{u}_{t}(s,x)\rangle,
\end{align}
so that we have 
\begin{align*}
\bar{\check{b}}_{t}(s,x)=\int_{\mathbb{R}^{d_{2}}}\check{b}_{t}(s,x,y)\rho^{x}(dy)=\langle \bar{b}_{t}(s,x),\nabla_{x}\hat{u}_{t}(s,x)\rangle =\left\langle \int_{\mathbb{R}^{d_{2}}}b(s,x,y)\rho^{x}(dy), \nabla_{x}\hat{u}_{t}(s,x) \right\rangle ,
\end{align*}
let $b(t,x,y)\in C_{p}^{\frac{v}{\alpha_{1}},1+\gamma,2}$, then $\bar{b}(t,x)\in C_{p}^{\frac{v}{\alpha_{1}},1+\gamma} $, with the boundedness of $b(s,x,y)$, and $\hat{u}_{t}(s,x)\in  C^{1,2+\gamma}_{b}$, we have
$ \check{b}_{t}(s,x,y),\bar{\check{b}}_{t}(s,x)\in C_{p}^{\frac{v}{\alpha_{1}},1+\gamma,2},$ then we can see that 
\begin{align*}
	\int_{\mathbb{R}^{ d_{2}}}(\check{b}_{t}(s,x,y)- \bar{\check{b}}_{t}(s,x))\rho^{x}(dy)
	&=\int_{\mathbb{R}^{ d_{2}}}\langle b(t,x,Y_{s}^{x,y})-\bar{b}(t,x),\nabla_{x}\hat{u}_{t}(s,x)\rangle \rho^{x}(dy)=0,
	%\int_{\mathbb{R}^{ d_{2}}/\Lambda}(\check{b}_{t}(s,x,y)- \bar{\check{b}}_{t}(s,x))\mu^{x}(dy)
	%&=\int_{\mathbb{R}^{ d_{2}}/\Lambda}\langle b(t,x,Y_{s}^{x,y})-\bar{b}(t,x),\nabla_{x}\hat{u}_{t}(s,x)\rangle \mu^{x}(dy)=0,
\end{align*}
thus $\int_{\mathbb{R}^{ d_{2}}}A_{2}\rho^{x}(dy)=0$.  %$\int_{\mathbb{R}^{ d_{2}}/\Lambda}A_{2}\mu^{x}(dy)=0$.
 Obviously, above analysis can be carried over to $A_{1}$, since 
\begin{align*}
	\int_{\mathbb{R}^{ d_{2}}}\left( \delta_{1}(t,x,y)-\bar{\delta}_{1}(t,x)\right) \rho^{x}(dy)=0,   %\quad\text{and}\quad	\int_{\mathbb{R}^{ d_{2}}/\Lambda}\left( \delta_{1}(t,x,y)-\bar{\delta}_{1}(t,x)\right) \mu^{x}(dy)=0,
\end{align*}
we have $\int_{\mathbb{R}^{ d_{2}}}A_{1}\rho^{x}(dy)=0$. %and $\int_{\mathbb{R}^{ d_{2}}/\Lambda}A_{1}\mu^{x}(dy)=0$ as well.

We next construct the nonlocal Poisson equation as   ``corrector equation" by \eqref{6.14},
\begin{equation}\label{6.19}
	\begin{split}
		\mathcal{L}_{2}\Phi(t,x,y)+ A_{1}+A_{2}=0,
	\end{split}
\end{equation}
here $ A_{1}, A_{2}$ are defined in \eqref{6.14}, and
\begin{equation}\label{6.18}
	\mathcal{L}_{2}\Phi(t,x,y)=-(-\Delta_{y})^{\frac{\alpha_{2}}{2}}\Phi(t,x,y)+f(x,y)\nabla_{y}\Phi(t,x,y),
\end{equation}
and \eqref{6.19} is to eliminate the difference between drifts.  We give some regularity estimates of $\Phi(t,x,y)$.

Our method follows from  \cite[Theorem 6.2, Theorem 6.3]{KY}, consider the  following nonlocal Poisson equation, 
\begin{equation}\label{6.19-1}
	\begin{split}
		\mathcal{L}_{2}\Phi(t,x,y)+( g(t,x,y)-\bar{g}(t,x))=0,
	\end{split}
\end{equation}
here we define  $\bar{g}(t,x)=\int_{\mathbb{R}^{d_{2}}}g(t,x,y)\rho^{x}(dy),$ thus
$$\int_{\mathbb{R}^{ d_{2}}}g(t,x,y)-\bar{g}(t,x)\rho^{x}(dy)=0.$$
\begin{remark}
From definition of $\phi(x)\in C^{2+\gamma}_{b}(\mathbb{R}^{d_{1}})$ in \eqref{6.10}, with analysis of \eqref{6.12} and \eqref{6.15}, regularities of $g(t,x,y)$ with respect to $t$ and $x$ are essential to our analysis, so we let $g(\cdot,\cdot,y)\in C_{b}^{1,1+\gamma}(\mathbb{R}^{1+d_{1}})$.
\end{remark}

We have the following theorem.

\begin{theorem}\label{T61-1}
Let $g(\cdot,\cdot,y)\in C_{b}^{1,1+\gamma}(\mathbb{R}^{1+d_{1}})$  satisfy Assumption $\ref{par-dis}$-Assumption $\ref{lip}$.	$\forall x\in \mathbb{R}^{d_{1}}, $  $y\in \mathbb{R}^{d_{2}} $, and $t\in [0,T]$,  $g(\cdot,\cdot,\cdot)\in C_{b}^{1,1+\gamma, 2}$, we define 
\begin{equation}\label{6.2-1}
	\begin{split}
	\Phi(t,x,y)=\int^{\infty}_{0}\big[ \mathbb{E}g(t,x,Y_{s}^{x,y})-\bar{g}(t,x)\big] ds,
	\end{split}
\end{equation}
then $\Phi(t,x,y)$ is a solution of \eqref{6.19-1}. We have %$\rho^x(u)=0$, and 
$\exists C_T>0$ s.t.,
\begin{equation}\label{6.3}
	\sup_{t\in [0,T]}\sup_{x\in \mathbb{R}^{d_{1}}}|\Phi(t,x,y)|\leq C_{T}(1+|y|).
\end{equation}
\end{theorem}
\begin{proof}
The estimates follow from \cite[Proposition 3.3]{SXX} and \cite[Theorem 6.2]{KY}, we can see that $\Phi(t,x,y)$ is a solution of \eqref{6.19-1} can be deduced by It\^{o}'s formula. %and properties of  $\Phi(t,\cdot,y)\in C^{v}_{p}(\mathbb{R}^{d_{1}})$, $\Phi(t,x,\cdot)\in C_{p}^{2+\gamma}(\mathbb{R}^{d_{2}})$ are deduced from regularities of $g(t,x,y)$,  
%$\mu^{x}(\Phi)=0$ can be deduced from $\mu^{x}(g-\bar{g})=0$. 
It is worthy of emphasizing that we do not need to estimate $|\nabla_{y}\Phi(t,x,y)|$  in the weak convergence analysis.

Similar to \eqref{5.3-1} in Theorem \ref{T51-1},  $\forall g(t,x,\cdot)\in C_{b}^{ 2}(\mathbb{R}^{d_{2}})$, by Lemma \ref{L-gen-exp},
\begin{align*}
	\sup_{t\in [0,T]}\sup_{x\in \mathbb{R}^{d_{1}}} |u(t,x,y)|\leq \int_{0}^{\infty}|\mathbb{E}g(t,x,Y_{s}^{x,y})-\bar{g}(t,x)|ds&\leq  C\int_{0}^{\infty}e^{-\frac{\beta s}{2}}(1+|y|)ds\leq  C(1+|y|),
\end{align*}
we obtain \eqref{6.3}, 	proof is complete.
\end{proof}

\begin{theorem}\label{T62}
Suppose that $b(t,x,y), \delta_{1}(t,x,y)\in C_{p}^{\frac{v}{\alpha_{1}},v,2+\gamma}\cap C_{b}^{1,1+\gamma,2}$, $v\in((\alpha_{1}-\alpha_{2})^{+},\alpha_{1}]$,  $\gamma\in (0,1)$ satisfy  Assumption $\ref{par-dis}$-Assumption $\ref{lip}$,  then we have 
\begin{align*}
	\sup_{t\in [0,T]}\mathbb{E}\int_{0}^{t}\left(A_{1}+A_{2}\right)ds\leq C_{T,x,y}\cdot  \left(\varepsilon^{\frac{v}{\alpha_{2}}}+\varepsilon^{1-\frac{\alpha_{1}-v}{\alpha_{2}}}\right).
\end{align*}
here $ A_{1}, A_{2}$ are defined in \eqref{6.14}.
%this estimate holds for both general and periodic cases.
\end{theorem}
\begin{proof}
Notice that $A_{1}+A_{2}$  has  zero mean.
Let $\Phi^{n}$ be the mollifyer of $\Phi$, which is the solution of \eqref{6.19}, applying It\^{o}'s formula, taking expectation and by the martingale property $\mathbb{E}M_{n,t}^{1,\varepsilon}=\mathbb{E}M_{n,t}^{2,\varepsilon}=0,$ we have
\begin{align*}
	\mathbb{E}\Phi^{n}(t,X_{t}^{\varepsilon},Y_{t}^{\varepsilon})	&=\Phi^{n}(0,x,y)+\mathbb{E}\int_{0}^{t}\partial_{s}\Phi^{n}(s,X_{s}^{\varepsilon},Y_{s}^{\varepsilon})ds+\mathbb{E}\int_{0}^{t} \mathcal{L}_{1}\Phi^{n}(s,X_{S}^{\varepsilon},Y_{s}^{\varepsilon})ds\\
	&+\frac{1}{\varepsilon}\left[ \mathbb{E}\int_{0}^{t}\mathcal{L}_{2}\Phi^{n}(s,X_{s}^{\varepsilon},Y_{s}^{\varepsilon})ds\right],
\end{align*}
then we have
\begin{align*}
	-\frac{1}{\varepsilon} \mathbb{E}\int_{0}^{t}\mathcal{L}_{2}\Phi^{n}(s,X_{s}^{\varepsilon},Y_{s}^{\varepsilon})ds&	=\Phi^{n}(0,x,y)-\mathbb{E}\Phi^{n}(s,X_{t}^{\varepsilon},Y_{t}^{\varepsilon})+\mathbb{E}\int_{0}^{t}\partial_{s}\Phi^{n}(s,X_{s}^{\varepsilon},Y_{s}^{\varepsilon})ds\\
	&+\mathbb{E}\int_{0}^{t} \mathcal{L}_{1}\Phi^{n}(s,X_{s}^{\varepsilon},Y_{s}^{\varepsilon})ds,
\end{align*}
from \eqref{6.19}, \eqref{6.19-1}, 
\begin{align*}
	\sup_{t\in [0,T]}\mathbb{E}\int_{0}^{t}\left(A_{1}+A_{2}\right)ds	&\leq\mathbb{E} \int_{0}^{T}\left|\mathcal{L}_{2}\Phi^{n}(s,X_{s}^{\varepsilon},Y_{s}^{\varepsilon})-\mathcal{L}_{2}\Phi(s,X_{s}^{\varepsilon},Y_{s}^{\varepsilon})\right|ds\\
	&+\varepsilon\sup_{t\in [0,T]}\big[ \mathbb{E}|\Phi^{n}(0,x,y)|+\mathbb{E}|\Phi^{n}(t,X_{t}^{\varepsilon},Y_{t}^{\varepsilon})|\big]  +\mathbb{E}\int_{0}^{T}|\partial_{s}\Phi^{n}(s,X_{s}^{\varepsilon},Y_{s}^{\varepsilon})|ds\\
	&+\mathbb{E}\int_{0}^{T} |\mathcal{L}_{1}\Phi^{n}(s,X_{s}^{\varepsilon},Y_{s}^{\varepsilon})|ds=I_{1}+I_{2}+I_{3}+I_{4},
\end{align*}
since $2+\gamma> \delta$, we can use \eqref{3.6}, \eqref{5.57} in Lemma \ref{L51}, for $I_{1}$, 
\begin{align*}
	I_{1}=\mathbb{E} \int_{0}^{T}\left|\mathcal{L}_{2}(x,y)u_{n}(s,X_{s}^{\varepsilon},Y_{s}^{\varepsilon})-\mathcal{L}_{2}(x,y)u(s,X_{s}^{\varepsilon},Y_{s}^{\varepsilon})\right|ds &\leq C_{T}n^{-v}\mathbb{E}\int_{0}^{T}\big(1+|Y_{s}^{\varepsilon}|\big) ds\\
	&\leq  C_{T}n^{-v}(1+|y|).
\end{align*}

For $I_{2}$, by \eqref{6.3} in Theorem \ref{T61-1}, and \eqref{3.6}, we have 
\begin{align*}
	\varepsilon\sup_{t\in [0,T]}\big[  \mathbb{E}|\Phi^{n}(0,x,y)|+\mathbb{E}|\Phi^{n}(t,X_{t}^{\varepsilon},Z_{t}^{\varepsilon})|\big]  &\leq \varepsilon C_{T}\sup_{t\in [0,T]}\big[  \mathbb{E}|\Phi(0,x,y)|+\mathbb{E}|\Phi(t,X_{t}^{\varepsilon},Y_{t}^{\varepsilon})|\big] \leq \varepsilon C_{T}(1+|y|).
\end{align*}

For $ I_3$, by \eqref{5.58-1} in Lemma \ref{L51},
\begin{align*}
\mathbb{E}\int_{0}^{T}\left|\partial_{s}u_{n}(s,X_{s}^{\varepsilon},Y_{s}^{\varepsilon})\right|ds\leq C_{T}n^{(\alpha_{1}-v)}(1+|y|),
\end{align*}
since we have Assumption \ref{gro-bou}, \eqref{5.58-4} and \eqref{5.58-3} in Lemma \ref{L51}, moment estimates in \eqref{3.4} and \eqref{3.6},
\begin{align*}
I_{4}&=\mathbb{E}\int_{0}^{T}|\mathcal{L}_{1}(s,x,y)u_{n}(s,X_{s}^{\varepsilon},Y_{s}^{\varepsilon})|ds\leq C_{T}  \mathbb{E} \int_{0}^{T}|(b(s,X_{s}^{\varepsilon},Y_{s}^{\varepsilon}),\nabla_{x}u_{n}(s,X_{s}^{\varepsilon},Y_{s}^{\varepsilon}))|ds	\\
&+ C_{T} \mathbb{E}\int_{0}^{T}|-(-\Delta_{x})^{\frac{\alpha_{1}}{2}}u_{n}(s,X_{s}^{\varepsilon},Y_{s}^{\varepsilon})|ds\leq C_{T}n^{(\alpha_{1}-v)}(1+|x|+|y|),
\end{align*}
set $n=\varepsilon^{-\frac{1}{\alpha_{2}}}$, we can obtain
\begin{align*}
\sup_{t\in [0,T]}\mathbb{E}\int_{0}^{t}\left(A_{1}+A_{2}\right)ds&\leq C_{T,x,y}\cdot \left(\varepsilon^{\frac{v}{\alpha_{2}}}+\varepsilon^{1-\frac{\alpha_{1}-v}{\alpha_{2}}}\right),
\end{align*}
proof is complete.
\end{proof}

\section{Statements of main results}\label{main}

In this section, we present the proofs of \textbf{Theorem  \ref{SCR}}, \textbf{Theorem  \ref{WCR}}. Our methods are based on the studies in \cite{CEB}, \cite[Section 4]{SXX} and \cite[Section 7]{KY}, which are beneficial for quantitative estimates.

\subsection{Proof of \textbf{Theorem \ref{SCR}} }
\begin{proof}
Here we let $\delta_{1}(t,x,y)=\delta_{1}(t)$,  analogous to \cite[356, Section 4.1]{SXX} and \cite[Section 7.1]{KY}, we have 
\begin{align*}
	d\bar{X}_{t}=\bar{b}(t,\bar{X}_{t})dt+\delta_{1}(t)dL_{t}^{1},
\end{align*}
so that
\begin{align*}
	X_{t}^{\varepsilon}-\bar{X}_{t}&=\int_{0}^{t}\big( b(s,X_{s}^{\varepsilon},Y_{s}^{\varepsilon})-\bar{b}(s,\bar{X}_{s})\big)ds\\
	&=\int_{0}^{t}\left( b(s,X_{s}^{\varepsilon},Y_{s}^{\varepsilon})-\bar{b}(s,X_{s}^{\varepsilon})\right)ds+\int_{0}^{t}\left( \bar b(s,X_{s}^{\varepsilon})-\bar{b}(s,\bar{X}_{s})\right)ds,
\end{align*}
by Lipschitz continuity of $\bar{b}$ we have for $m\in(1,\alpha_{1}\wedge\alpha_{2})$,
\begin{equation}
	\begin{split}
		\mathbb{E}\left( \sup_{t\in [0,T]}|X_{t}^{\varepsilon}-\bar{X}_{t}|^{m}\right)	&\leq\mathbb{E}\left( \sup_{t\in [0,T]}\left| \int_{0}^{t}\left( b(s,X_{s}^{\varepsilon},Y_{s}^{\varepsilon})-\bar{b}(s,X_{s}^{\varepsilon})\right) ds\right| ^{m}\right)\\
		&+  C_{T,m}\mathbb{E}\int_{0}^{t}|X_{s}^{\varepsilon}-\bar{X}_{s}|^{m}ds.
	\end{split}\nonumber
\end{equation}
then  from Grownall's inequality, 
and Theorem \ref{T52}, we know that 
\begin{align*}
\mathbb{E}\left( \sup_{t\in [0,T]}|X_{t}^{\varepsilon}-\bar{X_{t}}|^{m}\right)	
&\leq\mathbb{E}\left( \sup_{t\in [0,T]}\left| \int_{0}^{t}\left( b(s,X_{s}^{\varepsilon},Y_{s}^{\varepsilon})-\bar{b}(s,\bar{X}_{s})\right) ds\right| ^{m}\right)\\
&\leq  C_{T,m}\left(\varepsilon^{m\left[ \left( \frac{v}{\alpha_{2}}\right) \wedge \left( 1-\frac{1\vee(\alpha_{1}-v)}{\alpha_{2}}\right) \right] }+\varepsilon^{m(1-\frac{1-(1\wedge v)}{\alpha_{2}})}\right).
\end{align*}
\end{proof}

%Consider the following equation in \eqref{2.26},\begin{equation}\label{7.9}	\mathcal{L}_{2}(x,y)u(t,x,y)+ H(t,x,y)=0,	\end{equation}the proof is complete.

\begin{remark}\label{c7s}
We observe that when $v\geq1$, the following simplifications hold:
$$\varepsilon^{\left[ \left( \frac{v}{\alpha_{2}}\right) \wedge \left( 1-\frac{1\vee(\alpha_{1}-v)}{\alpha_{2}}\right) \right] }=\varepsilon^{\left[1-\frac{1-(1\wedge v)}{\alpha_{2}}\right]}=\varepsilon^{1-\frac{1}{\alpha_{2}}},$$
obviously the result corresponds to optimal strong convergence order $1-\frac{1}{\alpha_2}$ of \eqref{3} demostrated in \cite{SXX}, we can deduce that $1-\frac{1}{\alpha_{2}}$ is optimal after comparing the structures of \eqref{1.1} and \eqref{3}.
\end{remark}

\subsection{Proof of \textbf{Theorem \ref{WCR}}}
\begin{proof}
Observe that 
\begin{align*}
	d\bar{X}_{t}=\bar{b}(t,\bar{X}_{t})dt+\bar{\delta}_{1}(t,\bar{X}_{t})dL_{t}^{1},
\end{align*}
thus by regularity estimates in   Theorem \ref{T62}, for $\phi(x)\in C^{2+\gamma}_{b}(\mathbb{R}^{d_{1}})$ in \eqref{6.10}, we obtain
\begin{align*}
	\sup_{t\in [0,T]}|\mathbb{E}\phi(X_{t}^{\varepsilon})-\mathbb{E}\phi(\bar{X_{t}})|
	&\leq\sup_{t\in [0,T]}\left|\mathbb{E}  \int_{0}^{t}-\mathcal{\bar{L}}\hat{u}_{t}(s,X_{s}^{\varepsilon})+\mathcal{L}_{1}\hat{u}_{t}(s,X_{s}^{\varepsilon})ds \right|\leq C_{T,x,y}\cdot \left(\varepsilon^{\frac{v}{\alpha_{2}}}+\varepsilon^{1-\frac{\alpha_{1}-v}{\alpha_{2}}}\right).
\end{align*}
\end{proof}

\begin{remark}\label{c7w}
The parameter relationships become apparent when taking $v=\alpha_{1}=\alpha_{2}$, we have
$$\varepsilon^{\frac{v}{\alpha_{2}}}=\varepsilon^{1-\frac{\alpha_{1}-v}{\alpha_{2}}}=\varepsilon,$$
we observe that this aligns with the weak convergence order 1 for system \eqref{3} established in \cite{SXX}.
\end{remark}

\appendix
\section{technical estimates}\label{A}
As an auxiliary ingredient, we derive a geometric identity for the tangent map induced by a nonlinear immersion on the unit sphereand and its Jacobian determinant for the change of variables in the state-dependent Lévy measure. 
Next we propose some details for Jacobian determinant and derivatives in Section \ref{inv}.
\begin{lemma}\label{jac}
Let $z\in\mb{R}^{d_{2}}\setminus \{0\}$, $\hat{z}=\frac{z}{|z|}$, $\hat{\om}=\frac{\de_2 \hat{z}}{|\de_2\hat{z}|}$, $\de_2$ is $d_{2}\times d_{2}$ matrix-valued function satisfying Assumption $\ref{uni-ell}$-(1),  so we have nonlinear immersion $F(\hat{\om})=\hat{z}=\frac{\de^{-1}_2 \hat{\om}}{|\de^{-1}_2 \hat{\om}|}$, and tangent map $dF(\hat{\om}):T_{\hat{\om}}\s^{d_2-1}\rightarrow T_{\hat{z}}\s^{d_2-1}$, there exists a set of orthonormal basis $\{v_{1},v_{2},\dots,v_{d_2-1}\}\subset T_{\hat{\om}}\s^{d_2-1}$, s.t.,
\begin{align*}
dF(\hat{\om})(v_i)
&=\frac{1}{|\de^{-1}_2\hat{\om}|}\left(I-\dfrac{(\de^{-1}_2\hat{\om})(\de^{-1}_2\hat{\om})^T}{|\de^{-1}_2\hat{\om}|^{2}} \right)\de^{-1}_2v_i, \quad i=1,2,\dots,d_2-1
\end{align*}
the Jacobian determinant of $F(\hat{\om})$ is $J_{F}(\hat{\om})=det(dF({\hat{\om}}))=\dfrac{det(\de_2^{-1})}{|\de_2^{-1}\hat{\om}|^{d_{2}}},$ and particularly,
\begin{align*}
	\left(c_l/c_u\right)^{d_2}	\leq |J_F(\hat{\om})|\leq \left(c_u/c_l \right)^{d_2}.
\end{align*}
\end{lemma}
\begin{proof}
Inspired by computations in Remark \ref{d-0} and the fact that $\hat{z}\sim  \text{Uniform} (\s^{d_2-1})$ due to the isotropy of $\al_{2}$-stable process $L_t^2$, we consult some statistical theories on spheres \cite{YC,FKN,KP,CGS} and apply calculation methods based on tangent space and Jacobi field theory \cite{YC,JJ}.
Obviously $\hat{z}=\de_2^{-1}|\de_{2}\hat{z}|\hat{\om}$, $1=|\hat{z}|=|\de_2^{-1}\hat{\om}||\de_{2}\hat{z}|$, then define $F:\hat{\om}\rightarrow \hat{z}$ as $\hat{z}=F(\hat{\om})=\frac{\de^{-1}_2 \hat{\om}}{|\de^{-1}_2 \hat{\om}|}$. Notably, $\hat{z},  \hat{\om}\in\s^{d_2-1}$,  which claims that  %$T$ maps one $\s^{d_2-1}$ sphere to another $\s^{d_2-1}$ sphere, i.e.,
$F:\s^{d_2-1}\rightarrow\s^{d_2-1}$,  in fact, we are dealing with projection between $\s^{d_2-1}$ spherical manifolds rather than just a vector-valued function defined on $\mb{R}^{d_2}$, see more discussions in Remark \ref{d-0}.

Denote by $T_{\hat{\om}}\s^{d_2-1}$ the tangent space at $\hat{\om}$, and $T_{\hat{z}}\s^{d_2-1}=T_{F(\hat{\om})}\s^{d_2-1}$ the tangent space at $\hat{z}$. Then there exists a class of orthonormal basis $\{v_{1},v_{2},\dots,v_{d_2-1}\}$ in $T_{\hat{\om}}\s^{d_2-1}$, and another class of orthonormal basis
$\{u_{1},u_{2},\dots,u_{d_2-1}\}$ in $T_{\hat{z}}\s^{d_2-1}$, both $v_i$ and $u_j$ are $d_2$-dimensional vector, then let $V$ be $d_2\times (d_2-1)$ matrix whose columns are tangent vector $v_i$, $U$ denotes $d_2\times (d_2-1)$ matrix whose columns are tangent vector $u_j$ \cite[6, Section 1.2]{JJ}, hence, define $d_2\times d_2$ matrix $Q=(V,\hat{\om})$, $R=(U,\hat{z})$, since $|\hat{\om}|=|\hat{z}|=1$, $V,U$ are extended to two classes of orthogonal basis of $\mb{R}^{d_2}$, $Q$ and $R$ respectively.

Following calculations of Jacobi fields in \cite[Section 5.2]{JJ}, we start with computing the directional derivative $dF(\hat{\om})$ with respect to tangent vector $v_i$, i.e., $dF(\hat{\om})(v_i)$, 
especially $\forall v_i\in T_{\hat{\om}}\s^{d_2-1}$, $\hat{\om} v_i=0$, the notations are mainly referred to \cite[7, Section 1.2]{JJ},  then
\begin{align*}
	dF(\hat{\om})(v_i)=\frac{d}{ds}F(\hat{\om}+sv_i)\Big|_{s=0},
\end{align*}
 let $c(s)=\de^{-1}_2\hat{\om}+s\de^{-1}_2v_i$, then $c'(0)=\de^{-1}_2v_i$, $c(0)=\de^{-1}_2\hat{\om}$, $I$ is $d_2\times d_2$ unit matrix, 
\begin{align*}
	dF(\hat{\om})(v_i)=\frac{d}{dt}\dfrac{c(s)}{|c(s)|}\Big|_{s=0}&=\dfrac{c'(0)|c(0)|-\frac{c(0)c^T(0)}{|c(0)|}c'(0)}{|c(0)|^{2}}\\
	&=\dfrac{\de^{-1}_2v_i|\de^{-1}_2\hat{\om}|-\frac{(\de^{-1}_2\hat{\om})(\de^{-1}_2\hat{\om})^T}{|\de^{-1}_2\hat{\om}|}\de^{-1}_2v_i}{|\de^{-1}_2\hat{\om}|^{2}}\\
	&=\frac{1}{|\de^{-1}_2\hat{\om}|}\left(I-\dfrac{(\de^{-1}_2\hat{\om})(\de^{-1}_2\hat{\om})^T}{|\de^{-1}_2\hat{\om}|^{2}} \right)\de^{-1}_2v_i,
\end{align*}
here $I$ is $d_2\times d_2$ unit matrix,  obviously $dF({\hat{\om}})$ is linear and injective, then $F$ is an immersion.

Define the projection matrix $P=I-\dfrac{(\de^{-1}_2\hat{\om})(\de^{-1}_2\hat{\om})^T}{|\de^{-1}_2\hat{\om}|^{2}} $ onto the spaces orthogonal to $\de^{-1}_2\hat{\om}$, which means that $P$ is a projection onto $T_{\hat{z}}\s^{d_2-1}$. More precisely, %the projection matrix $P$ of rank $d_2-1$ has $d_2-1$ eigenvalues equal to 1 corresponding to eigenvectors spanning the tangent space $T_{\hat{\omega}}S^{d_2-1}$, and one eigenvalue equals to 0 corresponding to the eigenvector along the direction of $\delta^{-1}_2\hat{\omega}$, then 
we can deduce that $dF(\hat{\om})(v_i)$ is orthogonal to $\de^{-1}_2\hat{\om}$, and $dF(\hat{\om})(v_i)\in T_{\hat{z}}\s^{d_2-1}$ since $\de_{2}^{-1}\hat{\om}$ is parallel with $\hat{z}$, so that $P$ is a projection onto $T_{\hat{z}}\s^{d_2-1}$. Thus we can have the linear and injective derivative, or so called tangent map, Jacobian matrix,  $dF({\hat{\om}}):T_{\hat{\om}}\s^{d_2-1}\rightarrow T_{\hat{z}}\s^{d_2-1}$, $F$ can also be viewed as a nonlinear immersion \cite[10, Section 1.3]{JJ}.

Since $u_i\in T_{\hat{z}}\s^{d_2-1}$ is $d_2\times 1$, so that  $u_i^TP=(Pu_i)^T=u_i^T$, we represent coordinate of $dF(\hat{\om})(v_i)$ under the orthonormal basis $u_{j}$ of $T_{\hat{z}}\s^{d_2-1}$ as,
\begin{align*}
	u_j^TdF(\hat{\om})(v_i)=\frac{1}{|\de_{2}^{-1}\hat{\om}|}u_j^TP\de^{-1}_2v_i=\frac{1}{|\de_{2}^{-1}\hat{\om}|}(Pu_j)^T\de^{-1}_2v_i=\frac{1}{|\de_{2}^{-1}\hat{\om}|}u_j^T\de^{-1}_2v_i,
\end{align*}
then 
\begin{align*}
dF(\hat{\om})=\sum_{i=1}^{d_2-1}dF(\hat{\om})(v_i)=\sum_{i=1}^{d_2-1}\sum_{j=1}^{d_2-1}u_j^TdF(\hat{\om})(v_i),
	\end{align*}
additionally, in our case the Jacobian determinant $|J_F(\hat{\om})|=|det(dF({\hat{\om}}))|$, here $det(dF({\hat{\om}}))$ is the determinant of $dF({\hat{\om}})$. Denote $A=dF({\hat{\om}})$, and $A_{ij}=u_j^TdF(\hat{\om})(v_i)$, next we compute $det(A)$,
\begin{equation}\label{det_A}
det(A)=det\left(\frac{1}{|\de_{2}^{-1}\hat{\om}|}U^T\de^{-1}_2V \right)=\frac{1}{|\de_{2}^{-1}\hat{\om}|^{d_2-1}}det\left(U^T\de^{-1}_2V \right),
\end{equation}
here $U^T$ is the transpose of $U$. Remind that $U^T$ is $(d_2-1)\times d_2$ matrix in $T_{\hat{z}}\s^{d_2-1}$,  $\de_{2}^{-1}$ is $d_2\times d_2$ matrix in $\R^{d_2}$, $V$ is $d_2\times (d_2-1)$ matrix in $T_{\hat{\om}}\s^{d_2-1}$, consequently $U^T\de^{-1}_2V$ is a $(d_2-1)\times (d_2-1)$ matrix, however, structures of $U$ and $V$ prevent us from computing the exact value of $det(A)$ directly,
%which only demostrates the transformation between tangent spaces  $T_{\hat{\om}}\s^{d_2-1}$ and  $T_{\hat{z}}\s^{d_2-1}$, however the Jabocian determinant ought to describe the total deformation in $\R^{d_2}$, the vector $\de_{2}^{-1}\hat{\om}$ related to the deformation of normal vector $\hat{\om}$ can not be ignored,
we must employ $d_2\times d_2$ matrices $Q$ and $R$ to keep consistent with definition of determinant. 
 
Decomposing the matrix $Q^T\de_{2}^{-1}R$ as 
 $$Q^T\de_{2}^{-1}R=
 \begin{pmatrix}
 	U^T\de_{2}^{-1}V & U^T\de_{2}^{-1}\hat{\om} \\
 	\hat{z}^T\de_{2}^{-1}V & \hat{z}^T\de_{2}^{-1}\hat{\om} 
 \end{pmatrix},
 $$
recall that $\de_{2}^{-1}\hat{\om}$ is parallel with $\hat{z}$, while $U$ is orthogonal to $\hat{z}$, we formally have $U^T\de_{2}^{-1}\hat{\om}=0$, and by definition there has $\hat{z}^T\de_{2}^{-1}\hat{\om}=\dfrac{(\de_{2}^{-1}\hat{\om})^T(\de_{2}^{-1}\hat{\om})}{|\de_{2}^{-1}\hat{\om}|}=|\de_{2}^{-1}\hat{\om}|$, consequently
$$Q^T\de_{2}^{-1}R=
 \begin{pmatrix}
 	U^T\de_{2}^{-1}V & 0 \\
 	\hat{z}^T\de_{2}^{-1}V & |\de_{2}^{-1}\hat{\om}|
 \end{pmatrix},
 $$
which leads to $det(Q^T\de_{2}^{-1}R)=det(U^T\de^{-1}_2V )|\de_{2}^{-1}\hat{\om}|$, since $Q$ and $R$ are orthogonal matrices, we have $det(Q^T\de_{2}^{-1}R)=det(Q^T)\cdot det(\de_{2}^{-1})\cdot det(R)=det(\de_{2}^{-1}),$ substitute this into \eqref{det_A} we have
$
	det(A)=\frac{1}{|\de_{2}^{-1}\hat{\om}|^{d_2}}det(\de_{2}^{-1}),
$
we derive 
\begin{align*}
	|J_F(\hat{\om})|=|det(dF(\hat{\om}))|=|det(A)|=\frac{|det(\de_{2}^{-1})|}{|\de_{2}^{-1}\hat{\om}|^{d_2}}.
\end{align*}

Taking singular value decomposition on matrix $\de_2=L\Sigma H$, $L$ and $H$ are $d_2\times d_2$ orthonormal matrices, since $\delta_{2}$ satisfies Assumption \ref{uni-ell}-(1), $\Sigma$ is $d_2\times d_2$ diagonal matrix with positive ‌singular values‌ $\Sigma=\text{diag}(\sigma_1,\dots,\sigma_{d_2})$, $\sigma_i>0$, we can deduce that 
$
c_l\le \sigma_1,\dots,\sigma_{d_2}\le c_u,
$
then for $\delta_{2}^{-1}$ we have the positive ‌singular values‌ $s_1^{-1},\dots, s_d^{-1}$ and
$c_u^{-1}\le s_1^{-1},\dots, s_d^{-1}\le c_l^{-1},$
so $\forall \hat{\om}\in \s^{d_2-1}$,
$c_u^{-1}\le |\de_{2}^{-1}\hat{\om}|\le c_l^{-1},$
then 
\begin{align*}
|J_F(\hat{\om})|=\frac{|det(\de_{2}^{-1})|}{|\de_{2}^{-1}\hat{\om}|^{d_2}}=\frac{\prod_{i=1}^{d_2}\sigma_i^{-1}}{|\de_{2}^{-1}\hat{\om}|^{d_2}}=\frac{(\prod_{i=1}^{d_2} \sigma_i)^{-1}}{|\de_{2}^{-1}\hat{\om}|^{d_2}}, 
\end{align*}
consequently, 
\begin{align*}
	\left(c_l/c_u\right)^{d_2}	\leq |J_F(\hat{\om})|\leq \left(c_u/c_l \right)^{d_2},
\end{align*}
proof is complete.
\end{proof}

\begin{remark}\label{d-0}
Denote $y=\de_{2}^{-1}\hat{\om}$, $F_{i}=\frac{y_i}{|y|}$,  $|y|=\left(\sum_{i=1}^{d_2} y_i^{2}\right)^{\frac{1}{2}} $, since $ \frac{\partial y_i}{\partial \hat{\om}_j}=(\de^{-1}_{2})_{ij}$, then direct computations and reference \cite[199, 10.3.11-10.3.12]{KP}  both lead us to 
\begin{align*}
\dfrac{dF_i}{d\hat{\om}_j}=\dfrac{d}{d\hat{\om}_j}\left(\dfrac{y_i}{|y|} \right)&=\frac{|y|}{|y|^2}\frac{\partial y_i}{\partial \hat{\om}_j}-\frac{y_i}{|y|^2}\frac{\sum_{i=1}^{d_2} y_i \frac{\partial y_i}{\partial \hat{\om}_j}}{|y|}=\frac{1}{|y|}(\de^{-1}_{2})_{ij}-\frac{y_i}{|y|^2}\frac{\sum_{i=1}^{d_2} y_i (\de^{-1}_{2})_{ij}}{|y|},
\end{align*}
thus we have the Jacobian matrix 
\begin{align*}
\dfrac{dF(\hat{\om})}{d\hat{\om}}=\dfrac{d}{d\hat{\om}}\left(\dfrac{y}{|y|} \right)
&=\frac{\de^{-1}_{2}}{|y|}-\frac{yy^{T} }{|y|^3}\de^{-1}_{2}=\frac{1}{|y|}\left(I-\frac{yy^{T}}{|y|^2} \right)  \de^{-1}_{2},
\end{align*}
and the Jacobian determinant
\begin{align*}
 det\left( \dfrac{dF(\hat{\om})}{d\hat{\om}}\right) =det\left(\frac{1}{|y|}\left(I-\frac{yy^{T}}{|y|^2} \right) \de^{-1}_{2}\right) =\frac{1}{|y|^{d_2}}det\left(I-\frac{yy^{T}}{|y|^2} \right)  \cdot det( \de^{-1}_{2} ),
\end{align*}
 obviously $I-\frac{yy^{T}}{|y|^2}$ is a matrix of rank $d_2-1$, $det\big(I-\frac{yy^{T}}{|y|^2} \big)=0$, then $|J_{F}(\hat{\om})|=\left| det\left( \dfrac{dF(\hat{\om})}{d\hat{\om}}\right) \right|=0,$ which means that the $J_{F}(\hat{\om})$ is a matrix of rank $d_2-1$. Furthermore from definitions we observe that $\hat{z}, \hat{\om}\in\s^{d_2-1}$, thus we ought to decompose $\R^{d_2}$ into  tangent spaces and normal vectors, i.e., $T_{\hat{\om}}\s^{d_2-1}$ at point $\hat{\om}$, and $T_{\hat{z}}\s^{d_2-1}$ at $\hat{z}$, where we have nonlinear immersion $\hat{z}=F(\hat{\om}):\s^{d_2-1}\rightarrow \s^{d_2-1}$, then Jacobian determinant can be derived by the tangent map $dF({\hat{\om}}):T_{\hat{\om}}\s^{d_2-1}\rightarrow T_{\hat{z}}\s^{d_2-1}$.
\end{remark}

Form Lemma \ref{jac}, we derive an explicit formula for the tangent map between the tangent spaces of $ \s^{d-1}$ and the corresponding Jacobian determinant,  where the map is induced by a nonlinear immersion.
\begin{theorem}[Tangent map]\label{tan}
	Let $z\in\mb{R}^{d}\setminus\{0\}$,  $\hat{z}=\frac{z}{|z|}\in\mb{S}^{d-1}$, $\hat{\om}=\frac{A\hat{z}}{|A\hat{z}|}$, and $A\in GL(\mathbb{R}^{d})$ be an invertible matrix. Define the nonlinear immersion $F(\hat{\om})=\hat{z}=\frac{A^{-1} \hat{\om}}{|A^{-1} \hat{\om}|}$. Then the tangent map $dF(\hat{\om}):T_{\hat{\om}}\s^{d-1}\rightarrow T_{\hat{z}}\s^{d-1}$ satisfies: there exists an orthonormal basis $\{v_{1},v_{2},\dots,v_{d-1}\}\subset T_{\hat{\om}}\s^{d-1}$ such that 
	\begin{align*}
		dF(\hat{\om})(v_i)
		&=\frac{1}{|A^{-1}\hat{\om}|}\left(I-\dfrac{(A^{-1}\hat{\om})(A^{-1}\hat{\om})^T}{|A^{-1}\hat{\om}|^{2}} \right)A^{-1}v_i, \quad i=1,2,\dots,d-1,
	\end{align*}
	moreover, the Jacobian determinant of $F$ is given by $J_{F}(\hat{\om})=det(dF(\hat{\om}))=\dfrac{det(A^{-1})}{|A^{-1}\hat{\om}|^{d}}.$
\end{theorem}
\begin{proof}
	This theorem is an extended result of Lemma \ref{jac}, here we only need the invertibility of $A$ rather than uniformly elliptic conditions to make $dF(\hat{\om})$ and Jacobian determinant well-defined, see more details in Lemma \ref{jac}.
\end{proof}

\begin{lemma}\label{cp-psi}
We define
$$\psi(r)=
\begin{cases}
	1-e^{-c_1r}, & \quad r\in(0,2L_0];\\
	Ae^{c_2(r-2L_0)}+B(r-2L_0)^2+ (1-e^{-2c_1L_0}-A
	), & \quad r\in [2L_0,\infty),\\
\end{cases}$$
where
$A= \frac{c_1}{c_2}\mathrm{e}^{-2L_0c_1}>0,$ $B= -\frac{(c_1+c_2)c_1}{2}\mathrm{e}^{-2L_0c_1}<0,$
$c_1,c_2>0$. Then $\forall p\geq1$, $\exists 0<c(p)$ s.t. 
\begin{align}\label{cp-psi-2}
	|r|^p\leq c(p)\psi(r),
\end{align}
especially, $\exists C>0$ s.t. $\psi(r)\leq Cr$ when $r\in(0,2L_0]$.
\end{lemma}
\begin{proof}
For $r\in(0,2L_0]$, consider the positive function
$ f(r)=\frac{r^p}{1-e^{-c_1r}},$
we aim to show that $f$ is bounded. Obviously this is a continuous differential function, $f(2L_0)=\frac{2^pL_0^{p}}{1-e^{-2c_1L_0}}>0$, and by L'H\^{o}pital's rule, we have  $\lim_{r\rightarrow0} f(r)=\lim_{r\rightarrow0}\frac{pr^{p-1}}{c_1e^{-c_1r}}=0$, by continuity \eqref{cp-psi-2} is derived. Let $g(r)=\frac{1-e^{-2c_1r}}{Cr}$, then $\lim_{r\rightarrow0} g(r)=\lim_{r\rightarrow0}\frac{c_1e^{-c_1r}}{C}=\frac{c_1}{C}<\infty,$ which implies that $\psi(r)\leq Cr$ on $r\in(0,2L_0]$.
	
When $r\in [2L_0,\infty)$,  consider the function
$$ f(r)=\frac{r^p}{Ae^{c_2(r-2L_0)}+B(r-2L_0)^2+(1-e^{-2c_1L_0}-A)},$$
then $f(2L_0)=\frac{2^pL_0^{p}}{1-e^{-2c_1L_0}}>0$, since from \eqref{r-2l} have $\psi'>0$ on  $r\in [2L_0,\infty)$, so that $\psi>0$, $f$ does not blow up, it is easy to verify that $\lim_{r\rightarrow0} f(r)=\lim_{r\rightarrow\infty}=0$, so $f$ is bounded, the proof is complete.
\end{proof}

\section*{Acknowledgements}
The authors are very grateful to the invaluable guidance from Professor Xin Chen in Shanghai Jiao Tong University as the doctoral supervisor, particularly for  the crucial references and useful communications. Special acknowledgement is extended to  Professor Long-Jie Xie in Jiangsu Normal University, whose poineer works in multiscale stochastic systems and talkings have profoundly inspired the authors.


\begin{thebibliography}{99}
	
	
	\bibitem{DA} D. Applebaum, \textit{Lévy Processes and Stochastic Calculus}, 2nd ed., Cambridge University Press, 2009.
	
	\bibitem{BSWX} J.-H. Bao, X.-B. Sun, J. Wang, Y.-C. Xie, \textit{Quantitative estimates for Lévy driven SDEs with different drifts and applications}, J. Differential Equations, \textbf{398} (2024), 182–217.
	
	\bibitem{BLP} A. Bensoussan, J.-L. Lions, G. Papanicolaou, \textit{Asymptotic Analysis for Periodic Structures}, North-Holland, Amsterdam, 1978.
	
	\bibitem{CEB2} C. E. Bréhier, \textit{Strong and weak orders in averaging for SPDEs}, Stochastic Process. Appl., \textbf{122}(7) (2012), 2553–2593.
	
	\bibitem{CEB} C. E. Bréhier, \textit{Orders of convergence in the averaging principle for SPDEs: The case of a stochastically forced slow component}, Stochastic Process. Appl., \textbf{130} (2020), 3325–3368.
	
	%\bibitem{CF} S. Cerrai, M. Freidlin, \textit{Averaging principle for a class of stochastic reaction–diffusion equations}, Probab. Theory Relat. Fields, \textbf{144} (2009), 137–177.
	
	\bibitem{CSS} Y.-L. Chen, Y.-H. Shi, X.-B. Sun, \textit{Averaging principle for slow–fast stochastic Burgers equation driven by $\alpha$-stable process}, Appl. Math. Lett., \textbf{103} (2020), 106199.
	
	\bibitem{YC} Y. Chikuse, \textit{Statistics on Special Manifolds}, Springer, 2003.
	
	%\bibitem{DSXZ} Z. Dong, X.-B. Sun, H. Xiao, J.-L. Zhai, \textit{Averaging principle for one dimensional stochastic Burgers equation}, J. Differential Equations, \textbf{265}(10) (2018), 4749–4797.
	
	\bibitem{FKN} K.-T. Fang, S. Kotz, K. W. Ng, \textit{Symmetric Multivariate and Related Distributions}, Chapman and Hall, 1990.
	
	\bibitem{BF} B. Franke, \textit{A functional non-central limit theorem for jump-diffusions with periodic coefficients driven by stable Lévy-noise}, J. Theoret. Probab., \textbf{20} (2007), 1087–1100.
	
	%\bibitem{FWL} H. Fu, L. Wan, J. Liu, \textit{Strong convergence in averaging principle for stochastic hyperbolic–parabolic equations with two time-scales}, Stochastic Process. Appl., \textbf{125}(8) (2015), 3255–3279.
	
	%\bibitem{PG} P. Gao, \textit{Averaging principle for stochastic Korteweg-de Vries equation}, J. Differential Equations, \textbf{267} (2019), 6872–6909.
      
	\bibitem{DIR} D. Givon, I. G. Kevrekidis, R. Kupferman, \textit{Strong convergence of projective integration schemes for singularly perturbed stochastic differential systems}, Commun. Math. Sci., \textbf{4}(4) (2006), 707–729.

	\bibitem{GD} H. Gao, J. Duan, \textit{Dynamics of quasi-geostrophic fluid motion with rapidly oscillating Coriolis force}, Nonlinear Anal. Real World Appl., \textbf{4} (2003), 127–138


    
	\bibitem{JJ} J. Jost, \textit{Riemannian Geometry and Geometric Analysis}, 6th ed., Springer, 2011.



\bibitem{YK} Y. Kiffer, \textit{Averaging and climate models}, in: Stochastic Climate Models (Chorin, 1999), Progr. Probab., \textbf{49}, Birkhäuser, Basel, 2001, pp. 171–188.


    
	\bibitem{RZK} R. Z. Khas'minskii, \textit{The averaging principle for stochastic differential equations}, Probl. Peredachi Inf., \textbf{4}(2) (1968), 86–87; Problems Inform. Transmission, \textbf{4}(2) (1968), 68–69.
	
	\bibitem{ANK} A. N. Kochubei, \textit{Parabolic pseudodifferential equation, supersingular integrals and Markov processes}, Izv. Akad. Nauk SSSR Ser. Mat., \textbf{52}(5) (1982), 909–934 (in Russian); Math. USSR-Izv., \textbf{33} (1983), 233–259.
	
	\bibitem{VK} V. Kolokoltsov, \textit{Symmetric stable laws and stable-like jump diffusions}, Proc. London Math. Soc., \textbf{80} (2000), 725–768.
	
	\bibitem{CK} C. Kuehn, \textit{Multiple Time Scale Dynamics}, Applied Mathematical Sciences, \textbf{191}, Springer, Cham, 2015.
	
	
	\bibitem{LW} M.-J. Liang, J. Wang, \textit{Spatial regularity of semigroups generated by Lévy type operators}, Math. Nachr., \textbf{292}(7) (2019), 1551–1566.
	
	\bibitem{DL} D. Liu, \textit{Strong convergence rate of principle of averaging for jump-diffusion processes}, Front. Math. China, \textbf{7} (2012), 305–320.
	
	\bibitem{LRSX} W. Liu, M. Röckner, X.-B. Sun, Y.-C. Xie, \textit{Averaging principle for slow–fast stochastic differential equations with time dependent locally Lipschitz coefficients}, J. Differential Equations, \textbf{268} (2020), 2910–2948.
	
	\bibitem{DW} D.-J. Luo, J. Wang, \textit{Coupling by reflection and Hölder regularity for non-local operators of variable order}, Trans. Amer. Math. Soc., \textbf{371} (2019), 431–459.
	
	%\bibitem{MN} J. R. Magnus, H. Neudecker, \textit{Matrix Differential Calculus with Applications in Statistics and Econometrics}, 3rd ed., John Wiley \& Sons, 2007.

\bibitem{MTV} A. J. Majda, I. Timofeyev, E. Vanden Eijnden, \textit{A mathematical framework for stochastic climate models}, Comm. Pure Appl. Math., \textbf{54} (2001), 891–974.

    
	\bibitem{MBM} M. B. Majka, \textit{Coupling and exponential ergodicity for stochastic differential equations driven by Lévy processes}, Stochastic Process. Appl., \textbf{127}(12) (2017), 4083–4125.
	
	\bibitem{KP} K. V. Mardia, P. E. Jupp, \textit{Directional Statistics}, John Wiley \& Sons, 1999.
	
	\bibitem{EY1} E. Pardoux, Yu. Veretennikov, \textit{On the Poisson equation and diffusion approximation. I}, Ann. Probab., \textbf{29}(3) (2001), 1061–1085.
	
	\bibitem{EY2} E. Pardoux, A. Yu. Veretennikov, \textit{On Poisson equation and diffusion approximation. II}, Ann. Probab., \textbf{31}(3) (2003), 1166–1192.
	
	\bibitem{EY3} E. Pardoux, A. Yu. Veretennikov, \textit{On the Poisson equation and diffusion approximation. III}, Ann. Probab., \textbf{33}(3) (2005), 1111–1133.

    \bibitem{PS} G. A. Pavliotis, A. M. Stuart, \textit{Multiscale Methods: Averaging and Homogenization}, Texts in Applied Mathematics, \textbf{53}, Springer, New York, 2008.
	
	%\bibitem{PXW} B. Pei, Y. Xu, J.-L. Wu, \textit{Two-time-scales hyperbolic-parabolic equations driven by Poisson random measures: existence, uniqueness and averaging principles}, J. Math. Anal. Appl., \textbf{447} (2017), 243–268.
	
	\bibitem{RX} M. Röckner, L.-J. Xie, \textit{Diffusion approximation for fully coupled stochastic differential equations}, Ann. Probab., \textbf{49}(3) (2021), 1205–1236.
	
	\bibitem{RX2} M. Röckner, L.-J. Xie, \textit{Averaging principle and normal deviations for multiscale stochastic systems}, Comm. Math. Phys., \textbf{383} (2021), 1889–1937.
	
	\bibitem{CGS} C. G. Smalle, \textit{The Statistical Theory of Shape}, Springer, New York, 1996.
	
	%\bibitem{DWS} D. W. Stroock, \textit{Diffusion processes associated with Lévy generators}, Z. Wahrsch. Verw. Gebiete, \textbf{32} (1975), 209–244.
	
	\bibitem{SXX} X.-B. Sun, L.-J. Xie, Y.-C. Xie, \textit{Strong and weak convergence rates for slow-fast stochastic differential equations driven by $\alpha$-stable process}, Bernoulli, \textbf{28}(1) (2022), 343–369.
	
	\bibitem{AYV} A. Yu. Veretennikov, \textit{On the averaging principle for systems of stochastic differential equations}, Math. USSR-Sb., \textbf{69}(1) (1991), 271.
	
	\bibitem{JW} J. Wang, \textit{Exponential ergodicity and strong ergodicity for SDEs driven by symmetric $\alpha$-stable processes}, Appl. Math. Lett., \textbf{26} (2013), 654–658.
	
	\bibitem{JW2} J. Wang, \textit{$L^p$-Wasserstein distance for stochastic differential equations driven by Lévy processes}, Bernoulli, \textbf{22}(3) (2016), 1598–1616.
	
	%\bibitem{WC} P. Wang, G. Chen, \textit{Effective approximation of stochastic sine-Gordon equation with a fast oscillation}, J. Math. Phys., \textbf{62} (2021), 032702.
	
	\bibitem{KY} K. Yin, \textit{Strong and weak convergence rates for fully coupled multiscale stochastic differential equations driven by $\alpha$-stable processes}, preprint, arXiv:2505.10229, version 3.
	
	\bibitem{YXJ} H.-G. Yue, Y. Xu, R.-F. Wang, Z. Jiao, \textit{Averaging principle of stochastic Burgers equation driven by Lévy processes}, J. Math. Phys., \textbf{64}(10) (2023), 103506.
	
	\bibitem{ZHWWD} Y.-J. Zhang, Q. Huang, X. Wang, Z.-B. Wang, J.-Q. Duan, \textit{Weak averaging principle for multiscale stochastic dynamical systems driven by stable processes}, J. Differential Equations, \textbf{379} (2024), 721-761. 
	




\end{thebibliography}
\end{document}